\documentclass{article}

\usepackage[utf8]{inputenc}
\usepackage[normalem]{ulem} 	
\usepackage{xparse}         	
\usepackage{comment}

\usepackage{mathtools,amssymb,amsthm,empheq}
\usepackage{bm} 	   	
\usepackage{relsize}   	
\usepackage{mathrsfs}  	
\usepackage{cancel}   	
\usepackage{esint}    	
\usepackage{xcolor} 	
\usepackage{cases}

\usepackage{pgf,tikz}
\usetikzlibrary{cd,arrows,babel,shapes,positioning}

\usepackage{enumerate}

\usepackage{hyperref}
\hypersetup{
	colorlinks,
	citecolor=blue,
	filecolor=blue,
	linkcolor=blue,
	urlcolor=blue
}

\DeclareMathOperator{\Dist}{dist}
\DeclareMathOperator{\Div}{div}

\newcommand{\AP}{A_{\Vert}}
\newcommand{\Z}{\mathbb{Z}}

\newcommand{\supp}{\mathrm{supp}}

\newcommand{\EPS}{\varepsilon}

\newcommand{\ColorWord}[2]{\color{#1} #2 \color{black} }

\numberwithin{equation}{section}

\theoremstyle{plain}

\newtheorem{thm}[equation]{Theorem}
\newcommand{\refthm}[1]{\emph{\ColorWord{blue}{Theorem} \ref{#1}}}

\newtheorem{lemma}[equation]{Lemma}
\newcommand{\reflemma}[1]{\emph{\ColorWord{blue}{Lemma} \ref{#1}}}

\newtheorem{prop}[equation]{Proposition}
\newcommand{\refprop}[1]{\emph{\ColorWord{blue}{Proposition} \ref{#1}}}

\newtheorem{cor}[equation]{Corollary}
\newcommand{\refcor}[1]{\emph{\ColorWord{blue}{Corollary} \ref{#1}}}

\theoremstyle{definition}

\newtheorem{defin}[equation]{Definition}
\newcommand{\refdef}[1]{\emph{Definition \ref{#1}}}

\theoremstyle{remark}

\newtheoremstyle{named}{}{}{\itshape}{}{\bfseries}{}{.5em}{#1 #3}
\theoremstyle{named}
\allowdisplaybreaks

\title{Solvability of the Dirichlet problem for a new class of elliptic operators}
\author{Martin Ulmer}
\date{\today}

\begin{document}

\maketitle

\begin{abstract}
\noindent
We study an elliptic operator \(L:=\mathrm{div}(A\nabla \cdot)\) on the upper half space. It is known that if the matrix \(A\) is independent in the transversal \(t\)-direction, then we have \(\omega\in A_\infty(\sigma)\). In the present paper we improve on the \(t\)-independence condition by introducing a mixed \(L^1-L^\infty\) condition that only depends on \(\partial_t A\), the derivative of \(A\) in transversal direction. This condition is different from other conditions in the literature.

In the case of the upper half plane, we obtain the improvement that an \(L^1\)-Carleson condition on \(|\partial_tA|\) implies \(\omega\in A_\infty(\sigma)\). In particular, this condition is similar to an \(L^1\)-version of the DKP condition with derivative in only the transversal direction. 
\end{abstract}

\tableofcontents

\section{Introduction}

In this work let \(\Omega:=\mathbb{R}^{n+1}_+:=\mathbb{R}^n\times (0,\infty), n\geq 1\) be the upper half space and \(L:=\Div(A\nabla\cdot)\) an uniformly elliptic operator with bounded measurable coefficients, i.e. \(A(x,t)\) is a real not necessarily symmetric \(n+1\) by \(n+1\) matrix and there exists \(\lambda_0>0\) such that
\begin{align}
    \lambda_0 |\xi|^2\leq \xi^T A(x,t) \xi \leq \lambda_0^{-1}|\xi|^2 \qquad \textrm{ for all }\xi\in \mathbb{R}^{n+1},\label{eq:DefinitionOfUniformElliptic}
\end{align}
and a.e. \((x,t)=(x_1,...,x_n,t)\in \mathbb{R}^{n+1}_+\). We are interested in the solvability of the \(L^p\) Dirichlet boundary value problem 
\[\begin{cases} Lu=\Div(A\nabla u)=0 &\textrm{in }\Omega, \\ u=f &\textrm{on }\partial\Omega,\end{cases}\]
where \(f\in L^p(\partial\Omega)\) (see \refdef{def:L^pDirichletProblem}). It is well known that solvability for some \(1< p<\infty\) is equivalent to the elliptic (doubling) measure \(\omega\) corresponding to \(L\) lying in the Muckenhoupt space \(A_\infty(\sigma)\). This Muckenhoupt space yields a version of scale-invariant absolute continuity between the elliptic measure \(\omega\) and the surface measure \(\sigma\). Due to the counterexamples in \cite{caffarelli_completely_1981} and \cite{modica_construction_1980} we do not always have absolute continuity between \(\omega\) and \(\sigma\), even if we assume that the coefficients of \(A\) are continuous. In fact, these examples show that some regularity on the coefficients in transversal direction to the boundary of the domain is necessary to obtain absolute continuity. This observation gave rise to the study of so called \textit{t-independent} elliptic operators \(L\), i.e. operators where \(A(x,t)=A(x)\) is independent of the transversal direction. The first breakthrough in this direction came in \cite{jerison_dirichlet_1981}, where Jerison and Kenig showed via a ``Rellich" identity that if \(A\) is symmetric and \(t\)-independent with bounded and measurable coefficients on the unit ball, we have \(\omega\in B_2(\sigma)\subset A_\infty(\sigma)\). Although this ``Rellich" identity does not 
hold for operators with nonsymmetric matrices \(A\), the work \cite{kenig_new_2000} established (20 years later) \(\omega\in A_\infty(\sigma)\) for nonsymmetric operators in the upper half plane \(\mathbb{R}^2_+\). Additionally, they gave a counterexample showing that \(\omega\in B_2(\sigma)\) cannot be expected for nonsymmetric matrices and that the space \(A_\infty(\sigma)\) is sharp. Since we are not going to use the Muckenhoupt spaces explicitely in the following, we are not going to recall their defintions and their properties. However, they are an important tool in harmonic analysis and go back to \cite{muckenhoupt_weighted_1972} and \cite{coifman_weighted_1974}, and more results about them can be also found in \cite{grafakos_modern_2009}.

\medskip

For dimension \(n\) however, it took until the Kato conjecture was resolved in \cite{Auscher_Kato} after being open for 50 years, before the authors of \cite{hofmann_square_2015} could extend this result to matrices that are not necessarily symmetric and have merely bounded and measurable coefficients. Later, this work was streamlined in \cite{hofmann_dirichlet_2022}, where the authors also extended the result to the case of matrices whose antisymmetric part might be unbounded, but has a uniformly bounded BMO norm instead. The \(t\)-independence condition has also been adapted to the parabolic setting as a sufficient condition for solvability of the \(L^p\) Dirichlet problem (see \cite{auscher_dirichlet_2018}) or for the elliptic regularity boundary value problem \cite{hofmann_regularity_2015}. 

\medskip

Shortly after the breakthrough \cite{jerison_dirichlet_1981}, Jerison, Kenig and Fabes published \cite{fabes_necessary_1984}, where they showed that \(t\)-independence can be relaxed if we use continuous coefficients. More precisely, they show that if a symmetric \(A\) has continuous coefficients, \(\Omega\) is a bounded \(C^1\)-domain, and the modulus of continuity 
\[\eta(s)=\sup_{P\in \partial\Omega, 0<r<s}|A_{ij}(P-rV(P))-A_{ij}(P)|\]
with outer normal vector field \(V\) satisfies the Dini-type condition
\begin{align}\int_0\frac{\eta(s)^2}{s}ds<\infty,\label{DiniTypeCond}\end{align}
then \(\omega\in B_2(\sigma)\subset A_\infty(\sigma)\). Together with the counterexample in \cite{jerison_dirichlet_1981}, where completely singular measures \(\omega\) with respect to the surface measures are constructed for a given \(\eta\) with \(\int_0 \frac{\eta(s)^2}{s}ds=+\infty\), the Dini type condition (\ref{DiniTypeCond}) turns out to be sufficient for \(\omega\in A_\infty(d\sigma)\) and necessary in some sense, if \(A\) is symmetric with continuous bounded coefficients. The sense of necessity means that for every given \(\eta\) that fails to satisfy \eqref{DiniTypeCond} there exists a matrix \(A\) and a corresponding elliptic operator so that its deduced modulus of continuity \(\eta_A\) is less or equal to the given \(\eta\) but where the corresponding elliptic measure does not lie in \(A_\infty(\sigma)\). A little bit later in \cite{dahlberg_absolute_1986}, Dahlberg removed the assumption of continuity by considering perturbations from continuous matrices.

\medskip

In this paper we assume that the matrices have bounded and measurable coefficients and are not necessarily symmetric. We will work on the upper half space \(\Omega:=\mathbb{R}^{n+1}_+\). The main innvotation here is the introduction of a mixed \(L^1-L^\infty\) condition on \(\partial_t A(x,t)\), i.e.
\begin{align}\int_0^\infty\Vert\partial_t A(\cdot,t)\Vert_{L^\infty(\mathbb{R}^n)} ds\leq C<\infty.\label{cond:L1-Linfty}\end{align}
That is, we take the \(L^\infty\)-norm of \(\partial_t A(x,t)\) in the \(x\)-direction first, before imposing some quantified form of decay close to 0 and to infinity in \(L^1\) norm. The main theorem of this work is the following.

\begin{thm}\label{thm:MainTheorem}
Let \(L:=\Div(A\nabla\cdot)\) be an elliptic operator satisfying \eqref{eq:DefinitionOfUniformElliptic} on \(\Omega=\mathbb{R}^{n+1}\). If condition \eqref{cond:L1-Linfty} is satisfied and \(|\partial_t A|t\leq C<\infty\), then \(\omega\in A_\infty(d\sigma)\), i.e. the \(L^p\) Dirichlet boundary value problem for \(L\) is solvable for some \(1<p<\infty\).
\end{thm}

We would like to put our result \refthm{thm:MainTheorem} into the context of similar conditions in the literature. 


Although the Dini-type condition \eqref{DiniTypeCond} has only been considered on bounded domains with regular boundary and only for symmetric operators in \cite{fabes_necessary_1984} and \cite{dahlberg_absolute_1986}, we do have a similar and even stronger result for not necessarily symmetric matrices in the unbounded upper half space. Inspired by the application of perturbation theory in \cite{dahlberg_absolute_1986}, we can introduce the following Carleson measure condition
\begin{align}\sup_{\substack{\Delta\subset\partial\Omega \\ \Delta \textrm{ boundary ball}}}\int_{T(\Delta)} \frac{\sup_{\frac{1}{2}s\leq t\leq \frac{3}{2}s, y\in \Delta(x, s/2)}|A(x,t)-A(x,0)|^2}{s}dxds<\infty\label{DiniTypePerturbation}\end{align}
 on the matrix. Assuming \eqref{DiniTypePerturbation} allows us to consider the elliptic operator \(L_0=\Div(A\nabla \cdot)\) as a perturbation of the \(t\)-independent operator \(L_1=\Div(\tilde{A}\nabla \cdot)\) with \(\tilde{A}(x,t):=A(x,0)\). For the t-independent operator we have \(\omega_{\tilde{L}}\in A_\infty(\sigma)\) (cf. \cite{hofmann_dirichlet_2022}) and perturbation theory yields \(\omega_L\in A_\infty(\sigma)\) (for the history of perturbation theory we refer the reader to \cite{cavero_perturbations_2020} and references therein). It is easy to see that \eqref{DiniTypePerturbation} is a more general condition than what the classical Dini-type condition \eqref{DiniTypeCond} would correspond to on an unbounded domain like the upper half space. 
\medskip

Furthermore, there are also small \(L^\infty\) perturbations of \(t-\)independent operators like introduced in \cite{alfonseca_analyticity_2011}. If the matrix \(A\) is symmetric and \(\Vert A(x,t)-A(x,0)\Vert_{L^\infty(\Omega)}\) is sufficiently small, then the \(L^2\)-Dirichlet problem is solvable for \(L\). This conclusion even works if \(A\) has complex coefficients. But we can easily see that, even if a similar result was true for nonsymmetric matrices \(A\) and small \(L^\infty\) perturbations, this condition does not subsume \eqref{cond:L1-Linfty}.

\medskip

The condition \eqref{cond:L1-Linfty} generalises the \(t\)-independence condition, yet it is different from the Dini type condition \eqref{DiniTypeCond} or from condition \eqref{DiniTypePerturbation}. 
Intuitively speaking, the condition \eqref{cond:L1-Linfty} can cover operators with larger gradient of \(A\) in \(t-\)direction 
close to the boundary than the Dini-type condition \eqref{DiniTypeCond} as will be illustrated in the following example.
\smallskip

Consider the one-dimensional function \(f(t):=\frac{1}{-\ln(-\ln(t))}\) for \(t>0\). We can calculate
\(f'(t)=\frac{1}{\ln(-\ln(t))^2}\frac{1}{\ln(t)}\frac{1}{t}\), which means that
\[\int_0^{1/4}\frac{f(t)^2}{t}dt=\int_0^{1/4}\frac{1}{t\ln(-\ln(t))^2}dt=\infty,\]
and
\[\int_0^{1/4}|f'(t)|dt=\int_0^{1/4}\frac{1}{t\ln(t)\ln(-\ln(t))^2}dt<\infty.\]
Now, choosing for instance 
\[A(x,t):=\begin{cases}
    I+f(t)I & \textrm{ if }0<t<\frac{1}{4},
    \\
    I+f\big(\frac{1}{4}\big)I & \textrm{ if }t\geq \frac{1}{4},
\end{cases}\]
yields an elliptic operator with a continuous matrix \(A\), that does not satisfy \eqref{DiniTypeCond} or \eqref{DiniTypePerturbation}, but does satisfy \(|\partial_t A|\leq \frac{1}{t}\) and \eqref{cond:L1-Linfty}.

\medskip

If we restrict to the case of \(n=1\), i.e. the upper half plane \(\Omega:=\mathbb{R}^{2}_+\), we obtain an even wider class of operators for which we have solvability of the Dirichlet problem. We assume the \(L^1\) Carleson condition on \(\partial_t A(x,t)\) given by
\begin{align}\sup_{\Delta\subset\partial\Omega \textrm{ boundary ball}}\frac{1}{\sigma(\Delta)}\int_{T(\Delta)}\sup_{(y,s)\in B(x,t,t/2)}|\partial_s A(y,s)| dx dt\leq C<\infty.\label{condition:L^1Carlesontypecond}\end{align}
Then we can state the following theorem.

\begin{thm}\label{thm:MainTheoremn=1}
Let \(L:=\Div(A\nabla\cdot)\) be an elliptic operator satisfying \eqref{eq:DefinitionOfUniformElliptic} on \(\Omega=\mathbb{R}^{2}_+\). If condition \eqref{condition:L^1Carlesontypecond} is satisfied and \(|\partial_t A|t\leq C<\infty\), then \(\omega\in A_\infty(d\sigma)\), i.e. the \(L^p\) Dirichlet boundary value problem for \(L\) is solvable for some \(1<p<\infty\).
\end{thm}

It is clear that on the upper half plane the \(L^1\) Carleson condition on \(\partial_tA\) \eqref{condition:L^1Carlesontypecond} is weaker than the mixed \(L^1-L^\infty\) condition \eqref{cond:L1-Linfty}. Hence, this new condition gives rise to a larger class of elliptic operators with \(\omega\in A_\infty(\sigma)\) that is not covered by other conditions in the literature that we have discussed.

\medskip
The \(L^1\) Carleson condition on \(\partial_tA\) can also be put into context from a different point of view. In fact, Carleson conditions are a widely applied tool in the area of solvability of boundary value problems in the elliptic and parabolic setting. The famous DKP condition can be stated as 
\begin{align}
    \sup_{\Delta\subset\partial\Omega \textrm{ boundary ball}}\frac{1}{\sigma(\Delta)}\int_{T(\Delta)}|\nabla A(y,s)|^2t dx dt\leq C<\infty,\label{condition:TypicalL^2Carlesontypecond}
\end{align}
if we assume \(|\nabla A(x,t)|\leq \frac{C}{t}\), or with \(\mathrm{osc}_{(y,s)\in B(x,t,t/2)}|A(y,s)|\) instead of the term \(|\nabla A(y,s)|\), and originates from \cite{kenig_dirichlet_2001}. The DKP condition and versions thereof have been studied for elliptic and parabolic boundary value problems and yield solvability of the corresponding boundary value problem in many cases. The elliptic Dirichlet boundary value problem was studied in \cite{kenig_dirichlet_2001}, \cite{dindos_lp_2007}, \cite{hofmann_uniform_2021}, and \cite{david_small_2023}, and also for elliptic operators with complex coefficients (cf. \cite{dindos_regularity_2019}) or elliptic systems (cf. \cite{dindos__2021}). Furthermore, the DKP condition was also successfully studied for the elliptic regularity problem in \cite{dindos_boundary_2017}, \cite{mourgoglou_lp-solvability_2022}, \cite{dindos_etal_regularity_2023}, and \cite{feneuil_alternative_2023}. A helpful survey article of the elliptic setting is \cite{dindos_boundary_2023} which contains further references.
Lastly, we also have positive results for the parabolic Dirichlet and regularity boundary value problem in \cite{dindos_parabolic_2020} and \cite{dindos_lp_2024}. 

\medskip
In contrast to the typical DKP condition \eqref{condition:TypicalL^2Carlesontypecond}, the new Carleson condition in this work \eqref{condition:L^1Carlesontypecond} does not contain any derivative in any other direction than the transversal \(t-\)direction, but is of only \(L^1\)-type. Hence \eqref{condition:L^1Carlesontypecond} applies to a different class of operators.


\subsection*{Achknowledgements}

The author would like to thank Jill Pipher and Martin Dindo\v{s} for many helpful discussions, insights, and suggestions to improve on the presentation of the arguments.

\section{Notations and Setting}

Throughout this work \(\Delta=\Delta(P,r):=B(P,r)\cap\partial\Omega\) denotes a boundary ball centered in point \(P\in\partial\Omega\) with radius \(l(\Delta)=r>0\) and 
\[T(\Delta):=\{(x,t)\in\Omega; x\in\Delta, 0<t<l(\Delta)\}\]
its Carleson region. 
The cone with vertex in \(P\in \partial\Omega\) and aperture \(\alpha\) is denoted by
\[\Gamma_\alpha(P):=\{(x,t)\in \Omega; |x-P|<\alpha t\}.\] 
For an open set \(F\subset \Delta\) we define the saw-tooth region \(\Omega_\alpha(F):=\bigcup_{P\in F}\Gamma_\alpha(P)\) and the truncated saw-tooth region \(\tilde{\Omega}_\alpha(F):=\Omega_\alpha(F)\cap T(\Delta)\).
Furthermore, we set \(\mathcal{D}^\eta_k(\Delta)\) as the collection of certain boundary balls with radius \(\eta 2^{-k}\) such that their union covers \(\Delta\) but enlargements of the boundary balls have finite overlap. More explicitly, we ask that \(\chi_\Delta\leq\sum_{Q\in \mathcal{D}_k^\eta}\chi_Q\leq  \sum_{Q\in \mathcal{D}_k^\eta}\chi_{2Q}\leq N\) for some \(N\in\mathbb{N}\), where \(\chi_Q\) is the characteristic function over \(Q\).
Furthermore we denote by \((f)_E:=\fint_E f(x)dx=\frac{1}{\sigma(E)}\int_E f(x) dx\) the mean value of a function \(f\) on a set \(E\subset\mathbb{R}^n\).

\smallskip

For \(P\in \mathbb{R}^n=\partial\Omega\) we define the nontangential maximal function as
\[N_\alpha(u)(P):=\sup_{(x,t)\in \Gamma_\alpha(P)} |u(x,t)|,\]
the mean-valued nontangential maximal function as
\[\tilde{N}_\alpha(u)(P):=\sup_{(x,t)\in \Gamma_\alpha(P)} \Big(\fint_{\big\{(y,s); |y-x|^2+|s-t|^2<\frac{t^2}{4}\big\}}|u(y,s)|^2dyds\Big)^{1/2},\]
and a truncated version as
\[\tilde{N}^r_\alpha(u)(P):=\sup_{(x,t)\in \Gamma_\alpha(P)\cap B(x,t,r)} \Big(\fint_{\big\{(y,s); |y-x|^2+|s-t|^2<\frac{t^2}{4}\big\}}|u(y,s)|^2dyds\Big)^{1/2}.\]
Furthermore, we denote by
\[\mathcal{A}_\alpha(F)(P):=\Big(\int_{\Gamma_\alpha(P)} \frac{|F(x,t)|^2}{t^{n+1}}dxdt\Big)^{1/2}\qquad \textrm{ for }P\in \mathbb{R}^n=\partial\Omega\]
the area function.

\smallskip

We can now consider the \(L^p\)-Dirichlet boundary value problem.
\begin{defin}\label{def:L^pDirichletProblem}
We say the \textit{\(L^p\)-Dirichlet boundary value problem} is solvable for \(L\) if for all boundary data \(f\in C_c(\partial\Omega)\cap L^p(\Omega)\) the unique existent solution \(u\in W^{1,2}(\Omega)\) of 
\[\begin{cases} Lu=0 &\textrm{in }\Omega, \\ u=f &\textrm{on }\partial\Omega.\end{cases}\]
satisfies
\[\Vert N(u) \Vert_{L^p(\partial\Omega)}\leq C \Vert f\Vert_{L^p(\partial\Omega)},\]
where the constant \(C\) is independent of \(u\) and \(f\).
\end{defin}
As usual, we denote the elliptic measure by \(\omega\), the usual \(n\)-dimensional Hausdorff surface measure by \(\sigma\) and the Muckenhoupt and Reverse Hölder spaces by \( A_\infty(d\sigma)\) and \(B_q(d\sigma)\) respectively. The solvability of the \(L^p\) Dirichlet boundary value problem is equivalent to \(\omega\in B_{p'}(d\sigma)\), which means that solvability for some \(p\) is equivalent to \(\omega\in A_\infty(d\sigma)\). 

\medskip

In the setting of the upper half space we would like to introduce a bit more notation and specification. In the following we are 
considering the matrix
\begin{align} 
A(x,t)=\begin{pmatrix} \AP(x,t) & b(x,t) \\ c(x,t) & d(x,t)\end{pmatrix}\label{eq:DefinitionOfMatrixA}
\end{align}
for \((x,t)\in \Omega=\mathbb{R}^{n+1}_+\), where \(\AP\) is the upper \(n\) times \(n\) block of \(A\) and \(b,c,d\) are vector functions with co-domains \(\mathbb{R}^n\) or \(\mathbb{R}\) respectively. We are discussing the operator \(Lu(x,t):=\Div_{x,t}(A(x,t)\nabla_{x,t} u(x,t))\) for \(u\in W^{1,2}(\Omega)\) and the family of operators
\[L^sv(x):=\Div_x(\AP(x,s)\nabla_x v(x)) \qquad\textrm{ for }v\in W^{1,2}(\partial\Omega)=W^{1,2}(\mathbb{R}^n), s>0.\]
We tend to write \(\nabla=\nabla_x\) suppressing the index when dealing with the gradient only in the components of \(x\) and will clarify with \(\nabla_{x,t}\) where we mean the full gradient.

At last, we notice that from \(\frac{|\partial_t A|}{t}\leq C\) and the mixed \(L^1-L^\infty\) Carleson type condition \eqref{cond:L1-Linfty} we can deduce a mixed \(L^2-L^\infty\) Carleson type condition
\begin{align}
    \int_0^\infty\Vert\partial_s A(\cdot,s)\Vert_\infty^2 s ds\leq C<\infty.\label{condition:L^2Carlesontypecond}
\end{align}

A technical result we will need is the following lemma.
\begin{lemma}[\cite{Giaquinta+1984} Chapter V Lemma 3.1]\label{lemma:LEMMA5}
    Let \(f(t)\) be a nonnegative bounded function in \([r_0,r_1], r_0\geq 0\). Suppose that for \(r_0\leq s<t\leq r_1\) we have
    \[f(s)\leq (A(s-t)^{-\alpha}+B)+\sigma f(t)\]
    where \(A,B,\sigma, \alpha\) are nonnegative constants with \(0\leq \sigma<1\). Then for all \(r_0\leq \rho<R\leq r_1\) we have
    \[f(\rho)\leq c(A(R-\rho)^{-\alpha}+B)\]
    where c is a constant depending on \(\alpha\) and \(\sigma\).
\end{lemma}

\subsection{Overview and sketch of the proof}
According to \cite{kenig_square_2016} it is enough to show the local square function bound
\begin{align}\sup_{\Delta\subset\partial\Omega}\sigma(\Delta)^{-1}\int_{T(\Delta)}\vert\nabla_{x,t} u\vert^2t dxdt\lesssim \Vert f\Vert^2_{L^\infty(\partial\Omega)}\label{SquarefctBound}\end{align}
to get \(\omega\in A_\infty(d\sigma)\).


Combining this with the idea from Lemma 2.14 in \cite{auscher_extrapolation_2001} we only need to show that for any boundary cube \(\Delta\subset \partial\Omega\) and any subcube \(Q\subset\Delta\), there exists a good open subset \(F\subset Q\) with \(|F|\geq (1-\gamma) |Q|\) and

\begin{align}\int_{\tilde{\Omega}_\alpha(F)}|\nabla_{x,t} u|^2t dxdt\leq \beta|Q|\label{eq:SqFctonSawtooth},\end{align}

where \(u\in W^{1,2}(\Omega)\) is a solution to boundary data \(f\) with \(\Vert f\Vert_{L^\infty}=1\). The involved constants can be chosen as \(0<\gamma<1\) and \(0<\beta<\infty\) and \(\alpha(\gamma)>1\) sufficiently large. 
\medskip

The first part of the proof follows the outline from \cite{hofmann_dirichlet_2022}. We fix some boundary data \(f\in L^\infty(\partial\Omega)\) with \(\Vert f\Vert_{\infty}=1\) and denote by \(u\in W^{1,2}(\Omega)\) the solution to the Dirichlet problem with this boundary data \(f\). We construct a good set \(F\) with all above requirements that depends on the matrix \(A\) and the chosen boundary ball (see Section \ref{section:GoodSetF}). By introducing a smooth cut-off function \(\psi\) on the sawtooth region we can make use of integration by parts to obtain
\begin{align*}
    \int_{\tilde{\Omega}_\alpha(F)}|\nabla_{x,t} u|^2t dxdt&\lesssim \int_\Omega A\nabla_{x,t} u\cdot \nabla_{x,t} u \psi^2tdxdt
    \\
    &=\int_\Omega A\nabla_{x,t} u\cdot \nabla_{x,t} (u \psi^2t)dxdt + \int_\Omega A\nabla_{x,t} u\cdot u\nabla_{x,t}(\psi^2t)dxdt.
\end{align*}
The first term vanishes due to the PDE for \(u\) and the second one gives rise to several integral terms, most of which we will bound by using elementary estimates and the condition \eqref{cond:L1-Linfty}. However, one of these integral expressions will be 
\[J_{211}=\int_{\Omega}c\cdot \nabla_x(u^2\psi^2) dxdt,\]
where \(c\) is the component in the last row of the matrix \(A\). Here we go back to the idea of a Hodge decomposition for the component \(c\) which originated in \cite{hofmann_square_2015}. In contrast to \cite{hofmann_square_2015} where \(c\) does \textit{not} depend on \(t\), here we have to deal with a family of Hodge decompositions
        \[c(x,t)\chi_{2\Delta}(x)=\AP(x,t)\nabla_x\phi^t(x) + h^t(x)\qquad \textrm{with \(x\)-divergence free }h^t,\]
and we have to develop new ways of controlling this family uniformly (see Section \ref{section:HodgeDecomposition}).

Using this Hodge decomposition, we write
\begin{align}
    J_{211}=\int_{\Omega}A\nabla_x\phi^t\cdot \nabla_x(u^2\psi^2) dxdt&=\int_{\Omega}A\nabla_x\theta\cdot \nabla_x(u^2\psi^2) dxdt\nonumber
    \\
    &\qquad + \int_{\Omega}A\nabla_x\rho\cdot \nabla_x(u^2\psi^2) dxdt\label{eq:ideaofproofIntegralterms}
\end{align}
for an approximation function \(\rho\) and the difference function \(\theta=\phi-\rho\). The idea here is to choose the approximation function \(\rho\) is such a way that the second integral can be bounded and that the difference to \(\phi\) is still good enough so that the first integral involving \(\theta\) can also be bounded. 

In \cite{hofmann_square_2015} the authors suggest the use of an ellipticized heat semigroup \(\tilde{\rho}(x):=e^{-t^2L_{||}}\tilde{\phi}\) where \(L_{||}=\mathrm{div}_x(\AP\nabla_x \cdot)\) and \(\tilde{\phi}\) comes from the Hodge decomposition of their operator. This choice of \(\rho\) gives them the necessary estimates on \(\theta=\phi-\rho\) and \(\rho\). In our setting, we have \(t\) dependencies in the operator \(L_{||}\) and the Hodge functions \(\phi^t\), whence we define analogously \(L^t=L^t_{||}:=\mathrm{div}_x(\AP(\cdot,t)\nabla_x \cdot)\),
\[\rho(x,t):=e^{-t^2L^t}\phi^t, \qquad \textrm{and}\qquad \theta(x,t)=\phi^t(x)-\rho(x,t).\]
It will turn out that this choice allows us to obtain the required bounds on \(\rho\) and \(\theta\). Continuing from \eqref{eq:ideaofproofIntegralterms} we can use integration by parts to move the correct derivatives on the correct functions for which we have corresponding area function bounds of \(\rho\) and \(\theta\) (see \reflemma{lemma:L^2estimatesForSquareFunctions} for the bounds we need). In doing so we will also need to repeat the same steps for the component \(b\) of the matrix.

A significant part of the work in this article lies in establishing the necessary properties of the family \((\phi^t)_t\) and these area function bounds using our new operator \(\rho\). In contrast to \cite{hofmann_square_2015}, we cannot use all the properties of the ellipticized heat semigroup because our \(\rho\) does not satisfy any PDE apriori.
What saves us is the observation that we can decouple the \(t\) dependencies in \(\rho\) through considering \(\rho\) as a special case of 
    \[w(x,t;s):= e^{-t^2L^s}\phi^s(x)\]
        with three variables. It turns out that \(\partial_s w\) satisfies an inhomogeneous parabolic PDE as a function in \(x\) and \(t\) (cf. Section \ref{section:rho}) which gives us an explicit representation of \(\partial_s w\) and hence the partial derivative of \(\rho\) in \(t\)-direction. The use of the resolved Kato conjecture, properties of the heat semigroup and our imposed condition \eqref{cond:L1-Linfty}/\eqref{condition:L^1Carlesontypecond} enable us to prove the necessary area function bounds on the derivatives of \(\rho\) which replace the area function bounds from \cite{hofmann_dirichlet_2022}.
        \medskip
        
        We introduce the Hodge decomposition in detail with useful properties in Section \ref{section:HodgeDecomposition}, before we show the representation of the different partial derivatives of \(\rho\) via decoupling as \(w(x,t;s)\) in Section \ref{section:rho}. Section \ref{section:Kato} recalls the Kato conjecture and useful properties of the heat semigroup and lists consequences for our operator \(\rho\) and its partial derivatives. After that we construct the good set \(F\) in Section \ref{section:GoodSetF} and finally prove the different area function bounds on each of the partial derivatives separately in Section \ref{section:areaFctBounds}. The penultimate section provides the proof of \refthm{thm:MainTheorem} and the last section shows the improvement \refthm{thm:MainTheoremn=1} in the case of the upper half plane.

\section{Hodge decomposition}\label{section:HodgeDecomposition}

First, we fix a boundary ball \(\Delta\) and this boundary ball will remain fixed for the remainder of this article unless stated otherwise. Recall also that we write \(\nabla=\nabla_x\) in the following.
For each \(s>0\) we can find a Hodge decomposition consisting of \(\phi^s\in W_0^{1,2}(3\Delta)\) and \( h^s\in L^2(\Omega)\), where \(h^s\) is \(x\)-divergence free and
\[c(x,s)\chi_{3\Delta}(x)=\AP(x,s)\nabla\phi^s(x)+h^s(x),\]
which means that for any test function \(v\in C_0^\infty(3\Delta)\)
\[\int_{3\Delta}\AP(x,s)\nabla\phi^s(x)\cdot\nabla v(x) dx=\int_{3\Delta}c(x,s)\cdot\nabla v(x)\]
holds. This implies (see also Proposition 4 of \cite{hofmann_dirichlet_2022}) 
\begin{align}\fint_{3\Delta}|\nabla\phi^s|^2 dx\leq C(n,\lambda_0,\Lambda_0),\label{equation:L2ofnablaphiBounded}\end{align}
and through differentiating of both sides with respect to \(s\) also that
\begin{align*}\int_{3\Delta} \AP(x,s)\nabla\partial_s\phi^s(x)\cdot\nabla v(x) dx&= - \int_{3\Delta}\partial_s \AP(x,s)\nabla\phi^s(x)\cdot\nabla v(x) dx 
\\
&\qquad + \int_{3\Delta}\partial_s c(x,s)\cdot\nabla v(x).\end{align*}
Inserting \(v=\partial_s\phi^s\) leads to

\begin{align*}
\lambda_0\Vert \nabla\partial_s\phi^s\Vert_{L^2(3\Delta)}^2&\leq \Big|\int_{3\Delta} \AP(x,s)\nabla\partial_s\phi^s\cdot\nabla\partial_s\phi^s dx\Big|
\\
&\leq\Big|\int_{3\Delta}\partial_s \AP(x,s)\nabla\phi^s(x)\nabla\partial_s\phi^s(x) dx\Big|
\\
&\qquad+ \Big|\int_{3\Delta}\partial_s c(x,s)\cdot\nabla\partial_s\phi^s(x) dx\Big|
\\
&\lesssim\Vert\partial_s \AP(\cdot,s)\Vert_{L^\infty(3\Delta)}\Vert \nabla\phi^s\Vert_{L^2(3\Delta)}\Vert \nabla\partial_s\phi^s\Vert_{L^2(3\Delta)} 
\\
&\qquad + \Vert\partial_s c(\cdot,s)\Vert_{L^\infty(3\Delta)}\Vert \nabla\partial_s\phi^s\Vert_{L^2(3\Delta)}|\Delta|^{1/2},
\end{align*}

and hence 

\begin{align}
\Vert \nabla\partial_s\phi^s\Vert_{L^2(3\Delta)}&\lesssim \Vert\partial_s\AP(\cdot,s)\Vert_{L^\infty(3\Delta)}(\Vert \nabla\phi^s\Vert_{L^2(3\Delta)}+|\Delta|^{1/2})\nonumber
\\
&\lesssim \Vert\partial_s \AP(\cdot,s)\Vert_{L^\infty(3\Delta)}|\Delta|^{1/2}.\label{equation:L2ofpartial_snablaphiBounded}
\end{align}

Analogously we can find a Hodge decomposition of \(b\), i.e. a family of functions \(\tilde{\phi}^s\) such that \(b\chi_{3\Delta}=\AP\nabla\tilde{\phi}^s+\tilde{h}^s\). All the results about \(\phi^s\) hold analogously for \(\tilde{\phi}^s\).

Due to Theorem 1 in \cite{mazya_sobolev_2011}, we obtain that the partial derivative of \(\phi^s\) with respect to \(s\) is not only a distributional derivative but also the weak derivative in the direction transversal to the boundary almost everywhere.

\section{Approximation function \(\rho\) and difference function \(\theta\)}\label{section:rho}

Let \(\eta>0\) be a parameter to be determined later in the proof of \refthm{thm:MainTheorem}. We define 
\[w(x,t;s):=e^{-tL^s}\phi^s(x)\]
as the solution to the ("t-independent") heat equation
\begin{align}\begin{cases} \partial_t w(x,t;s)-L^s w(x,t;s)=0 &\textrm{for }(x,t)\in \Omega, \\ w(x,0;s)=\phi^s(x) &\textrm{for }x\in\partial\Omega\end{cases}\label{eq:PDEforw}\end{align}
for fixed \(s>0\) using the heat semigroup.
We define the by \(\eta\) scaled approximation function with ellipticized homogeneity
\[\rho_\eta(x,s):=w(x,\eta^2s^2;s),\]
and the difference function
\[\theta_\eta(x,s):=\phi^s(x)-\rho_\eta(x,s).\]
The approximating function \(\rho_\eta(x,s)\) is not the solution of one heat equation anymore.  
We can think of \(\rho\) as the diagonalization of a family of solutions with elliptic homogeneity to different heat equations \(\partial_t-L^s\) with different with initial data \(\phi^s\) where \(s\) is the time at which we evaluate the corresponding solution.

\subsection{The different parts of the derivatives of \(\theta\)}

We will have to take the partial derivative of \(\theta_\eta\) in the \(s\)-component. Thereby, we see that
\begin{align*}
    \partial_s \theta_\eta(x,s)&=\partial_s \phi^s(x) - \partial_s w(x,\eta^2s^2;s)
    \\
    &=\partial_s\phi^s(x)-2\eta^2s\partial_t w(x,t;s)|_{t=\eta^2 s^2}-\partial_s w(x,t;s)|_{t=\eta^2 s^2}.
\end{align*}

We can calculate the second and third term more explicitly. For the first of them we get by the PDE \eqref{eq:PDEforw} that \(w(x,t,s)\) satisfies for fixed \(s\) that
\begin{align*}
    \partial_t w(x,t;s)|_{t=\eta^2 s^2}=L^s w(x,\eta^2 s^2;s).
\end{align*}

For the second one we need to work a bit more. To start with we can observe that \(\partial_sw\) satisfies a PDE. From \eqref{eq:PDEforw} we obtain
\begin{align*}
    \partial_s\partial_t w(x,t;s)&=\partial_s \Div(\AP(x,s)\nabla w(x,t;s))
    \\
    &=\Div(\partial_s \AP(x,s)\nabla w(x,t;s))+L^s\partial_sw(x,t;s).
\end{align*}
We can now set \(v_1(x,t)\) as the solution to
\[\begin{cases} \partial_t v_1(x,t)=L^s v_1(x,t) &\textrm{for }(x,t)\in\Omega, \\ v_1(x,0)=\partial_s\phi^s(x) &\textrm{for }(x,0)\in\partial\Omega.\end{cases}\]
and \(v_2(x,t)\) as the solution to
\[\begin{cases} \partial_t v_2(x,t)=L^s v_2(x,t) + \Div(\partial_s \AP(x,s)\nabla w(x,t;s)) &\textrm{for }(x,t)\in\Omega, \\ v_2(x,0)=0 &\textrm{for }(x,0)\in\partial\Omega.\end{cases}\]

Since \(\partial_s w(x,0;s)=\partial_s\phi^s\) we note that \(\partial_s w\) and \(v_1+v_2\) satisfy the same linear PDE with same boundary data and hence must be equal. Next, we can give explicit representations for \(v_1\) by the heat semigroup and for \(v_2\) by applying Duhamel's principle. These are
\[v_1(x,t;s)=e^{-tL^s}\partial_s\phi^s(x),\]
and
\[v_2(x,t;s)=\int_0^t e^{-(t-\tau)L^s}\Div(\partial_s \AP(x,s)\nabla w(x,\tau,s))d\tau.\]

Together we get for the derivative of \(\theta\) in the \(s\)-component
\begin{align*}
    \partial_s \theta_\eta(x,s)&=\partial_s\phi^s-2\eta^2 s L^s w(x,\eta^2 s^2;s)-e^{-\eta^2s^2L^s}\partial_s\phi^s(x)
    \\
    &\qquad-\int_0^{\eta^2s^2} e^{-(\eta^2s^2-\tau)L^s}\Div(\partial_s\AP(x,s)\nabla w(x,\tau;s))d\tau
    \\
    &=\partial_s\phi^s-2\eta^2 s L^s w(x,\eta^2 s^2;s)-e^{-\eta^2s^2L^s}\partial_s\phi^s(x)
    \\
    &\qquad-\int_0^{s} 2\eta^2\tau e^{-\eta^2(s^2-\tau^2)L^s}\Div(\partial_s\AP(x,s)\nabla w(x,\eta^2\tau^2;s))d\tau
    \\
    &\eqcolon \partial_s\phi^s(x)-w_t(x,s)-w_s^{(1)}(x,s)-w_s^{(2)}(x,s).
\end{align*}

We are going to need Caccioppoli type inequalities for the parts \(w_t\) and \(w_s^{(2)}\) of the derivative of \(\theta\). Since we ellipticized the heat semigroup in \(\rho\), we are expecting Caccioppoli inequalities for usual elliptic Whitney cubes.

\begin{lemma}\label{lemma:CaccTypeForw_t}
Let \(2\hat{Q}\subset\Omega\) be a Whitney cube with \(\mathrm{dist}(\hat{Q},\partial\Omega)\approx l(\hat{Q})\approx s\). If we set
\[v(x,t;s):=\partial_t w(x,t^2;s)=\partial_t e^{-t^2L^s}\phi^s,\]
we have the Caccioppoli type inequality
\[\int_{\hat{Q}}|\nabla v|^2dxdt\lesssim \frac{1}{s^2}\int_{2\hat{Q}}|v|^2dxdt.\]
\end{lemma}

\begin{proof}\label{remark:CaccForpartial_tw}
We observe that \(v\) satisfies the PDE
\begin{align*}
\partial_t v(x,t;s)&=\partial_t\partial_t w(x,t^2;s) = \partial_t (-2tL^s w(x,t^2;s)) 
\\
&= -2L^s w(x,t^2;s) - 2tL^s v(x,t;s)=\frac{v(x,t;s)}{t} - 2tL^s v(x,t^2;s).
\end{align*}

Analogously to the standard proof for Caccioppoli's inequality we take a smooth cut-off function \(\psi\in C^\infty(\Omega)\) that satisfies 
\[\psi\equiv 1 \textrm{ on } \hat{Q},\]
and
\[\psi\equiv 0 \textrm{ on } \Omega\setminus 2\hat{Q},\]
and
\[\Vert \nabla \psi \Vert \lesssim \frac{1}{l(\hat{Q})}\approx \frac{1}{s}.\]
Since \(s\) is fixed for the whole argument, we can now use above PDE to get
\begin{align*}
    \int_{\hat{Q}}|\nabla v|^2dxdt&\lesssim \frac{1}{l(\hat{Q})}\int_{2\hat{Q}}t|\nabla v|^2\psi^2dxdt\lesssim \frac{1}{l(\hat{Q})}\int_{2\hat{Q}}t A(x,s)\nabla v \cdot \nabla v \psi^2dxdt
    \\
    &\lesssim \frac{1}{l(\hat{Q})}\int_{2\hat{Q}}t L^s v  v \psi^2dxdt +\frac{1}{l(\hat{Q})}\int_{2\hat{Q}}t A(x,s)\nabla v \cdot \nabla\psi^2 v dxdt
    \\
    &\lesssim \frac{1}{l(\hat{Q})}\int_{2\hat{Q}} \partial_t v^2(x,t;s) \psi^2dxdt - \frac{1}{l(\hat{Q})^2}\int_{2\hat{Q}}  v^2 \psi^2dxdt 
    \\
    &\qquad+ \Big(\int_{2\hat{Q}}|\nabla v|^2 \psi^2dxdt\Big)\Big(\int_{2\hat{Q}}|\nabla \psi|^2 |v|^2dxdt\Big)^{1/2}
\end{align*}
Note that for the first term
\[\Big|\int_{2\hat{Q}} \partial_t v^2 \psi^2dxdt\Big|\leq\Big|\int_{2\hat{Q}} \partial_t (v^2 \psi^2)dxdt\Big| + \Big|\int_{2\hat{Q}}  v^2 \partial_t\psi^2dxdt\Big|\lesssim  \frac{1}{l(\hat{Q})}\int_{2\hat{Q}} v^2 dxdt,\]
and
\[\int_{2\hat{Q}}|\nabla \psi|^2 |v|^2dxdt\lesssim \frac{1}{l(\hat{Q})^2}\int_{2\hat{Q}} v^2 dxdt.\]
For the third term, hiding the first expression on the left side with the same estimate, we obtain
\begin{align*}
\int_{\hat{Q}}|\nabla v|^2dxdt\lesssim \frac{1}{s^2}\int_{2\hat{Q}}|v|^2dxdt.
\end{align*}

\end{proof}

 \begin{lemma}\label{Lemma:CacciopolliTypeInequalities}
     Let \(2\hat{Q}\subset \Omega\) be a Whitney cube with \(\mathrm{dist}(\hat{Q},\partial\Omega)\approx l(\hat{Q})\approx s\).
     If we set 
     \[v(x,t;s):=v_2(x,t^2;s)=\int_0^t 2\tau e^{-(t^2-\tau^2)L^s}\Div(\partial_s\AP(x,s)\nabla e^{-\tau^2 L^s}\phi^s(x))d\tau,\]
     then we have
     \[\int_{\hat{Q}} |\nabla_x v(x,t;s)|^2dxdt\lesssim \frac{1}{s^2}\int_{2\hat{Q}}|v(x,t;s)|^2dxdt + \int_{2\hat{Q}} |\partial_s\AP(x,s)\nabla e^{-t^2L^s}\phi^s|^2dxdt,\]
     and
     \begin{align*}
         \int_{\hat{Q}} |\nabla_x\partial_t v(x,t;s)|^2dxdt&\lesssim \frac{1}{s^2}\int_{2\hat{Q}}|\partial_t v(x,t;s)|^2dxdt
         \\
         & \qquad + \frac{1}{s^2}\int_{2\hat{Q}} |\partial_t\AP(x,s)\nabla e^{-t^2L^s}\phi^s|^2dxdt
         \\
         &\qquad + \int_{2Q}|\partial_t\AP(x,s)\nabla\partial_t e^{-t^2L^s}\phi^s|^2dxdt.
     \end{align*}
 \end{lemma}

\begin{proof}
    First we denote \(\tilde{v}:=\partial_t v\) to shorten notation. We observe that \(v\) and \(\tilde{v}\) satisfy the following PDEs
    \begin{align*}
        \partial_t v &= 2t\Div(\partial_s \AP\nabla e^{-t^2L^s}\phi^s) - \int_0^t 4t\tau L^se^{-(t^2-\tau^2)L^s}\Div(\partial_s \AP(x,s)\nabla e^{-\tau^2 L^s}\phi^s(x))d\tau
        \\
        &= 2t\Div(\partial_s \AP\nabla e^{-t^2L^s}\phi^s) - 2t L^s v,
    \end{align*}
    and
    \begin{align*}
        \partial_t \tilde{v} &= \Div(\partial_s \AP\nabla e^{-t^2L^s}\phi_s) + 2t\Div(\partial_s \AP\nabla\partial_t e^{-t^2L^s}\phi_s) - 2 L^s v - 2t L^s \tilde{v}.
    \end{align*}

    Let \(\psi\) be a smooth cut-off with \(\psi\equiv 1\) on \(\hat{Q}\) and \(\psi\equiv 0\) on \(\Omega\setminus 2\hat{Q}\) with \(|\nabla \psi|+|\partial_t\psi|\lesssim \frac{1}{s}\) as in the previous proof. First for \(v\), we calculate

    \begin{align*}
        \int_{\hat{Q}} |\nabla v(x,t;s)|^2dxdt&\leq \frac{1}{l(\hat{Q})}\int_{2\hat{Q}} tA(x,s)\nabla v(x,t;s)\cdot \nabla v(x,t;s) \psi^2(x,t)dxdt
        \\
        &=\frac{1}{l(\hat{Q})}\int_{2\hat{Q}} tL^s\nabla v(x,t;s) \nabla v(x,t;s) \psi^2(x,t)dxdt 
        \\
        &\qquad- \frac{1}{l(\hat{Q})}\int_{2\hat{Q}}t A(x,s)\nabla v(x,t;s)\cdot v(x,t;s) \nabla\psi^2(x,t)dxdt
        \\
        &=:I_1+I_2
    \end{align*}
    For \(I_2\) we directly get for a small \(\sigma\)
    \begin{align*}I_2&\lesssim \Big(\int_{2\hat{Q}}|\nabla v\psi|^2dxdt\Big)^{1/2}\Big(\int_{2\hat{Q}}|\nabla\psi v|^2dxdt\Big)^{1/2}
    \\
    &\lesssim \sigma\int_{2\hat{Q}}|\nabla v\psi|^2dxdt + C_\sigma\frac{1}{l(\hat{Q})^2}\int_{2\hat{Q}}|v|^2dxdt,\end{align*}
    and can hide the first term on the left side.

    For \(I_1\) we have by use of the PDE for \(v\) and integration by parts
    \begin{align*}
       I_1&\lesssim -\frac{1}{l(\hat{Q})}\int_{2\hat{Q}}\partial_t v v \psi^2 dxdt + \int_{2\hat{Q}} 2\Div(\partial_s\AP\nabla e^{-t^2L^s}\phi^s) v\psi^2dxdt
       \\
       &\lesssim \frac{1}{l(\hat{Q})}\int_{2\hat{Q}}\partial_t (v^2 \psi^2) dxdt + \frac{1}{l(\hat{Q})}\int_{2\hat{Q}} v^2 \partial_t\psi^2 dxdt
       \\
       & \qquad+ \int_{2\hat{Q}} 2\partial_s\AP\nabla e^{-t^2L^s}\phi^s (\nabla v\psi^2 + v\nabla\psi^2)dxdt
       \\
       &\lesssim \frac{1}{l(\hat{Q})^2}\int_{2\hat{Q}} |v|^2 dxdt + C_\sigma\int_{2\hat{Q}} |\partial_s \AP(x,s)\nabla e^{-t^2L^s}\phi^s|^2dxdt + \sigma\int_{2\hat{Q}} |\nabla v|^2\psi^2 dxdt
    \end{align*}
    Hiding the last term on the left side gives in total
    \[\int_{\hat{Q}} |\nabla v(x,t;s)|^2dxdt\lesssim \frac{1}{s^2}\int_{2\hat{Q}} |v(x,t;s)|^2 dxdt + \int_{2\hat{Q}} |\partial_s \AP(x,s)\nabla e^{-t^2L^s}\phi^s|^2dxdt,\]
    which completes the proof of the first Caccioppoli type inequality.

    \bigskip

    For the second one for \(\tilde{v}\) we proceed analogously with just more terms to handle. We have
    \begin{align*}
        \int_{\hat{Q}} |\nabla \tilde{v}(x,t;s)|^2dxdt&\leq \frac{1}{l(\hat{Q})}\int_{2\hat{Q}} tA(x,s)\nabla \tilde{v}(x,t;s)\cdot \nabla \tilde{v}(x,t;s) \psi^2(x,t)dxdt
        \\
        &=\frac{1}{l(\hat{Q})}\int_{2\hat{Q}}t L^s \tilde{v}(x,t;s) \tilde{v}(x,t;s) \psi^2(x,t)dxdt 
        \\
        &\qquad- \frac{1}{l(\hat{Q})}\int_{2\hat{Q}}t A(x,s)\nabla \tilde{v}(x,t;s)\cdot \tilde{v}(x,t;s) \nabla\psi^2(x,t)dxdt
        \\
        &=:II_1+II_2
    \end{align*}
    The term \(II_2\) works completely analogous to \(I_2\). For \(II_1\) however, we have
    \begin{align*}
       II_1&\lesssim \frac{1}{l(\hat{Q})}\int_{2\hat{Q}} 2\Div(\partial_s\AP\nabla e^{-t^2L^s}\phi^s) \tilde{v}\psi^2dxdt + \int_{2\hat{Q}} 2\Div(\partial_s\AP\nabla \partial_t e^{-t^2L^s}\phi^s) \tilde{v}\psi^2dxdt
       \\
       &\qquad-\frac{1}{l(\hat{Q})}\int_{2\hat{Q}} L^s v \tilde{v}\psi^2 dxdt-\frac{1}{l(\hat{Q})}\int_{2\hat{Q}}\partial_t \tilde{v} \tilde{v} \psi^2 dxdt
    \end{align*}
    Using integration by parts on the first, second and forth and the PDE for \(v\) on the third term yields
    \begin{align*}
       II_1&\lesssim \frac{1}{l(\hat{Q})}\int_{2\hat{Q}} 2\partial_s\AP\nabla e^{-t^2L^s}\phi^s \cdot\nabla\tilde{v}\psi^2dxdt + \int_{2\hat{Q}} 2\partial_s\AP\nabla \partial_t e^{-t^2L^s}\phi^s \cdot\nabla\tilde{v}\psi^2dxdt
       \\
       &\qquad +\frac{1}{l(\hat{Q})}\int_{2\hat{Q}}2 \partial_s\AP\nabla e^{-t^2L^s}\phi^s \tilde{v}\cdot\nabla\psi^2dxdt + \int_{2\hat{Q}} 2\partial_s\AP\nabla \partial_t e^{-t^2L^s}\phi^s \tilde{v}\cdot\nabla\psi^2dxdt
       \\
       &\qquad-\frac{1}{l(\hat{Q})^2}\int_{2\hat{Q}}  \partial_t v \tilde{v}\psi^2 dxdt - \frac{1}{l(\hat{Q})}\int_{2\hat{Q}} 2\Div(\partial_s\AP\nabla e^{-t^2L^s}\phi^s) \tilde{v}\psi^2dxdt 
       \\
       &\qquad-\frac{1}{2l(\hat{Q})}\int_{2\hat{Q}}\partial_t (\tilde{v}^2 \psi^2) dxdt +\frac{1}{2l(\hat{Q})}\int_{2\hat{Q}}\tilde{v}^2 \partial_t \psi^2 dxdt.
    \end{align*}
    Furthermore, we recall that \(\partial_t v=\tilde{v}\) and proceeding with all the eight integrals in the from above known ways we get
    \begin{align*}
        II_1&\lesssim C_\sigma\frac{1}{l(\hat{Q})^2}\int_{2\hat{Q}} |\partial_s\AP(x,s)\nabla e^{-t^2L^s}\phi^s|^2dxdt + \sigma \int_{2\hat{Q}} |\nabla \tilde{v}\psi|^2 dxdt
        \\
        &\qquad + C_\sigma\int_{2\hat{Q}} |\partial_s\AP(x,s)\nabla \partial_t e^{-t^2L^s}\phi^s|^2dxdt + \sigma\int_{2\hat{Q}} |\nabla \tilde{v}\psi|^2 dxdt
        \\
        &\qquad+ \frac{1}{l(\hat{Q})^2}\int_{2\hat{Q}} |\partial_s\AP(x,s)\nabla e^{-t^2L^s}\phi^s|^2dxdt + \frac{1}{l(\hat{Q})^2}\int_{2\hat{Q}} |\tilde{v}\psi|^2 dxdt
        \\
        &\qquad + \int_{2\hat{Q}} |\partial_s\AP(x,s)\nabla \partial_t e^{-t^2L^s}\phi^s|^2dxdt + \frac{1}{l(\hat{Q})^2}\int_{2\hat{Q}} | \tilde{v}\psi|^2 dxdt
        \\
        &\qquad + \frac{1}{l(\hat{Q})^2}\int_{2\hat{Q}} |\tilde{v}|^2 dxdt + \frac{1}{l(\hat{Q})^2}\int_{2\hat{Q}} |\tilde{v}|^2 dxdt.
    \end{align*}
    The last step requires hiding the terms with a small factor \(\sigma\) on the left side, which completes the proof.

\end{proof}

\section{Kato conjecture and properties of the heat semigroup}
\label{section:Kato}

The following two results are the solution to the Kato conjecture which was fully resolved in \cite{Auscher_Kato} for \(p=2\). The \(L^p\) theory for other \(p\) was first fully established in \cite{auscher_necessary_2007} with partial \(L^p\) theory results appearing earlier (please refer to the introduction of \cite{auscher_necessary_2007} for the full history). Note that we define \(\dot W^{1,p}(\mathbb{R}^n)\) as the closure of \(C_0^\infty(\mathbb{R}^n)\) under the seminorm given by \(\Vert f\Vert_{\dot W^{1,p}}:=\Vert \nabla f\Vert_{L^p(\mathbb{R}^n)}\). The Kato \(L^p\) theory for the most general elliptic operator is shown in Thm 4.77 in \cite{hofmann_lp_2022}.
\begin{prop}\label{prop:KatoConjecture}
    For a function \(f\in \dot W^{1,p}(\mathbb{R}^n)\) we have
    \[\Vert L^{1/2} f\Vert_{L^p(\mathbb{R}^n)}\leq C(n,\lambda_0,\Lambda_0)\Vert \nabla f\Vert_{L^p(\mathbb{R}^n)}\]
    for all \(1<p<\infty\). Furthermore, there exists \(\varepsilon_1>0\) such that if \(1<p<2+\varepsilon_1\) then
    \[\Vert \nabla f\Vert_{L^p(\mathbb{R}^n)}\leq C(n,\lambda_0,\Lambda_0)\Vert L^{1/2} f\Vert_{L^p(\mathbb{R}^n)}.\]
\end{prop}

The classical Kato solution is the following case, the \(L^2\) case.
\begin{prop}\label{cor:NablaL^-1/2Bounded}
    For a function \(f\in L^2(\mathbb{R}^n)\) we have
    \[\Vert \nabla L^{-1/2} f\Vert_{L^2(\mathbb{R}^n)}\approx\Vert f\Vert_{L^2(\mathbb{R}^n)}\]
    and if f is a vector valued function
    \[\Vert L^{-1/2} \Div(f)\Vert_{L^2(\mathbb{R}^n)}\approx\Vert f\Vert_{L^2(\mathbb{R}^n)}.\]
\end{prop}

The following result assumes a uniform elliptic \(t\)-independent operator \(L\), i.e. \(L_{||}u(x)=\Div_x(\AP(x)\nabla_x u(x))\) for \(x\in\mathbb{R}^n\).

\begin{prop}[Proposition 4.3  from \cite{hofmann_dirichlet_2022}]\label{prop:L2NormBoundsOfHeatSemigroup}
If \(L_{||}\) is t-independent then the following norm estimates hold for \(f\in L^2(\mathbb{R}^n),l\in \mathbb{N}_0, t>0\) and constants \(c_l,C>0\):

\begin{itemize}
    \item \(\Vert \partial_t^l e^{-tL_{||}}f\Vert_{L^2(\mathbb{R}^n)}\leq c_l t^{-l} \Vert f\Vert_{L^2(\mathbb{R}^n)}\),
    \item \(\Vert \nabla e^{-tL_{||}}f\Vert_{L^2(\mathbb{R}^n)}\leq C t^{-1/2} \Vert f\Vert_{L^2(\mathbb{R}^n)}\).
\end{itemize}
\end{prop}

The following proposition is a generalization of Equation (5.6) in \cite{hofmann_lp_2022} and its proof also borrows the main ideas from there.

\begin{prop}\label{prop:L2NormGoesTo0}
If we have a compact set \(K\subset\mathbb{R}^n\), a family \(f_s\in W_0^{1,2}(K)\) and a family of uniform elliptic operators \(L^s\) with the same constants \(\lambda_0\) and \(\Lambda_0\) then
\[\lim_{s\to \infty}\Vert e^{-sL^s}f_s\Vert_{L^2(\mathbb{R}^n)}=0.\]
\end{prop}

\begin{proof}
This proof works analogously to the proof of  Equation (5.6) in \cite{hofmann_lp_2022}. Note that the uniform support of \(f_s\) ensures that the support of the chosen \(\varphi_\epsilon^s\) is also subset of \(K\) and has uniformly bounded measure.
\end{proof}

The kernel of the heat semigroup admits certain kernel bounds.

\begin{prop}[Kernel bounds, Prop 4.3 in \cite{hofmann_lp_2022} or Thm 6.17 in \cite{ouhabaz_analysis_2004}]
Let \(l\in\mathbb{N}\) and \(K_t(x,y)\) the kernel of the operator \(e^{-tL_{||}}\), then there exists \(C=C(n,\lambda)>0,\beta=\beta(n,\lambda)\in (0,1)\) such that 
    \begin{align}|\partial^l_tK_t(x,y)|\leq C_lt^{-\frac{n}{2}-l}e^{-\frac{\beta|x-y|^2}{t}}.\label{eq:kernelbounds}\end{align}
\end{prop}

We have the two following important results
\begin{prop}\label{Proposition11}[Proposition 11 in \cite{hofmann_dirichlet_2022}]
    Let \(\eta>0,\alpha>0\) and \(L_{||}\) be \(t\)-independent. Then we have for all \((y,t)\in\Gamma_{\eta\alpha}(x)\)
    \[\eta^{-1}\partial_t e^{-(\eta t)^2L_{||}}f(y)\lesssim M(\nabla f)(x),\]
    and hence
    \[\Vert \eta^{-1} N_{\eta\alpha}[\partial_t e^{-(\eta t)^2L_{||}}f]\Vert_{L^p}\lesssim \Vert\nabla f\Vert_{L^p}\]
    for all \(p>1\) and \(f\in W^{1,p}(\mathbb{R}^n)\).
\end{prop}

For us relevant is the adaptation to our case with the family \(L^s\) instead of the \(t\)-independent \(L_{||}\).

\begin{cor}\label{cor:PointwiseBoundsofw_tAndrhoByM(nablaphi^s)}
    Let \(\eta>0,\alpha>0\). Then we have for all \((y,t)\in\Gamma_{\eta\alpha}(x)\)
    \begin{align} \eta^{-1}w_t(y,s)=sL^se^{-(\eta s)^2L^s}\phi^s\lesssim M(\nabla \phi^s)(x),\label{eq:w_tpointwiseBounded}\end{align}
    and
    \begin{align}\frac{\rho_\eta(y,s)}{\eta s}=\frac{e^{-(\eta s)^2L^s}(\phi^s-m)(y)}{s}\lesssim M(\nabla \phi^s)(x),\label{eq:rhopointwiseBounded}\end{align}
    where
    \[m(s):=\fint_{\Delta(x, 2\eta\alpha s)}\phi^s(x) dx=(\phi^s(\cdot))_{\Delta(x, 2\eta\alpha s)}.\]
    Furthermore, we also have
    \begin{align}
        \eta^{-2}s^2L^sw_t(y,s)=\eta^{-2}s^3L^se^{-(\eta s)^2L^s}\phi^s\lesssim M(\nabla \phi^s)(x).\label{eq:L^sw_tpointwiseBounded}
    \end{align}
\end{cor}

\begin{proof}
The proof of this corollary is completely analogous to the proof of \refprop{Proposition11} or Proposition 11 in \cite{hofmann_dirichlet_2022}. We still would like to mention all of the minor observations that lead to that conclusion.
\smallskip

    The first inequality \eqref{eq:w_tpointwiseBounded} follows directly from \refprop{Proposition11}. For \eqref{eq:rhopointwiseBounded} we can observe that the only property of \(\partial_t e^{-t^2L_{||}}\) that was used in the proof of \refprop{Proposition11} was the kernel bound \eqref{eq:kernelbounds} of \(\partial_t e^{-t^2L_{||}}\). However, the same proposition gives the same kernel bound for \(\frac{e^{-t^2L_{||}}}{t}\) and \(tL_{||}\partial_t e^{-t^2L_{||}}\) and hence we get with a completely analogous proof to that of \eqref{eq:w_tpointwiseBounded} also \eqref{eq:rhopointwiseBounded} and \eqref{eq:L^sw_tpointwiseBounded}. Since the operator \(\partial_t e^{-t^2L_{||}}\) kills constants but \(\frac{e^{-t^2L_{||}}}{t}\) does not, we need to subtract the mean value in the case of \eqref{eq:rhopointwiseBounded}.
\end{proof}

We will also apply the following basic lemma which can be proved with the above introduced statements and ideas.

\begin{lemma}\label{lemma:nablaSemigroupBoundedByNablaf}
For \(t\)-independent \(L_{||}\), if \(f\in W_0^{1,2}(\mathbb{R}^n)\) then 
\[\Vert \nabla e^{-tL_{||}}f\Vert_{L^2(\mathbb{R}^n)}\leq C \Vert \nabla f\Vert_{L^2(\mathbb{R}^n)},\]
and
\[\Vert \nabla \partial_t e^{-t^2L_{||}}f\Vert_{L^2(\mathbb{R}^n)}\leq C t^{-1}\Vert \nabla f\Vert_{L^2(\mathbb{R}^n)}\]
\end{lemma}

\begin{proof}
    The first inequality is an easy corollary from the Kato conjecture, i.e. we have with \refprop{prop:KatoConjecture} and \refprop{prop:L2NormBoundsOfHeatSemigroup}
    \begin{align*}
        \Vert \nabla e^{-tL_{||}}f\Vert_{L^2(\mathbb{R}^n)}&\lesssim \Vert L_{||}^{1/2} e^{-tL_{||}}f\Vert_{L^2(\mathbb{R}^n)}\lesssim \Vert  e^{-tL_{||}}L_{||}^{1/2} f\Vert_{L^2(\mathbb{R}^n)}
        \\
        &\lesssim \Vert L_{||}^{1/2} f\Vert_{L^2(\mathbb{R}^n)}\lesssim \Vert  \nabla f\Vert_{L^2(\mathbb{R}^n)}.
    \end{align*}
    For the second one we have
    \begin{align*}
        \Vert\nabla \partial_t e^{-t^2L_{||}}f\Vert_{L^2(\mathbb{R}^n)}^2 &\lesssim \int_{\mathbb{R}^n} A\nabla \partial_t e^{-t^2L_{||}}f \cdot \nabla \partial_t e^{-t^2L_{||}}f dx
        \\
        &\lesssim \int_{\mathbb{R}^n} \frac{\partial_{tt} e^{-t^2L_{||}}f}{2t}\partial_t e^{-t^2L_{||}}f dx
        \\
        &\lesssim \frac{1}{t}\Vert\partial_{tt} e^{-t^2L_{||}}f\Vert_{L^2(\mathbb{R}^n)} \Vert\partial_t e^{-t^2L_{||}}f\Vert_{L^2(\mathbb{R}^n)}
    \end{align*}
    Again, we can observe in the proof of in \refprop{Proposition11} (or Proposition 11 in \cite{hofmann_dirichlet_2022}) that the only properties of \(\partial_t e^{-t^2L_{||}}\) that were used are bounds of its kernel \eqref{eq:kernelbounds}. However, the same proposition gives the same bounds for \(t\partial_{tt}e^{-t^2L_{||}}\) and hence we get 
    \[\Vert\partial_{tt} e^{-t^2L_{||}}f\Vert_{L^2(\mathbb{R}^n)}\lesssim \frac{1}{t}\Vert\nabla f\Vert_{L^2(\mathbb{R}^n)}\qquad\textrm{and}\qquad \Vert\partial_{t} e^{-t^2L_{||}}f\Vert_{L^2(\mathbb{R}^n)}\lesssim \Vert\nabla f\Vert_{L^2(\mathbb{R}^n)}.\]
    So in total we have
    \[ \Vert\nabla \partial_t e^{-t^2L_{||}}f\Vert_{L^2(\mathbb{R}^n)}^2\lesssim \frac{1}{t^2}\Vert\nabla f\Vert_{L^2(\mathbb{R}^n)}^2.\]
\end{proof}

Furthermore, we have off-diagonal estimates for operators involving the heat semigroup. We quote the following \(L^2-L^2\) case (see Prop. 3.1 in \cite{auscher_necessary_2007}).
\begin{prop}\label{prop:off-diagonal}
    We say an operator family \(T=(T_t)_{t>0}\) satisfies \(L^2-L^2\) off-diagonal estimates, if there exists \(C,\alpha>0\) such that for all closed sets \(E\) and \(F\) and all \(t>0\)
    \[\Vert T_t(h)\Vert_{L^2(F)}\leq Ce^{-\alpha\frac{d(E,F)^2}{t}}\Vert h\Vert_{L^2(E)},\]
    where \(\supp(h)=E\) and \(d(E,F)\) is the semi-distance induced on sets by Euclidean distance.

    \medskip
    Then the families \((e^{-tL})_{t>0}, (t\partial_te^{-tL})_{t>0}\), and \((\sqrt{t}\nabla e^{-tL})_{t>0}\) satisfy \(L^2-L^2\) off-diagonal estimates.
\end{prop}

\section{The good set \(F\) and local estimates on \(F\)}\label{section:GoodSetF}

In this section we want to construct the good set \(F\). The main idea from here is that we want to bound \(M[\nabla\phi^s]\) pointwise on the good set \(F\). In contrast to \cite{hofmann_dirichlet_2022}, we now have to bound the family of functions \((M[\nabla\phi^s])_s\) pointwise and uniformly in \(s\) on the same good set \(F\). This can only be expected if there is some condition on the behavior of \(\partial_s\AP\) and this step simplifies in the case of the upper half plane (see proof of \refthm{thm:MainTheoremn=1}). For all dimensions \(n\in \mathbb{N}\) we have the following lemma.

\begin{lemma}\label{lemma:UniformBoundOnM[nablaphi^s]}
Assume condition \eqref{cond:L1-Linfty}. Then for small \(\gamma>0\) there exists a large \(\kappa_0>0\) and \(F\subset 3\Delta\) such that
\[M[\nabla\phi^s](x)\leq M[|\nabla\phi^s|^2]^{1/2}(x)\leq C(n,\lambda_0,\Lambda_0)\kappa_0\qquad \textrm{for all }x\in F, \textrm{ and }s>0,\]
where \(|3\Delta\setminus F|\leq \gamma |\Delta|\).
\end{lemma}

\begin{proof}
First, by Fundamental Theorem of Calculus we observe
\[|\nabla\phi^s(x)|^2=|\nabla\phi^1(x)|^2+\int_1^s\partial_t\big(|\nabla\phi^t(x)|^2\big)dt.\]
Hence, we can set \(\Phi(x):=|\nabla\phi^1(x)|^2+2\int_0^\infty|\partial_t\nabla\phi^t(x)||\nabla\phi^t|dt\) and note that
\begin{align*}
    M[|\nabla \phi^s|^2](x)&\leq M\Big[|\nabla\phi^1|^2+\int_1^s|\partial_t|\nabla\phi^t|^2|dt\Big](x)
    \\
    &=M\Big[|\nabla\phi^1|^2+2\int_1^s|\partial_t\nabla\phi^t||\nabla\phi^t(x)|dt\Big](x)\leq M[\Phi](x).
\end{align*}
Since the maximal function is bounded from \(L^1\) to \(L^{1,\infty}\), we have for \(\kappa_0>0\)
\begin{align*}
    &|\{x\in 3\Delta; M[\Phi](x)>\kappa_0^2\}|\lesssim \frac{1}{\kappa_0^2}\int_{3\Delta}|\Phi(x)| dx 
    \\
    &\qquad\leq \frac{1}{\kappa_0^2}\Big(\int_{3\Delta}|\nabla\phi^1(x)|^2dx+2\int_0^\infty\int_{3\Delta}|\partial_t\nabla\phi^t(x)||\nabla\phi^t(x)|dx dt\Big)
    \\
    &\qquad \leq \frac{1}{\kappa_0^2}\Big(|\Delta|+2\int_0^\infty\Vert\partial_t\nabla\phi^t\Vert_{L^2(3\Delta)}\Vert\nabla\phi^t\Vert_{L^2(3\Delta)}dt\Big)
\end{align*}

With \eqref{equation:L2ofnablaphiBounded} and \eqref{equation:L2ofpartial_snablaphiBounded} we can bound this expression by
\begin{align*} 
    \frac{1}{\kappa_0^2}\Big(1+2\int_0^\infty\Vert\partial_t A(\cdot,t)\Vert_{L^\infty(3\Delta)} dt\Big)|\Delta|.
\end{align*}

Due to the \(L^1-L^\infty\) condition \eqref{cond:L1-Linfty} we can choose \(\kappa_0=\kappa_0(\gamma)\) sufficiently large to ensure
\[|\{x\in 3\Delta; M[\Phi](x)>\kappa_0\}|\lesssim \gamma |\Delta|.\]
Setting \(F:=\{x\in 3\Delta; M[\Phi](x)\leq\kappa_0\}\) and observing \(M[f]^2\leq M[|f|^2]^{1/2}\) finishes the proof.
\end{proof}

This lemma leads to the a pointwise bound on \(\theta\) which will be used in the proof of \refthm{thm:MainTheorem}.

\begin{lemma}\label{lemma:thetaPointwiseBoundOnGoodSetF}
Assume \eqref{cond:L1-Linfty}. Then for small \(\gamma>0\) there exists a large \(\kappa_0>0\) and \(F\subset 3\Delta\) such that
\[\theta_\eta(x,s)\leq C(n,\lambda_0,\Lambda_0)\kappa_0\eta s\qquad \textrm{for all }x\in F, s>0\]
where \(|3\Delta\setminus F|\leq \gamma |\Delta|\).
\end{lemma}

\begin{proof}
We take \(F\) from \reflemma{lemma:UniformBoundOnM[nablaphi^s]}. By
\refcor{cor:PointwiseBoundsofw_tAndrhoByM(nablaphi^s)} we have \(\partial_t w(y,(\eta t)^2,s)\leq C(n,\lambda_0,\Lambda_0) \eta M[\nabla\phi^s](y)\) for all \(y\in \mathbb{R}^n\). Hence we get from \reflemma{lemma:UniformBoundOnM[nablaphi^s]} for \(x\in F\)
\[\theta_\eta(x,s)=\int_0^s \partial_t w(x,(\eta t)^2;s)dt\lesssim \int_0^s \eta\kappa_0 dt\leq \kappa_0\eta s.\]
\end{proof}

The next lemma is an extension of \reflemma{lemma:thetaPointwiseBoundOnGoodSetF}.

\begin{lemma}\label{lemma:thetaPointwiseBoundOnOmega(F)}
Assume \eqref{cond:L1-Linfty} and let \(F\) be the good set from \reflemma{lemma:thetaPointwiseBoundOnGoodSetF} (or \reflemma{lemma:UniformBoundOnM[nablaphi^s]}). Then
\[\theta_\eta(x,s)\leq C(n,\lambda_0,\Lambda_0)\kappa_0\eta s\qquad \textrm{for all }(x,s)\in \tilde{\Omega}(F).\]
\end{lemma}

\begin{proof}
    Completely analogous to Lemma 8 in \cite{hofmann_dirichlet_2022}.
\end{proof}

The next lemma demonstrates that also the spatial gradient of \(\rho\) behaves well on the sawtooth region over this good set \(F\).

\begin{lemma}\label{lemma:UnifromBoundOnrholocally}
Let \(\alpha>0,\eta>0\) and \(F\) like in \reflemma{lemma:UniformBoundOnM[nablaphi^s]}. then we have for all \((x,s)\in \tilde{\Omega}_{\eta\alpha}(F)\) a \(P\in F\) such that
\[\fint_{\Delta(x,{\eta\alpha} s)}|\nabla\rho_\eta(y,s)|^2 dy \lesssim_{\alpha,\eta} M[\nabla\phi^s](P)\lesssim\kappa_0.\]
\end{lemma}

\begin{proof}
Fix \(0<\EPS< \frac{1}{100}\) small and let \(\psi\in C^\infty\) be a smooth cut-off function with \(\mathrm{supp}(\psi)\subset (1+\EPS)\Delta(x,{\eta\alpha} s)\) and \(\psi\equiv 1\) on \(\Delta(x,{\eta\alpha} s)\), while \(|\nabla\psi|\lesssim \frac{1}{{\EPS\eta\alpha} s}\). Let \[m(s):=\fint_{\Delta(x, (1+\EPS)\eta\alpha s)}\phi^s(x) dx=(\phi^s(\cdot))_{B_x(x, 2\eta\alpha s)}.\]
Then we get
\begin{align*}
    \fint_{\Delta(x,{\eta\alpha} s)}|\nabla\rho_\eta|^2 dy&\leq \fint_{(1+\EPS)\Delta(x,{\eta\alpha} s)}A\nabla\rho_\eta\cdot \nabla\rho_\eta \psi^2 dy
    \\
    &=\fint_{(1+\EPS)\Delta(x,{\eta\alpha} s)}A\nabla\rho_\eta\cdot \nabla((\rho_\eta-m) \psi^2) dy
    \\
    &\qquad + \fint_{(1+\EPS)\Delta(x,{\eta\alpha} s)}A\nabla\rho_\eta\cdot (\rho_\eta-m) \psi\nabla\psi dy \eqcolon I+J.
\end{align*}
Since \(s L^s \rho_\eta(x,s)=w_t(x,s)\) we have by \refcor{cor:PointwiseBoundsofw_tAndrhoByM(nablaphi^s)}
\[I=\fint_{(1+\EPS)\Delta(x,s)}w_t(y,s) \frac{(\rho_\eta(y,s)-m(s))\psi^2}{\eta\alpha s} dy\lesssim M[\nabla\phi^s](P)^2\]
for every \((P,s)\in \Gamma_{\eta\alpha}(x)\). Choosing \(P\in F\) leads with \reflemma{lemma:UniformBoundOnM[nablaphi^s]} to \(I\lesssim \kappa_0^2\).

\medskip

For \(J\) we can calculate for \(0<\sigma<1\)
\begin{align*}
J&\lesssim \Big(\fint_{(1+\EPS)\Delta(x,\eta\alpha s)}|\nabla\rho_\eta|^2\psi^2 dy\Big)^{1/2}\Big(\fint_{(1+\EPS)\Delta(x,\alpha\eta s)}\frac{|\rho_\eta-m|^2}{s^2} dy\Big)^{1/2}
\\
&\lesssim \sigma \fint_{(1+\EPS)\Delta(x,\alpha\eta s)}|\nabla\rho_\eta|^2\psi^2 dy+ M[\nabla\phi^s](P)^2
\\
&\leq \sigma \fint_{(1+\EPS)\Delta(x,\alpha\eta s)}|\nabla\rho_\eta|^2\psi^2 dy+ \kappa_0^2.
\end{align*}
Hiding the first term on the left side finishes the proof of the claim.
\end{proof}

\section{Area function bounds on the parts of the derivative of \(\theta\)}\label{section:areaFctBounds}

Recall that a boundary ball \(\Delta\) was fixed and the Hodge decomposition and hence also \(\rho\) and \(\theta\) were defined in dependence of the chosen boundary ball.

The following lemma collects all important area function estimates that we use in the proof of \refthm{thm:MainTheorem}.
 
\begin{lemma}\label{lemma:L^2estimatesForSquareFunctions}
Assuming \eqref{cond:L1-Linfty} we have the following estimates:
\begin{enumerate}[(a)]
    \item \(\int_{T(\Delta)}|\nabla \partial_s \rho_\eta(x,s)|^2 sdxds\lesssim |\Delta|,\)
    \item \(\int_{T(\Delta)}|L^s \rho_\eta(x,s)|^2s dxds\lesssim |\Delta|,\)
    \item \(\int_{T(\Delta)}|\partial_s A(x,s)\nabla\rho_\eta(x,s)|^2s dxds\lesssim |\Delta|,\)
    \item \(\int_{T(\Delta)}\frac{|\partial_s \theta_\eta(x,s)|^2}{s}dxds\lesssim |\Delta|,\)
    \item \(\int_{T(\Delta)}\frac{|\theta_\eta(x,s)|^2}{s^3}dxds\lesssim |\Delta|.\)
\end{enumerate}
\end{lemma}
The proof will appear at the end of this section. However, the idea is to split the functions the derivatives of \(\rho\) or \(\theta\) into its components and show these area function bounds for the components individually. Hence, let us discuss the \(L^2\) area function expression involving \(\partial_s\phi^s,w_t,w_s^{(1)}\) and \(w_s^{(2)}\) separately. 

\subsection{The partial derivative parts \(\partial_s\phi^s\) and \(w_s^{(1)}\)}

Instead of discussing \(\partial_s\phi^s\) and \(w_s^{(1)}\) separately, we are going to use that their difference
\[\partial_s\phi^s-w_s^{(1)}=\partial_s\phi^s-e^{-(\eta s)^2L^s}\partial_s\phi^s\]
gives us better properties.

\begin{lemma}\label{lemma:SqFctBoundsForpartial_sPhiAndw_s^1}
If we assume \eqref{condition:L^2Carlesontypecond}, the following area function bounds hold:
\begin{enumerate}[(i)]
    \item \(\Vert \mathcal{A}(\partial_s\phi^s-w_s^{(1)})\Vert_{L^2(\mathbb{R}^n)}^2=\Big\Vert\Big(\int_0^\infty\frac{|\partial_s\phi^s-w_s^{(1)}|^2}{s}ds\Big)^{1/2}\Big\Vert_{L^2(\mathbb{R}^n)}^2\lesssim |\Delta|, \textrm{ and}\)\label{lemma:w_s^1SqFctBound1}
    \item \begin{align*}
        &\Vert \mathcal{A}(\partial_s\phi^s-w_s^{(1)})\Vert_{L^2(\mathbb{R}^n)}^2 + \Vert \mathcal{A}(\partial_s\phi^s-w_s^{(1)})\Vert_{L^2(\mathbb{R}^n)}^2
        \\
        &=\Big\Vert\Big(\int_0^\infty\frac{|s\nabla\partial_s\phi^s|^2}{s}ds\Big)^{1/2}\Big\Vert_{L^2(\mathbb{R}^n)}^2+\Big\Vert\Big(\int_0^\infty\frac{|s\nabla w_s^{(1)}|^2}{s}ds\Big)^{1/2}\Big\Vert_{L^2(\mathbb{R}^n)}^2\lesssim |\Delta|.
    \end{align*}\label{lemma:w_s^1SqFctBound2}
\end{enumerate}
\end{lemma}

\begin{proof}
    For \eqref{lemma:w_s^1SqFctBound1} we use the idea that already appeared in the proof of \reflemma{lemma:thetaPointwiseBoundOnGoodSetF}, where we use \refprop{prop:KatoConjecture} to get 
    \begin{align*}
        (\partial_s\phi^s-e^{-(\eta s)^2L^s}\partial_s\phi^s)(x)&=\int_0^s\partial_t e^{-(\eta t)^2L^s}\partial_s \phi^sdt\lesssim \eta\int_0^s M[\nabla\partial_s\phi^s](x)dt
        \\
        &\lesssim M[\nabla\partial_s\phi^s](x)\eta s.
    \end{align*}
    Hence,
    \begin{align*}
        \Big\Vert\Big(\int_0^\infty\frac{|\partial_s\phi^s-w_s^{(1)}|^2}{s}ds\Big)^{1/2}\Big\Vert_{L^2(\mathbb{R}^n)}^2&\lesssim \int_0^\infty \int_{\mathbb{R}^n}|M[\nabla\partial_s\phi^s](x)|^2sdxds
        \\
        &\lesssim \int_0^\infty \Vert\nabla\partial_s\phi^s\Vert_{L^2(\mathbb{R}^n)}^2sds,
    \end{align*}
    
    and with \eqref{equation:L2ofpartial_snablaphiBounded} we get \(\int_0^\infty\Vert\partial_s A(\cdot,s)\Vert_\infty^2 |\Delta| s ds\).

    \hfill\\
    For \eqref{lemma:w_s^1SqFctBound2} we apply \reflemma{lemma:nablaSemigroupBoundedByNablaf} and \eqref{equation:L2ofpartial_snablaphiBounded} to obtain 
    \[\Vert\nabla w_s^{(1)}\Vert_{L^2(\mathbb{R}^n)}^2\lesssim \Vert \nabla\partial_s\phi^s\Vert_{L^2(\mathbb{R}^n)}^2\lesssim \Vert\partial_s A(\cdot,s)\Vert_\infty^2 |\Delta|,\] which implies
    \begin{align*}
        \Big\Vert\Big(\int_0^\infty\frac{|s\nabla\partial_s\phi^s|^2}{s}ds\Big)^{1/2}\Big\Vert_{L^2(\mathbb{R}^n)}^2&,\Big\Vert\Big(\int_0^\infty\frac{|s\nabla w_s^{(1)}|^2}{s}ds\Big)^{1/2}\Big\Vert_{L^2(\mathbb{R}^n)}^2
        \\
        &\qquad\qquad\lesssim \int_0^\infty\Vert\partial_s A(\cdot,s)\Vert_\infty^2 |\Delta| s ds.
    \end{align*}
    
\end{proof}

\subsection{The partial derivative part \(w_s^{(2)}\)}

\begin{lemma}\label{lemma:w_s^2/nablaw_s^2spatialL^2Bound}
    It holds that
    \begin{enumerate}[(i)]
        \item \label{lemma:w_s^2spatialL^2Bound} \(\Vert w_s^{(2)}(\cdot,s)\Vert_{L^2(\mathbb{R}^n)}\lesssim \Vert \partial_s A(\cdot,s)\Vert_{L^\infty(\mathbb{R}^n)}|\Delta|^{1/2}s,\) and hence
        \item \(\Vert\mathcal{A}(w_s^{(2)})\Vert_{L^2(\mathbb{R}^n)}^2=\Big\Vert\Big(\int_0^\infty\frac{|w_s^{(2)}|^2}{s}ds\Big)^{1/2}\Big\Vert_{L^2(\mathbb{R}^n)}^2\lesssim |\Delta|.\)\label{lemma:w_s^2SqFctBound1}
    \end{enumerate}
\end{lemma}

\begin{proof}
    We begin by proving \eqref{lemma:w_s^2spatialL^2Bound}. By Minkowski's inequality, we obtain
    \begin{align*}
        &\Big\Vert\int_0^{s} \eta^2\tau e^{-\eta^2(s^2-\tau^2)L^s}\Div(\partial_s A(\cdot,s)\nabla w(x,\eta^2\tau^2;s))d\tau\Big\Vert_{L^2(\mathbb{R}^n)}
        \\
        &\qquad =\int_0^{s} \Vert\eta^2\tau e^{-\eta^2(s^2-\tau^2)L^s}\Div(\partial_s A(\cdot,s)\nabla w(x,\eta^2\tau^2;s))\Vert_{L^2(\mathbb{R}^n)} d\tau.
    \end{align*}
    By Kato conjecture \refprop{prop:KatoConjecture}, \refprop{cor:NablaL^-1/2Bounded} and \refprop{Proposition11} we can observe 
    \begin{align*}
        &\Vert\eta^2\tau e^{-\eta^2(s^2-\tau^2)L^s}\Div(\partial_s \AP(\cdot,s)\nabla w(x,\eta^2\tau^2;s))\Vert_{L^2(\mathbb{R}^n)}
        \\
        &\lesssim\Vert\eta^2\tau \nabla e^{-\eta^2(s^2-\tau^2)L^s}(L^s)^{-1/2}\Div(\partial_s \AP(\cdot,s)\nabla w(x,\eta^2\tau^2;s))\Vert_{L^2(\mathbb{R}^n)}
        \\
        &\lesssim\frac{\eta^2\tau}{\sqrt{s^2-\tau^2}}\Vert M[(L^s)^{-1/2}\Div(\partial_s \AP(\cdot,s)\nabla w(x,\eta^2\tau^2;s))]\Vert_{L^2(\mathbb{R}^n)}
        \\
        &\lesssim\frac{\eta^2\tau}{\sqrt{s^2-\tau^2}}\Vert\partial_s \AP(\cdot,s)\nabla w(x,\eta^2\tau^2;s))\Vert_{L^2(\mathbb{R}^n)}
        \\
        &\lesssim\frac{\eta^2\tau}{\sqrt{s^2-\tau^2}}\Vert\partial_s \AP(\cdot,s)\Vert_{L^\infty(\mathbb{R}^n)} \Vert\nabla\phi^s\Vert_{L^2(\mathbb{R}^n)}
        \\
        &\lesssim\frac{\eta^2\tau}{\sqrt{s^2-\tau^2}}\Vert\partial_s \AP(\cdot,s)\Vert_{L^\infty(\mathbb{R}^n)} |\Delta|^{1/2}.
    \end{align*}
    Integrating \(\frac{\tau}{\sqrt{s^2-\tau^2}}\) from \(0\) to \(s\) gives the factor \(s\), whence \eqref{lemma:w_s^2spatialL^2Bound} follows.

    Now \eqref{lemma:w_s^2SqFctBound1} is an easy consequence of \eqref{lemma:w_s^2spatialL^2Bound} and \eqref{condition:L^2Carlesontypecond}, since
    \[\Big\Vert\Big(\int_0^\infty\frac{|w_s^{(2)}|^2}{s}ds\Big)^{1/2}\Big\Vert_{L^2(\mathbb{R}^n)}^2\lesssim \int_0^\infty\Vert \partial_s A(\cdot,s)\Vert_{L^\infty(\mathbb{R}^n)}^2|\Delta|s ds\lesssim |\Delta|.\]

\end{proof}

Next we show the area function bound for \(\nabla w_s^{(2)}\). For that recall that
\[v_2(x,t^2;s)=\int_0^t2\tau e^{-(t^2-\tau^2)L^s}\mathrm{div}(\partial_s\AP(x,s)\nabla e^{-\tau^2L^s}\phi^s)(x)d\tau,\]
and \(w_s^{(2)}(x,s)=v_2(x,\eta^2s^2,s)\).

\begin{lemma}\label{lemma:SqFctBoundsForw_s^2}
Let \(Q\subset \partial\Omega\) be a boundary cube of size s, i.e. \(l(Q)\approx s\), and the index \(i\in \mathbb{Z}\) such that \(s\in [2^{-i},2^{-i+1})\). If we assume \eqref{condition:L^2Carlesontypecond}, we have the following local bound involving \(w_s^{(2)}\)
\begin{align}
    \int_Q |\nabla w_s^{(2)}(x,s)|^2dx&\lesssim_\eta \frac{1}{s^2}\fint_{2^{-i-1}}^{2^{-i+2}}\int_{2Q}|v_2(x,\eta^2k^2;s)|^2dkdx\nonumber
    \\
    &\qquad+ \Vert\partial_s A(\cdot,s)\Vert_{L^\infty(3Q)}^2\fint_{2^{-i-3}}^{2^{-i+4}}\int_{2Q} |\nabla e^{-\eta^2t^2L^s}\phi^s|^2dxdt \nonumber
    \\
    &\qquad+ \Vert\partial_s A(\cdot,s)\Vert_{L^\infty(3Q)}^2 s^2 \fint_{2^{-i-3}}^{2^{-i+4}}\int_{2Q}|\nabla\partial_ke^{-\eta^2k^2L^s}\phi^s|^2dxdk.\label{lemma:localEstimateNablaw_s^2}
\end{align}
As a consequence we have 
\begin{align}
    \Vert \nabla w_s^{(2)}\Vert_{L^2(\mathbb{R}^n)}\lesssim \Vert \partial_s A(\cdot, s)\Vert_{L^\infty(\mathbb{R}^n)}|\Delta|^{1/2}\label{lemma:Nablaw_s^2SpatialBound},
\end{align}
and the area function bound
\begin{align}
    \Vert\mathcal{A}(s|\nabla w_s^{(2)}|)\Vert_{L^2(\mathbb{R}^n)}^2=\Big\Vert\Big(\int_0^\infty\frac{|s\nabla w_s^{(2)}|^2}{s}ds\Big)^{1/2}\Big\Vert_{L^2(\mathbb{R}^n)}^2\lesssim |\Delta|.\label{lemma:w_s^2SqFctBound2}
\end{align}

\end{lemma}

\begin{proof}
In this proof we are going to abbreviate notation and set \(v(x,t;s)=v_2(x,\eta^2t^2;s)\). Let \(i\in \mathbb{Z}\) such that \(s\in [2^{-i},2^{-i+1})\). Here \(s\) will remain fixed for the whole argument of \eqref{lemma:localEstimateNablaw_s^2}. The parameter \(\eta\) does not play any role and in \(\lesssim\) suppressed constants might depend on \(\eta\).
We begin with the use of the Fundamental Theorem of Calculus and Jensen's inequality to obtain
\begin{align*}
    \int_Q |\nabla w_s^{(2)}(x,s)|^2 dx &= \int_Q |\nabla v(x,s;s)|^2 dx  = \fint_{2^{-i}}^{2^{-i+1}}\int_Q |\nabla v(x,s;s)|^2dtdx
    \\
    &=\fint_{2^{-i}}^{2^{-i+1}}\int_Q \Big|\nabla v(x,t;s) + \int_t^s\nabla\partial_k v(x,k;s)dk\Big|^2dtdx
    \\
    &\lesssim\fint_{2^{-i}}^{2^{-i+1}}\int_Q |\nabla v(x,t;s)|^2 dxdt
    \\
    &\qquad + \fint_{2^{-i}}^{2^{-i+1}}|s-t|\int_Q\int_t^s|\nabla \partial_k v(x,k;s)|^2dkdtdx
    \\
    &\lesssim\fint_{2^{-i}}^{2^{-i+1}}\int_Q |\nabla v(x,t;s)|^2 dxdt
    \\
    &\qquad  + s\int_Q\fint_{2^{-i}}^{2^{-i+1}}|\nabla\partial_k v(x,k;s)|^2dkdx=I_1+sI_2.
\end{align*}
For \(I_1\) we get by Caccioppoli type inequality \reflemma{Lemma:CacciopolliTypeInequalities}
\[I_1\lesssim \frac{1}{s^2}\fint_{2^{-i-1}}^{2^{-i+2}}\int_{2Q}|v(x,t;s)|^2dxdt + \Vert\partial_s A(\cdot,s)\Vert_{L^\infty(2Q)}^2\fint_{2^{-i-1}}^{2^{-i+2}}\int_{2Q}|\nabla e^{-\eta^2t^2L^s}\phi^s|^2dxdt.\]

For \(I_2\) we apply the other conclusion of \reflemma{Lemma:CacciopolliTypeInequalities} to get
\begin{align*}
    I_2&\lesssim\frac{1}{s^2}\Vert\partial_s A(\cdot,s)\Vert_{L^\infty(2Q)}^2\int_{2^{-i-1}}^{2^{-i+2}}\int_{2Q}|\nabla e^{-\eta^2t^2L^s}\phi^s|^2dxdt
    \\
    &\qquad + \Vert\partial_s A(\cdot,s)\Vert_{L^\infty(2Q)}^2 \int_{2^{-i-1}}^{2^{-i+2}}\int_{\frac{5}{4}Q} |\nabla \partial_t e^{-\eta^2t^2L^s}\phi^s|^2dxdt
    \\
    &\qquad+ \frac{1}{s^2}\int_{2^{-i-1}}^{2^{-i+2}}\int_{\frac{5}{4}Q} |\partial_k v(x,k;s)|^2 dxdk
\end{align*}

Further, we will have to apply \reflemma{lemma:LEMMA5}. For that, let us first introduce the parameters \(\frac{5}{4}\leq\EPS_1<\EPS_2\leq\frac{3}{2}\) and  \(d:=(\EPS_2-\EPS_1)/2\). Depending on \(\EPS_1\) and \(\EPS_2\), we can introduce enlargements of 
\[\frac{5}{4}Q\subset Q_{1}\subset \tilde{Q}\subset Q_{2}\subset\frac{3}{2}Q\]
by \(Q_{1}:=\EPS_1Q,\tilde{Q}:=(\EPS_1+d)Q\), and \(Q_2:=\EPS_2Q\). Corresponding the these cubes in \(x\), we also define the corresponding enlargements of 
\[Q^*:=\frac{5}{4}Q\times [2^{-i-1},2^{-i+2}]\subset Q_1^*\subset \tilde{Q}^*\subset Q_2^*\subset \frac{3}{2}Q\times [2^{-i-2},2^{-i+3}]=:\hat{Q}^*\]
in the full \((x,t)\)-space by 
\begin{align*}
&Q_1^*:=Q_1\times[\frac{1}{\EPS_1}2^{-i-1},\EPS_12^{-i+2}],\qquad \tilde{Q}^*:=\tilde{Q}\times[\frac{1}{\EPS_1+d}2^{-i-1},(\EPS_2+d) 2^{-i+2}], \textrm{ and}
\\
& Q_2^*:=Q_2\times[\frac{1}{\EPS_2}2^{-i-1},\EPS_2 2^{-i+2}].
\end{align*}
Now we can also introduce a smooth cut-off function \(\psi\) with \(\psi\equiv 1\) on \(Q_1^*\) and \(\psi\equiv 0\) on \(\Omega\setminus \tilde{Q}^*\) with \(|\nabla \psi|\lesssim \frac{1}{d s}\).

Then we have by the PDE that \(v\) satisfies (see proof of \reflemma{Lemma:CacciopolliTypeInequalities})

\begin{align*}
    &\int_{Q_1^*}|\partial_k v(x,k;s)|^2dkdx=\int_{\tilde{Q}^*}\partial_k v\big(2kL^s v + k\Div(\partial_s A\nabla e^{-\eta^2k^2L^s}\phi^s)\big)\psi^2 dkdx 
    \\
    &\qquad=\int_{\tilde{Q}^*}\nabla \partial_k v \cdot\AP\nabla v\psi^2kdxdk + \int_{\tilde{Q}^*}\nabla \partial_k v \cdot \partial_s A\nabla e^{-\eta^2k^2L^s}\phi^s\psi^2kdxdk
    \\
    &\qquad\quad + \int_{\tilde{Q}^*} \partial_k v \AP\nabla v\cdot\nabla\psi^2kdxdk + \int_{\tilde{Q}^*} \partial_k v \partial_s A\nabla e^{-\eta^2k^2L^s}\phi^s\cdot\nabla\psi^2kdxdk
    \\
    &\qquad=: II_1+II_2+II_3+II_4.
\end{align*}
We estimate each integral separately using the same techniques consisting of Hölder's inequality and Caccioppoli type inequality \reflemma{Lemma:CacciopolliTypeInequalities} while also hiding terms with the factor \(0<\sigma<1\) on the left hand side. Although this is straight forward, we will be more precise and show these calculations for all terms in detail. First, for \(II_1\)
\begin{align*}
    II_1&\lesssim s\Big(\int_{\tilde{Q}^*}|\nabla v|^2dxdk\Big)^{1/2}\Big(\int_{\tilde{Q}^*}|\nabla\partial_k v|^2dxdk\Big)^{1/2}
    \\
    &\lesssim s\Big(\frac{1}{d^2s^2}\int_{Q_2^*}|v(x,k;s)|^2dxdk + \Vert\partial_s A(\cdot,s)\Vert_{L^\infty}^2\int_{Q_2^*}|\nabla e^{-\eta^2k^2L^s}\phi^s|^2dxdk\Big)^{1/2}
    \\
    &\quad\cdot \Big(\frac{1}{d^2s^2}\int_{Q_2^*}|\partial_k v(x,k;s)|^2dxdk + \frac{1}{d^2s^2}\Vert\partial_s A(\cdot,s)\Vert_{L^\infty}^2\int_{Q_2^*}|\nabla e^{-\eta^2k^2L^s}\phi^s|^2dxdk 
    \\
    &\qquad + \Vert\partial_s A(\cdot,s)\Vert_{L^\infty}^2\int_{Q_2^*}|\nabla\partial_k e^{-\eta^2k^2L^s}\phi^s|^2dxdk\Big)^{1/2}
    \\
    &\lesssim \sigma\int_{Q_2^*}|\partial_k v(x,k;s)|^2\eta^2dxdk + \frac{1}{d^4s^2}\int_{Q_2^*}|v(x,k;s)|^2dxdk
    \\
    &\qquad + \Vert\partial_s A(\cdot,s)\Vert_{L^\infty}^2\int_{Q_2^*}|\nabla e^{-\eta^2k^2L^s}\phi^s|^2dxdk
    \\
    &\qquad + \Vert\partial_s A(\cdot,s)\Vert_{L^\infty}^2 d^2s^2 \int_{Q_2^*}|\nabla\partial_k e^{-\eta^2k^2L^s}\phi^s|^2dxdk.
\end{align*}

Furthermore, we have
\begin{align*}
    II_2&\lesssim \sigma d^2s^2\int_{\tilde{Q}^*}|\nabla \partial_k v|^2dxdk + \frac{1}{d^2}\Vert\partial_s A(\cdot,s)\Vert_{L^\infty}^2\int_{\tilde{Q}^*}|\nabla e^{-\eta^2k^2L^s}\phi^s|^2dxdk
    \\
    &\lesssim \sigma\int_{Q_2^*}|\partial_k v(x,k;s)|^2dxdk + \big(\frac{1}{d^2}+1\big)\Vert\partial_s A(\cdot,s)\Vert_{L^\infty}^2\int_{Q_2^*}|\nabla e^{-\eta^2 k^2L^s}\phi^s|^2dxdk
    \\
    &\qquad + \Vert\partial_s A(\cdot,s)\Vert_{L^\infty}^2 d^2s^2 \int_{Q_2^*}|\nabla\partial_ke^{-\eta^2 k^2L^s}\phi^s|^2dxdk.
\end{align*}
For the last two integrals, we have
\begin{align*}
    II_3&\lesssim \sigma \int_{\tilde{Q}^*}|\partial_k v|^2dxdk + \frac{C_\sigma}{d^2}\int_{\tilde{Q}^*}|\nabla v|^2dxdk
    \\
    &\lesssim \sigma \int_{Q_2^*}|\partial_k v|^2dxdk + \frac{1}{d^4s^2}\int_{Q_2^*}|v(x,k;s)|^2dxdk
    \\
    &\qquad+ \frac{1}{d^2}\Vert\partial_s A(\cdot,s)\Vert_{L^\infty}^2\int_{Q_2^*}|\nabla e^{-\eta^2k^2L^s}\phi^s|^2dxdk,
\end{align*}
and
\begin{align*}
    II_4\lesssim \sigma \int_{\tilde{Q}^*}|\partial_k v|^2dxdk + \frac{1}{d^2}\Vert\partial_s A(\cdot,s)\Vert_{L^\infty}^2\int_{\tilde{Q}^*}|\nabla e^{-\eta^2k^2L^s}\phi^s|^2dxdk.
\end{align*}

Putting all these integrals together we obtain 
\begin{align*}
    \int_{Q_1^*}|\partial_k v|^2dxdk&\lesssim \sigma\int_{Q_2^*}|\partial_k v(x,k;s)|^2dxdk + \frac{1}{d^4s^2}\int_{Q_2^*}|v(x,k;s)|^2dxdk
    \\
    &\qquad + \big(1+\frac{1}{d^2}\big)\Vert\partial_s A(\cdot,s)\Vert_{L^\infty}^2\int_{Q_2^*}|\nabla e^{-\eta^2k^2L^s}\phi^s|^2dxdk
    \\
    &\qquad + \Vert\partial_s A(\cdot,s)\Vert_{L^\infty}^2 d^2s^2 \int_{Q_2^*}|\nabla\partial_k e^{-\eta^2k^2L^s}\phi^s|^2dxdk.
\end{align*}

With \reflemma{lemma:LEMMA5} this leads to
\begin{align*}
    I_2&\lesssim\frac{1}{s^2}\Vert\partial_s A(\cdot,s)\Vert_{L^\infty}^2\int_{\hat{Q}^*}|\nabla\partial_k e^{-\eta^2k^2L^s}\phi^s|^2dxdk
    \\
    &\qquad + \Vert\partial_s A(\cdot,s)\Vert_{L^\infty}^2 \int_{\hat{Q}^*} |\nabla \partial_k e^{-\eta^2k^2L^s}\phi^s|^2dxdk
    + \frac{1}{s^2}\int_{\hat{Q}^*} |\partial_k v(x,k;s)|^2 dxdk
    \\
    &\lesssim \frac{1}{s^4}\int_{\hat{Q}^*}|v(x,k;s)|^2dkdx
     + \frac{1}{s^2}\Vert\partial_s A(\cdot,s)\Vert_{L^\infty}^2 \int_{\hat{Q}^*} |\nabla e^{-\eta^2k^2L^s}\phi^s|^2dxdk 
    \\
    &\qquad+ \Vert\partial_s A(\cdot,s)\Vert_{L^\infty}^2 \int_{\hat{Q}^*}|\nabla\partial_ke^{-\eta^2k^2L^s}\phi^s|^2dxdk,
\end{align*}

and hence
\begin{align*}
    \int_Q |\nabla w_s^{(2)}(x,s)|^2dx&=I_1+sI_2
    \\
    &\lesssim \frac{1}{s^3}\int_{2^{-i-3}}^{2^{-i+4}}\int_{2Q}|v(x,k;s)|^2dkdx
    \\
    &\qquad + \frac{1}{s}\Vert\partial_s A(\cdot,s)\Vert_{L^\infty}^2\int_{2^{-i-3}}^{2^{-i+4}}\int_{2Q} |\nabla e^{-\eta^2k^2L^s}\phi^s|^2dxdk 
    \\
    &\qquad+  s\Vert\partial_s A(\cdot,s)\Vert_{L^\infty}^2\int_{2^{-i-3}}^{2^{-i+4}}\int_{2Q}|\nabla\partial_ke^{-\eta^2k^2L^s}\phi^s|^2dxdk.
\end{align*}

Now, we can show the global estimates (\ref{lemma:Nablaw_s^2SpatialBound}) and (\ref{lemma:w_s^2SqFctBound2}). Recall that \(\mathcal{D}_i(\mathbb{R}^n)\) denotes a collection of cubes \(Q\) of size \(l(Q)\approx 2^{-i}\approx s\) that cover \(\mathbb{R}^n\) with finite overlap. First for (\ref{lemma:Nablaw_s^2SpatialBound}) we have 
\begin{align*}
    &\int_{\mathbb{R}^n} |\nabla w_s^{(2)}(x,s)|^2 sdxds\leq \sum_{Q\in \mathcal{D}_i(\mathbb{R}^n)} |\nabla w_s^{(2)}(x,s)|^2 sdxds
    \\
    &\qquad\lesssim \sum_{Q\in \mathcal{D}_i(\mathbb{R}^n)}\Big(\frac{1}{s^2}\fint_{2^{-i-3}}^{2^{-i+4}}\int_{2Q}|v(x,k;s)|^2dkdx
    \\
    &\qquad\qquad + \Vert\partial_s A(\cdot,s)\Vert_{L^\infty}^2\int_{2^{-i-3}}^{2^{-i+4}}\int_{2Q} |\nabla e^{-\eta^2k^2L^s}\phi^s|^2dxdk 
    \\
    &\qquad\qquad+  s^2\Vert\partial_s A(\cdot,s)\Vert_{L^\infty}^2\int_{2^{-i-3}}^{2^{-i+4}}\int_{2Q}|\nabla\partial_ke^{-\eta^2k^2L^s}\phi^s|^2dxdk\Big)
    \\
    &\qquad\lesssim \frac{1}{s^2}\fint_{2^{-i-3}}^{2^{-i+4}}\int_{\mathbb{R}^n}|v(x,k;s)|^2dkdx
    \\
    &\qquad\qquad + \Vert\partial_s A(\cdot,s)\Vert_{L^\infty}^2\fint_{2^{-i-3}}^{2^{-i+4}}\int_{\mathbb{R}^n} |\nabla e^{-\eta^2k^2L^s}\phi^s|^2dxdk 
    \\
    &\qquad\qquad+  s^2\Vert\partial_s A(\cdot,s)\Vert_{L^\infty}^2\fint_{2^{-i-3}}^{2^{-i+4}}\int_{\mathbb{R}^n}|\nabla\partial_ke^{-\eta^2k^2L^s}\phi^s|^2dxdk.
\end{align*}

We note here that due to \reflemma{lemma:nablaSemigroupBoundedByNablaf} and (\ref{equation:L2ofnablaphiBounded})
\[\Vert\partial_s A(\cdot,s)\Vert_{L^\infty(3\Delta)}^2 \int_{2^{-i-2}}^{2^{-i+3}}\int_{\mathbb{R}^n}|\nabla e^{-\eta^2t^2L^s}\phi^s|^2dxdk\lesssim \Vert\partial_s A(\cdot,s)\Vert_{L^\infty(3\Delta)}^2|\Delta|,\]
and
\[\Vert\partial_s A(\cdot,s)\Vert_{L^\infty(3\Delta)}^2 s^2 \fint_{2^{-i-3}}^{2^{-i+4}}\int_{\mathbb{R}^n}|\nabla\partial_t e^{-\eta^2t^2L^s}\phi^s|^2dxdk\lesssim \Vert\partial_s A(\cdot,s)\Vert_{L^\infty(3\Delta)}^2|\Delta|.\]

Since we can also prove 
\[\int_{\mathbb{R}^n}|v(x,t;s)|^2dx\lesssim \Vert\partial_s A(\cdot,s)\Vert_{L^\infty(3\Delta)}^2|\Delta|s^2 \]
analogously to \eqref{lemma:w_s^2spatialL^2Bound} in \reflemma{lemma:w_s^2/nablaw_s^2spatialL^2Bound}, we obtain 

\begin{align*}
    \int_{3\Delta} |\nabla w_s^{(2)}(x,s)|^2 sdxds\lesssim \Vert\partial_s A(\cdot,s)\Vert_{L^\infty(\mathbb{R}^n)}^2|\Delta|,
\end{align*}
which yields (\ref{lemma:Nablaw_s^2SpatialBound}). At last, (\ref{lemma:w_s^2SqFctBound2}) is a direct consequence of (\ref{lemma:Nablaw_s^2SpatialBound}). To see this we note with (\ref{condition:L^2Carlesontypecond}) 
\begin{align*}
    \int_0^\infty \int_{\mathbb{R}^n} |\nabla w_s^{(2)}(x,s)|^2 sdxds \lesssim |\Delta|\int_0^\infty \Vert\partial_s A(\cdot,s)\Vert_{L^\infty(\mathbb{R}^n)}^2s ds\lesssim |\Delta|.
\end{align*}

\end{proof}

\subsection{The partial derivative part \(w_t\)}

\begin{lemma}\label{lemma:SqFctBoundsForw_t}
If we assume \eqref{cond:L1-Linfty}, the following area function bounds involving \(w_t\) hold
\begin{enumerate}[(i)]
    \item \[\Vert \mathcal{A}(w_t)\Vert_{L^2(\mathbb{R}^n)}^2=\Big\Vert\Big(\int_0^\infty\frac{|sL^se^{-\eta^2s^2L^s}\phi^s|^2}{s}ds\Big)^{1/2}\Big\Vert_{L^2(\mathbb{R}^n)}^2\lesssim |\Delta|,\]\label{lemma:w_tSqFctBound1}
    \item \[\Vert \mathcal{A}(s\nabla w_t)\Vert_{L^2(\mathbb{R}^n)}^2=\Big\Vert\Big(\int_0^\infty\frac{|s\nabla w_t|^2}{s}ds\Big)^{1/2}\Big\Vert_{L^2(\mathbb{R}^n)}^2\lesssim |\Delta|,\textrm{ and}\label{lemma:w_tSqFctBound2}\]
    \item \[\Vert \mathcal{A}(s^2L^sw_t)\Vert_{L^2(\mathbb{R}^n)}^2=\Big\Vert\Big(\int_0^\infty\frac{|s^2L^sw_t|^2}{s}ds\Big)^{1/2}\Big\Vert_{L^2(\mathbb{R}^n)}^2\lesssim |\Delta|.\label{lemma:w_tSqFctBound3}\]
\end{enumerate}
\end{lemma}

\begin{proof}

Before we commence with the proof, we point out that all the implicit constants might depend on \(\eta\). In fact, without loss of generality we can set \(\eta=1\) since it will only change the involved constants but has otherwise no influence on the argument. Furthermore we observe the following using only the kernel bounds \eqref{eq:kernelbounds} of the operators \(e^{-tL},\partial_te^{-tL}\): Let \(f\) be a function with \(\mathrm{supp}(f)\subset E\) and let \(E, F\subset\mathbb{R}^n\) be two disjoint sets with \(d:=\mathrm{dist}(E,F)>0\). We call \(\tilde{F}:=F+B(0,d/2)\) an enlargement of \(F\) and we choose a cut-off function \(\psi\in C_0(\tilde{F})\) such that \(\psi\equiv 1\) on \(F\) and \(|\nabla \psi|\lesssim \frac{1}{r}\leq \frac{1}{2d}\). Then
\begin{align*}
    \Vert \nabla e^{-s^2L^s}f\Vert_{L^2(F)}^2&\lesssim \int_{\tilde{F}} A(x,s)\nabla e^{-s^2L^s}f(x)\cdot \nabla e^{-s^2L^s}f(x)\psi^2(x) dx
    \\
    &\lesssim \int_{\tilde{F}} L^s e^{-s^2L^s}f(x) e^{-s^2L^s}f(x)\psi^2(x)
    \\
    &\qquad + 2\psi(x)\nabla e^{-s^2L^s}f(x)\cdot \nabla \psi(x)  e^{-s^2L^s}f(x) dx
    \\
    &\lesssim \Vert L^s e^{-s^2L^s}f\Vert_{L^2(\tilde{F})} \Vert e^{-s^2L^s}f\Vert_{L^2(\tilde{F})}
    \\
    &\qquad + \sigma \Vert \psi\nabla e^{-s^2L^s}f \Vert^2_{L^2(\mathbb{R}^n)} + \frac{1}{\sigma}\Vert \frac{1}{r}e^{-s^2L^s}f\Vert_{L^2(\tilde{F})}^2.
\end{align*}
We can hide the third term on the left hand side and use the kernel estimates \eqref{eq:kernelbounds} to observe for \(x\in \tilde{F}\)
\begin{align*}
    e^{-s^2L^s}f(x)=\int_{E}K_{s^2}(x,y)f(y)dy\lesssim \frac{1}{s^n}e^{-c\frac{r^2}{s^2}}\int_{E}f(y)dy=\frac{1}{s^n}e^{-c\frac{r^2}{s^2}}\Vert f\Vert_{L^1}
\end{align*}
and similarly
\[L^s e^{-s^2L^s}(f)(x)\lesssim \frac{1}{s^{n+2}}e^{-c\frac{r^2}{s^2}}\Vert f\Vert_{L^1},\]
whence in total
\begin{align}
    \Vert \nabla e^{-s^2L^s}(f)\Vert_{L^2(F)}^2\lesssim \big(\frac{1}{s^{2n+2}}+\frac{1}{s^{2n}r^2}\big)e^{-c\frac{r^2}{s^2}}\Vert f\Vert_{L^1}^2\sigma(F)\label{observation:KernelIntoL1}.
\end{align}
\\

Now, let us begin with proving \eqref{lemma:w_tSqFctBound1}. By ellipticity of \(A\) and integration by parts we have
\begin{align*}
    &\int_\Omega\frac{|sL^se^{-s^2L^s}\phi^s|^2}{s}dxds 
    \\
    &=-\int_0^\infty\int_{\mathbb{R}^n}\AP(x,s)\nabla e^{-s^2 L^s}\phi^s(x) \cdot \nabla L^te^{-s^2 L^s}\phi^s(x) s dx ds
    \\
    & \lesssim\int_0^\infty\frac{1}{s}\Vert\nabla e^{-s^2 L^s}\phi^s\Vert_{L^2(\mathbb{R}^n)}^2 +  s^3\Vert\nabla L^se^{-s^2 L^s}\phi^s\Vert_{L^2(\mathbb{R}^n)}^2  dx ds
    \\
    &\lesssim\int_0^\infty C_\sigma\frac{1}{s}\Vert\nabla e^{-s^2 L^s}\phi^s\Vert_{L^2(\mathbb{R}^n)}^2 +  \sigma s\Vert\nabla e^{-s^2 L^s}(sL^s\phi^s)\Vert_{L^2(\mathbb{R}^n)}^2  dx ds
    \\
    & =:\int_0^\infty (C_\sigma I+\sigma II) ds.
\end{align*}
Here \(\sigma\) is a small constant which later will allow us to hide the integral \(II_3\) appearing in the estimate of \(II\) on the left hand side.
\smallskip

We start with handling \(I\). Let us fix a small \(a>0\) and consider \(\int_a^\infty I ds\). For this choice of \(a\) there exists a scale \(k\) with \(2^{-k}\approx a\) and the collection \(\mathcal{D}_k\) consists of boundary balls \(Q\) with \(l(Q)\approx 2^{-k}\approx a\) that cover \(\partial\Omega\) in such a way that the collection of \(2Q\) have finite overlap, i.e. \(|\sum_{Q\in \mathcal{D}_k} \chi_{Q}|\leq N\) for some \(N\in \mathbb{N}\). Then for \(s\geq a\) we have
\begin{align*}
s^2I&=\sum_{Q\in\mathcal{D}_k} s\Vert\nabla e^{-s^2 L^s}\phi^s\Vert_{L^2(Q)}^2=\sum_{Q\in\mathcal{D}_k}s\Vert\nabla e^{-s^2 L^s}(\phi^s-(\phi^s)_{2Q})\Vert_{L^2(Q)}^2
\\
&=\sum_{Q\in\mathcal{D}_k}s\Vert\nabla e^{-s^2 L^s}(\chi_{2Q}(\phi^s-(\phi^s)_{2Q}))\Vert_{L^2(Q)}^2
\\
&\qquad + \sum_{Q\in\mathcal{D}_k} s\Vert\nabla e^{-s^2 L^s}(\chi_{\mathbb{R}^n\setminus 2Q}(\phi^s-(\phi^s)_{2Q}))\Vert_{L^2(Q)}^2:=J+K.
\end{align*}

By observation \eqref{observation:KernelIntoL1} and Poincar\'{e} inequality we obtain for the second term
\begin{align*}
     K\lesssim &\sum_{Q\in\mathcal{D}_k}\sum_{l\geq 1} s\Vert\nabla e^{-s^2 L^s}(\chi_{2^{l+1}Q\setminus 2^lQ}(\phi^s-(\phi^s)_{2Q}))\Vert_{L^2(Q)}^2
     \\
     &\lesssim \sum_{Q\in\mathcal{D}_k}\sum_{l\geq 1}\Big(\frac{a^n}{s^{2n+1}}+\frac{a^{n-2}}{s^{2n-1}}\Big)e^{-c\frac{2^{2l}a^2}{s^2}}\Vert \phi^s-(\phi^s)_{2Q}\Vert_{L^1(2^{l+1}Q\setminus 2^lQ)}^2
     \\
     &\lesssim \sum_{Q\in\mathcal{D}_k}\sum_{l\geq 1}\frac{a^{n-2}}{s^{2n-1}}e^{-c\frac{2^{2l}a^2}{s^2}}\Big(\Vert \phi^s-(\phi^s)_{2^{l+1}Q}\Vert_{L^1(2^{l+1}Q)}^2
     \\
     &\qquad\qquad + \sum_{m=1}^{l} \Vert (\phi^s)_{2^{m+1}Q}-(\phi^s)_{2^{m}Q}\Vert_{L^1(2^{l+1}Q)}^2 \Big)
     \\
     &\lesssim \sum_{Q\in\mathcal{D}_k}\sum_{l\geq 1}\frac{a^{n-2}}{s^{2n-1}}e^{-c\frac{2^{2l}a^2}{s^2}}\Big(2^{2(n+1)(l+1)}a^{2(n+1)}\inf_{x\in Q}M[\nabla\phi^s]^2(x)
     \\
     &\qquad\qquad+ \sum_{m=1}^{l} 2^{2m+2n(l+1)}a^{2(n+1)}\inf_{x\in Q}M[\nabla\phi^s]^2(x) \Big)
     \\
     &\lesssim \sum_{Q\in\mathcal{D}_k}\sum_{l\geq 1}\frac{(l+1)2^{2(n+1)(l+1)}a^{2(n-1)}}{s^{2n-1}}e^{-c\frac{2^{2l}a^2}{s^{2}}}\int_{Q}M[\nabla\phi^s]^2(x)dx
     \\
     &\lesssim \sum_{l\geq 1}\frac{(l+1)2^{2(n+1)(l+1)}a^{2(n-1)}}{s^{2n-1}}e^{-c\frac{2^{2l}a^2}{s^2}}\Vert M[\nabla\phi^s]\Vert_{L^2(\mathbb{R}^n)}^2.
\end{align*}
We can bound \((l+1)\leq 2^l\) and hence consider the sum over \(l\) as Riemann sum of the integral
\[\frac{a^{2(n-1)}}{s^{2n-1}}\int_1^\infty y^{2(n+1)}e^{-y^2\frac{a^2}{s^2}}dy=\frac{s^4}{a^3}\int_{a/s}^\infty z^{2(n+1)}e^{-z}dz.\]
Since \(n\) is even, \(\int_{a/s}^\infty z^{2(n+1)}e^{-z^2}dz=P(a/s)e^{-\frac{a^2}{s^2}}\) where \(P\) is a polynomial of degree \(2n+1\). Hence
\[\frac{s^4}{a^3}\int_{a/s}^\infty z^{2(n+1)}e^{-z}dz\lesssim P(a/s)\frac{s^4}{a^3}.\]
Thus, we obtain that
\begin{align*}
    \int_a^\infty \frac{1}{s^2}K ds\lesssim \int_a^\infty \frac{s^2}{a^3}P(a/s)\Vert\nabla\phi^s\Vert_{L^2}^2 ds\lesssim |\Delta|,
\end{align*}
where the integral is bounded independently of the choice of \(a\), whence the same bound remains valid when \(a\) tends to \(0\).
\medskip

For the first term \(J\), we abbreviate notation by setting \(f_{s,Q}:=\chi_{2Q}(\phi^s-(\phi^s)_{2Q})\) and continue with
\begin{align*}
J&\lesssim \sum_{Q\in\mathcal{D}_k}s\Vert\nabla e^{-s^2 L^s}f_{s,Q}\Vert_{L^2(Q)}^2=\sum_{Q\in\mathcal{D}_k}\int_{\mathbb{R}^n}\AP \nabla e^{-s^2 L^{s}}f_{s,Q}  \cdot \nabla e^{-s^2 L^{s}}f_{s,Q} s dx 
\\
&=\sum_{Q\in\mathcal{D}_k}\int_{\mathbb{R}^n}L^{s}e^{-s^2 L^{s}}f_{s,Q}  \cdot e^{-s^2 L^{s}}f_{s,Q} s dx 
\\
&=\sum_{Q\in\mathcal{D}_k}\Big[\int_{\mathbb{R}^n}\partial_s\big(e^{-s^2 L^{s}}f_{s,Q}\big)  \cdot e^{-s^2 L^{s}}f_{s,Q} s dx
\\
&\qquad -\int_{\mathbb{R}^n}\Big(\int_0^s2\tau e^{-(s^2-\tau^2) L^{s}}\mathrm{div}(\partial_s\AP\nabla e^{-\tau^2L^s}f_{s,Q})d\tau\Big) e^{-s^2 L^s}f_{s,Q} dx
\\
&\qquad -\int_{\mathbb{R}^n}e^{-s^2L^s}(\partial_s f_{s,Q})e^{-s^2L^2}f_{s,Q}dx\Big]
\\
&=\sum_{Q\in\mathcal{D}_k} J^{s,Q}_1+J^{s,Q}_2+J^{s,Q}_3.
\end{align*}
Here we used that \(\partial_s\big(e^{-s^2 L^{s}}f_{s,Q}\big)\) can be computed like \(\partial_s\rho=\partial_s\big(e^{-s^2 L^{s}}\phi_s\big)\) in Section \ref{section:rho}, since both are the same operator applied on a different on \(s\) depending function. The term on the left hand side on the penultimate equality corresponds to \(\int_{\mathbb{R}^n} w_t(x,s)\rho(x,s)\), while \(J_2\) corresponds to \(\int_{\mathbb{R}^n}w_s^{(2)}(x,s) \rho(x,s) dx\) and \(J_3\) corresponds to \(\int_{\mathbb{R}^n}w_s^{(1)}(x,s) \rho(x,s) dx\).
\medskip

First, we can observe for \(J_3^{s,Q}\) that
\begin{align*}
    \sum_{Q\in\mathcal{D}_k}J_3^{s,Q}&\lesssim \sum_{Q\in\mathcal{D}_k}\Vert e^{-s^2L^2}(\partial_sf_{s,Q})\Vert_{L^2}\Vert e^{-s^2L^2}f_{s,Q}\Vert_{L^2}
    \\
    &\lesssim \sum_{Q\in\mathcal{D}_k} \Vert \partial_s\phi^s-(\partial_s\phi^s)_{2Q}\Vert_{L^2(2Q)}\Vert \phi^s-(\phi^s)_{2Q}\Vert_{L^2(2Q)}
    \\
    &\lesssim s^2\Vert \nabla \partial_s\phi^s\Vert_{L^2}\Vert \nabla \phi^s\Vert_{L^2}\lesssim s^2\Vert \partial_s\AP(\cdot,s)\Vert_{\infty}|\Delta|.
\end{align*}
Here we used \refprop{prop:L2NormBoundsOfHeatSemigroup}, Poincar\'{e}'s inequality and \eqref{equation:L2ofpartial_snablaphiBounded}.
\medskip

Furthermore, we have for \(J_2^{s,Q}\) with Minkowski's inequality, \refprop{prop:L2NormBoundsOfHeatSemigroup} and Poincar\'{e}'s inequality
\begin{align*}
    \sum_{Q\in\mathcal{D}_k} J^{s,Q}_2&\lesssim \sum_{Q\in\mathcal{D}_k}\int_0^s2\tau \Vert e^{-(s^2-\tau^2) L^{s}}\mathrm{div}(\partial_s\AP\nabla e^{-\tau^2L^s}f_{s,Q})\Vert_{L^2} \Vert e^{-s^2 L^s}f_{s,Q}\Vert_{L^2} d\tau
    \\
    &\lesssim \sum_{Q\in\mathcal{D}_k}\Vert \partial_s\AP(\cdot, s)\Vert_\infty\Vert \phi^s-(\phi^s)_{2Q}\Vert_{L^2(2Q)}^2\int_0^s2\frac{1}{\sqrt{s^2-\tau^2}}  d\tau
    \\
    &\lesssim s^2\sum_{Q\in\mathcal{D}_k}\Vert \partial_s\AP(\cdot, s)\Vert_\infty\Vert \nabla\phi^s\Vert_{L^2(2Q)}^2
    \\
    &\lesssim s^2\Vert \partial_s\AP(\cdot,s)\Vert_{\infty}|\Delta|.
\end{align*}
Hence
\begin{align*}
    \int_a^\infty \frac{1}{s^2}\sum_k (J_2^{s,Q} + J_3^{s,Q})\lesssim \int_a^\infty \Vert \partial_s\AP(\cdot,s)\Vert_{\infty}|\Delta| ds\lesssim |\Delta|,
\end{align*}
where we used the \(L^1-L^\infty\) condition \eqref{cond:L1-Linfty}.
\medskip

Lastly, for \(J_1^{s,Q}\) we have
\begin{align*}
     \int_a^\infty \frac{1}{s^2}\sum_{Q\in\mathcal{D}_k} J_1^{s,Q}=-\sum_{Q\in\mathcal{D}_k}\int_a^\infty \frac{1}{s^2}\partial_s\big(\Vert e^{-s^2L^s}f_{s,Q}\Vert_{L^2(Q)}^2\big)ds.
\end{align*}
Since the integrand is nonnegative, we can again use \refprop{prop:L2NormBoundsOfHeatSemigroup} and Poincar\'{e}'s inequality to bound this above by
\begin{align*}
     \int_a^\infty \frac{1}{s^2}\sum_k J_1^{s,Q}&\lesssim -\frac{1}{a^2}\sum_{Q\in\mathcal{D}_k}\int_a^\infty \partial_s\big(\Vert e^{-s^2L^s}f_{s,Q}\Vert_{L^2(Q)}^2\big)ds 
     \\
     &\lesssim -\frac{1}{a^2}\sum_{Q\in\mathcal{D}_k}\Vert e^{-a^2L^a}f_{a,Q}\Vert_{L^2(Q)}^2
     \\
     &\lesssim -\frac{1}{a^2}\sum_{Q\in\mathcal{D}_k}\Vert \phi^a-(\phi^a)_{2Q}\Vert_{L^2(2Q)}^2\lesssim \Vert \nabla\phi^a\Vert_{L^2(\mathbb{R}^n)}^2\lesssim |\Delta|.
\end{align*}
Since these upper bounds are all independent of \(a\), if we take the limit when \(a\) tends to \(0\), we obtain that \(\int_0^\infty I ds\lesssim |\Delta|\).
\medskip

For the second term \(II\), we have similar to before in \(J\)
\begin{align*}
&II=s\Vert\nabla e^{-s^2 L^s}(sL^s\phi^s)\Vert_{L^2(\mathbb{R}^n)}^2
\\
&\qquad=\int_{\mathbb{R}^n}L^{s}e^{-s^2 L^{s}}(sL^s\phi^s)  \cdot e^{-s^2 L^{s}}(sL^s\phi^s) t dx
\\
&\qquad=\int_{\mathbb{R}^n}\partial_s(sL^se^{-s^2 L^{t}}\phi^s) \cdot sL^se^{-t^2 L^t}\phi^s dx
\\
&\qquad=\partial_s\int_{\mathbb{R}^n}(sL^se^{-s^2 L^{s}}\phi^s)^2 dx -\int_{\mathbb{R}^n}\partial_s (sL^se^{-t^2L^s}\phi^s)|_{t=s} sL^se^{-s^2 L^s}\phi^s dx
\\
&\qquad=\partial_s\int_{\mathbb{R}^n}(sL^se^{-s^2 L^{s}}\phi^s)^2 dx -\int_{\mathbb{R}^n}s\mathrm{div}(\partial_s\AP\nabla e^{-s^2L^s}\phi^s)\cdot sL^se^{-s^2 L^s}\phi^s dx
\\
&\qquad\qquad-\int_{\mathbb{R}^n}L^s e^{-s^2L^s}\phi^s\cdot sL^se^{-s^2 L^s}\phi^s dx
\\
&\qquad\qquad- \int_{\mathbb{R}^n}sL^s\Big(\int_0^s2\tau e^{-(s^2-\tau^2) L^{s}}\mathrm{div}(\partial_s\AP\nabla e^{-\tau^2L^a}\phi^s)d\tau\Big) sL^se^{-s^2 L^s}\phi^s dx
\\
&\qquad=:II_1+II_2+II_3+II_4.
\end{align*}
First, we note that \(II_3\) can be hidden on the left hand side since the whole integral term \(II\) is multiplied by a small constant \(\sigma\). Furthermore we obtain by integration by parts and \eqref{lemma:Nablaw_s^2SpatialBound} for \(II_4\)
\begin{align*}
    II_4&=\int_{\mathbb{R}^n}sL^sw_s^{(2)}(x) \cdot sL^se^{-s^2 L^s}\phi^s(x) dx
    \\
    &=-\int_{\mathbb{R}^n}s\AP(x,s)\nabla_{||}w_s^{(2)}(x) \cdot \nabla_{||}sL^s e^{-s^2 L^s}\phi^s(x) dx
    \\
    &=s\Vert\nabla w_s^{(2)}\Vert_{L^2}^2 +  \sigma s\Vert\nabla L^se^{-s^2 L^s}\phi^s\Vert_{L^2}^2
    \\
    &=s\Vert\partial_s \AP\Vert_\infty^2\Vert \nabla \phi^s\Vert_{L^2}^2 +  \sigma s\Vert\nabla L^se^{-s^2 L^s}\phi^s\Vert_{L^2}^2.
\end{align*}

We can hide the last term on the left hand side. Next, we have by integration by parts the same bound
\begin{align*}
    II_2&=\int_{\mathbb{R}^n}s\partial_s\AP(x,s)\nabla e^{-s^2L^s}\phi^s(x)\cdot s\nabla_{||}L^se^{-s^2 L^s}\phi^s(x) dx
    \\
    & \lesssim s\Vert\partial_s \AP\Vert_\infty^2\Vert \nabla \phi^s\Vert_{L^2}^2 +  \sigma s\Vert\nabla L^se^{-s^2 L^s}\phi^s\Vert_{L^2}^2.
\end{align*}
Together, we obtain
\begin{align*}
\int_a^\infty IIds&=\int_a^\infty II_1 dx +s\Vert\partial_s \AP\Vert_\infty^2\Vert \nabla \phi^s\Vert_{L^2}^2ds
\\
&=\int_a^\infty\partial_s\Big(\int_{\mathbb{R}^n}(sL^se^{-s^2 L^{s}}\phi^s)^2 dx\Big) +s\Vert\partial_s \AP\Vert_\infty^2\Vert \nabla \phi^s\Vert_{L^2}^2ds
\\
&= \Vert aL^ae^{-a^2 L^{a}}\phi^a\Vert_{L^2(\mathbb{R}^n)}^2 +|\Delta|\int_a^\infty\Vert\partial_s \AP\Vert_\infty^2sds
\\
&= \Vert \nabla\phi^a\Vert_{L^2(\mathbb{R}^n)}^2 +|\Delta|\int_a^\infty\Vert\partial_s \AP\Vert_\infty^2sds\lesssim |\Delta|.
\end{align*}
Here we used \refprop{Proposition11} and \eqref{condition:L^2Carlesontypecond}.


\color{black}

\hfill\\
The proofs of \eqref{lemma:w_tSqFctBound2} and \eqref{lemma:w_tSqFctBound3} rely on \eqref{lemma:w_tSqFctBound1} and \refprop{prop:L2NormBoundsOfHeatSemigroup}. For \eqref{lemma:w_tSqFctBound2} we see that
\begin{align*}
    &\Big\Vert\Big(\int_0^\infty\frac{|s^2\nabla L^s e^{-\eta^2s^2L^s}\phi^s|^2}{s}ds\Big)^{1/2}\Big\Vert_{L^2(\mathbb{R}^n)}^2
    =\int_{\mathbb{R}^{n+1}_+}\frac{|s^2\nabla L^s e^{-\eta^2s^2L^s}\phi^s|^2}{s} dxds
    \\
    &=\int_{\mathbb{R}^{n+1}_+}\frac{|s\nabla e^{-\frac{\eta^2s^2}{2}L^s}(sL^s e^{-\frac{\eta^2s^2}{2}L^s}\phi^s)|^2}{s} dxds
    \lesssim \int_{\mathbb{R}^{n+1}_+}\frac{|s L^s e^{-\frac{\eta^2s^2}{2}L^s}\phi^s|^2}{s} dxds,
\end{align*}
where the last integral is bounded by \eqref{lemma:w_tSqFctBound1}. Similarly for \eqref{lemma:w_tSqFctBound3}, we have
\begin{align*}
    &\Big\Vert\Big(\int_0^\infty\frac{|s^2L^s L^s e^{-\eta^2s^2L^s}\phi^s|^2}{s}ds\Big)^{1/2}\Big\Vert_{L^2(\mathbb{R}^n)}^2
    =\int_{\mathbb{R}^{n+1}_+}\frac{|s^2 L^s L^s e^{-\eta^2s^2L^s}\phi^s|^2}{s} dxds
    \\
    &=\int_{\mathbb{R}^{n+1}_+}\frac{|s L^s e^{-\frac{\eta^2s^2}{2}L^s}(sL^s e^{-\frac{\eta^2s^2}{2}L^s}\phi^s)|^2}{s} dxds
    \lesssim \int_0^\infty\int_{\mathbb{R}^n}\frac{|s L^s e^{-\frac{\eta^2s^2}{2}L^s}\phi^s|^2}{s} dxds.
\end{align*}

\end{proof}

In contrast to \(\partial_s\phi^s-w_s^{(1)}\) and \(w_s^{(2)}\) we have a certain local Harnack-type property for \(w_t\). First we have 

\begin{prop}[Proposition 6 in \cite{hofmann_dirichlet_2022}]\label{prop:Proposition6}
    Let \(L_{||}\) be a \(t\)-independent operator, \(Q\subset\mathbb{R}^n\) be a cube with side length \(l(Q)=R_0\) and \(\hat{Q}:=(1+\varepsilon)Q\) be an enlarged cube for some fixed \(\frac{1}{2}>\varepsilon>0\). Then
    \[\sup_{Q\times (R_0,2R_0]}|\partial_t e^{-t^2L_{||}}f(x)|^2\lesssim \fint_{\hat{Q}}\int_{(1-\varepsilon)R_0}^{2(1+\varepsilon)R_0}\frac{|\partial_t e^{-t^2L_{||}}f(x)|^2}{t}dxdt.\]
\end{prop}

We would like to have a similar result for \(w_t(x,s)\) whihc involves our family of operators \(L^s\). We are able to prove a weaker version which will still be enough for the proof of \refthm{thm:MainTheorem}.

\begin{lemma}\label{lemma:localHarnackTypeInequalityForw_t}
    Fix a boundary cube \(\Delta\) and let \(k_0\) be the scale of boundary cubes such that \(l(\Delta)\leq 2^{-k_0}\leq 3l(\Delta)\). Then it holds that
    \begin{align*}
    &\sum_{k\geq k_0}\sum_{Q\in \mathcal{D}^\eta_k(\Delta)}|Q|\sup_{(x,s)\in Q\times (2^{-k},2^{-k+1}]}|w_t(x,s)|^2
    \\
    &\lesssim \sum_{k\geq k_0}\sum_{Q\in \mathcal{D}_k^\eta(\Delta)}\int_{2Q}\int_{2^{-k-1}}^{2^{-k+2}}\frac{|w_t(x,s)|^2}{s}dxds + |\Delta|\int_0^{3l(\Delta)}\Vert\partial_s A(\cdot,s)\Vert_\infty^2 s ds.
    \end{align*}
\end{lemma}

\begin{proof}
Recall that the boundary cubes in \(\mathcal{D}_k^\eta\) have length comparable to \(\eta 2^{-k}\). Since \(w_t(x,s)=\partial_te^{-\eta^2t^2L^s}\phi^s|_{t= s}\) is also recalled by \(\eta\) in \(t\) the Whitney cube \(Q\times (2^{-k},2^{-k+1})\) for \(w_t\) corresponds to a usual cube for the operator \(\partial_te^{-t^2L^s}\) like in \refprop{prop:Proposition6}. All implicit constants that follow might depend on \(\eta\). Now, for every fixed operator \(L^s\) with \(s\in [2^{-k},2^{-k+1}]\), we get by \refprop{prop:Proposition6}

\begin{align*}
    \sup_{(x,s)\in Q\times [2^{-k},2^{-k+1}]}|w_t(x,s)|&\lesssim \sup_{s\in [2^-k,2^{-k+1}]}\fint_{5/4Q}\int_{2^{-k-1}}^{2^{-k+2}}\frac{|\partial_t w(x,\eta^2t^2;s)|^2}{t}dxdt
    \\
    &\lesssim \fint_{5/4Q}\int_{2^{-k-1}}^{2^{-k+2}}\frac{|\partial_t w(x,\eta^2 t^2;s^*)|^2}{t}dxdt,
\end{align*}    
where \(s^*\) is a value of \(s\) where half of the supremum is attained.
\smallskip

We would like to apply \reflemma{lemma:LEMMA5} again, and we will use the same notation for the Whitney cubes and their enlargements as previously in the proof of \reflemma{lemma:SqFctBoundsForw_s^2}. Let \(5/4<\EPS_1<\EPS_2<\frac{3}{2}\) with \(d:=\frac{\EPS_2-\EPS_1}{2}\). We define the enlargements \(5/4Q\subset Q_1\subset\tilde{Q}\subset Q_2\subset 3/2Q=:\hat{Q}\) of the boundary cube \(Q\) by
\begin{align*}
    Q_1:=\EPS_1Q,\qquad\tilde{Q}:=(\EPS_1+d)Q, \qquad \textrm{and }Q_2:=\EPS_2Q.
\end{align*}
Let us also define the corresponding enlargements of the Whitney cubes in \(\Omega\) satisfying \(l(Q)\approx \eta 2^{-k}\) and
\[5/4Q\times [2^{-k-1},2^{-k+2}]\subset Q_1^*\subset\tilde{Q}^*\subset Q_2^*\subset 3/2 Q\times [2^{-k-2}, 2^{-k+3}]=\hat{Q}^*,\]
where
\begin{align*}
    &Q_1^*:=Q_1\times [\frac{1}{\EPS_1}2^{-k-1},\EPS_1 2^{-k+2}],
    \\
    &\tilde{Q}^*:= \tilde{Q}\times [\frac{1}{\EPS_1+d}2^{-k-1},(\EPS_1+d)2^{-k+2}], \quad \textrm{and }
    \\
    &Q_2^*:=Q_2\times [\frac{1}{\EPS_2}2^{-k-1},\EPS_2 2^{-k+2}].
\end{align*}
Now we can also introduce a smooth cut-off function \(\psi\) with \(\psi\equiv 1\) on \(Q_1^*\) and \(\psi\equiv 0\) on \(\Omega\setminus \tilde{Q}^*\) with \(|\nabla \psi(x,t)|\lesssim \frac{1}{ td}\).

Furthermore, let us calculate \(\partial_s\partial_t w(x,t^2,s)\) (like in Section \ref{section:rho})
\begin{align}
    \partial_s\partial_t w(x,t^2;s)&=\partial_s (2tL^sw(x,t^2;s))\nonumber
    \\
    &=2t\Div(\partial_s\AP(x,s)\nabla w(x,t^2;s))+2tL^s\partial_sw(x,t^2;s)\nonumber
    \\
    &=2t\Div(\partial_s\AP(x,s)\nabla w(x,t^2;s))\nonumber
    \\
    &\qquad + 2tL^sv_1(x,t^2;s)+2tL^sv_2(x,t^2;s).\label{eq:partialstofw}
\end{align}
Here we recall that the functions \(w,v_1\) and \(v_2\) can be found in Section \ref{section:rho}.
\medskip

We also note that \(w_t(x,t)=\partial_t w(x,\eta^2t^;,r)\big|_{r=t}\) and we can apply the Fundamental Theorem of Calculus in the variable \(s\) and obtain
\begin{align*}
    &\fint_{Q_1^*}\frac{|\partial_t w(x,\eta^2t^2;s^*)|^2}{t}dxdt
    \leq \fint_{\tilde{Q}^*}\frac{|\partial_t w(x,\eta^2 t^2;s^*)|^2}{t}\psi^2(x,t) dxdt
    \\
    &= \fint_{\tilde{Q}^*}\partial_t w(x,\eta^2t^2;s^*)\Big(w_t(x,t) + \int_t^{s^*}\partial_s\partial_t w(x,\eta^2t^2;s)ds\Big)\psi^2(x,t) dxdt
\end{align*}
By \eqref{eq:partialstofw} we can continue with
\begin{align*}
    &=\fint_{\tilde{Q}^*}\partial_t w(x,\eta^2t^2;s^*)w_t(x,t)\psi^2(x,t) dxdt
    \\
    &\qquad+2\eta^2\fint_{\tilde{Q}^*}\int_t^{s^*}\partial_t w(x,\eta^2t^2;s^*)t\Div(\partial_s\AP(x,s)\nabla w(x,\eta^2t^2;s))\psi^2(x,t) dsdxdt
    \\
    &\qquad +2\eta^2\fint_{\tilde{Q}^*}\int_t^{s^*}\partial_t w(x,\eta^2t^2;s^*)tL^sv_1(x,\eta^2t^2;s)\psi^2(x,t) dsdxdt
    \\
    &\qquad+2\eta^2\fint_{\tilde{Q}^*}\int_t^{s^*}\partial_t w(x,\eta^2t^2;s^*)tL^sv_2(x,\eta^2t^2;s)\psi^2(x,t) dsdxdt
    \\
    &=:I_1+I_2+I_3+I_4.
\end{align*}
First, let us consider \(I_1\). We have
\begin{align*}
    I_1&\lesssim \sigma \fint_{\tilde{Q}^*}|\partial_t w(x,\eta^2t^2;s^*)|^2\psi^2 dxdt +C_\sigma\frac{1}{|Q|}\int_{\tilde{Q}^*}\frac{|w_t(x,t)|^2}{t} dxdt,
\end{align*}
where the first term can be hidden on the left hand side for a small choice of the parameter \(\sigma\), and the second one remains.
\smallskip

For \(I_2\) we use integration by parts to get
\begin{align*}
    I_2&=\fint_{\tilde{Q}^*}\int_t^{s^*}\partial_t \nabla w(x,\eta^2t^2;s^*)\psi^2(x,t)\cdot\partial_s\AP(x,s)\nabla w(x,\eta^2t^2;s) tdsdxdt
    \\
    &\qquad+\fint_{\tilde{Q}^*}\int_t^{s^*} \nabla\psi^2(x,t) \partial_t w(x,\eta^2t^2;s^*)\cdot \partial_s\AP(x,s)\nabla w(x,\eta^2t^2;s) tdsdxdt
    \\
    &\lesssim\fint_{\tilde{Q}^*}\int_t^{s^*}|\partial_t \nabla w(x,\eta^2t^2;s^*)||\partial_s\AP(x,s)||\nabla w(x,\eta^2t^2;s)| tdsdxdt
    \\
    &\qquad+\fint_{\tilde{Q}^*}\int_t^{s^*} \frac{|\partial_t w(x,\eta^2t^2;s^*)|}{td}|\partial_s\AP(x,s)||\nabla w(x,\eta^2t^2;s)| tdsdxdt
    \\
    &\lesssim \int_{\frac{1}{\EPS_1+d}2^{-k-1}}^{(\EPS_1+d)2^{-k+2}}\Big(\fint_{\tilde{Q}^*}|\partial_t \nabla w(x,\eta^2t^2;s^*)|^2dxdt\Big)^\frac{1}{2}
    \\
    &\qquad\qquad\cdot\Vert \partial_s\AP(\cdot,s)\Vert_\infty \Big(\fint_{\tilde{Q}^*}|\nabla w(x,\eta^2t^2;s)|^2dxdt\Big)^\frac{1}{2}sds
    \\
    &\qquad +\int_{\frac{1}{\EPS_1+d}2^{-k-1}}^{(\EPS_1+d)2^{-k+2}}\Big(\fint_{\tilde{Q}^*}\frac{|\partial_t w(x,\eta^2t^2;s^*)|^2}{t^2d^2}dxdt\Big)^\frac{1}{2}
    \\
    &\qquad\qquad \cdot\Vert \partial_s\AP(\cdot,s)\Vert_\infty \Big(\fint_{\tilde{Q}^*}|\nabla w(x,\eta^2t^2;s)|^2dxdt\Big)^\frac{1}{2}sds
\end{align*}

Since we have a Caccioppoli inequality for \(\partial_t w(x,\eta^2t^2;s^*)\) (see \reflemma{remark:CaccForpartial_tw}) we obtain
\[\fint_{\tilde{Q}^*}|\partial_t \nabla w(x,\eta^2t^2;s^*)|^2dxdt\lesssim\fint_{Q_2^*}\frac{|\partial_t w(x,\eta^2t^2;s^*)|^2}{t^2d^2}dxdt,\]
which gives together for a small \(\sigma>0\) and Hölder's inequality
\begin{align*}I_2&\lesssim \sigma \frac{1}{|Q|}\int_{Q_2^*}\frac{|\partial_t w(x,\eta^2t^2;s^*)|^2}{t} dxdt
\\
&\qquad+\frac{1}{d^2} \int_{\frac{1}{\EPS_1+d}2^{-k-1}}^{(\EPS_1+d)2^{-k+2}}\Vert \partial_s\AP(\cdot,s)\Vert_\infty^2\Big(\fint_{\tilde{Q}^*}|\nabla w(x,\eta^2t^2;s)|^2dxdt\Big) s ds.\end{align*}

The integral \(I_3\) and \(I_4\) can be handled similarly. 
First for \(I_3\), we use integration by parts to obtain
\begin{align*}
    I_3&=\fint_{\tilde{Q}^*}\int_t^{s^*}\partial_t \nabla w(x,\eta^2t^2;s^*)\psi^2(x,t)\AP(x,s)\nabla v_1(x,\eta^2t^2;s) sdsdxdt
    \\
    &\qquad+\fint_{\tilde{Q}^*}\int_t^{s^*} \nabla\psi^2(x,t) \partial_t w(x,\eta^2t^2;s^*) \AP(x,s)\nabla v_1(x,\eta^2t^2;s) sdsdxdt
    \\
    &\lesssim\fint_{\tilde{Q}^*}\int_t^{s^*}|\partial_t \nabla w(x,\eta^2t^2;s^*)||\nabla v_1(x,\eta^2t^2;s)| sdsdxdt
    \\
    &\qquad+\fint_{\tilde{Q}^*}\int_t^{s^*} \frac{|\partial_t w(x,\eta^2t^2;s^*)|}{dt}|\nabla v_1(x,\eta^2t^2;s)| sdsdxdt
    \\
    &\lesssim \int_{\frac{1}{\EPS_1+d}2^{-k-1}}^{(\EPS_1+d)2^{-k+2}}\Big(\fint_{Q_2^*}\frac{|\partial_t w(x,\eta^2t^2;s^*)|^2}{d^2t^2}dxdt\Big)^\frac{1}{2} \Big(\fint_{\tilde{Q}^*}|\nabla v_1(x,\eta^2t^2;s)|^2dxdt\Big)^\frac{1}{2}sds.
\end{align*}

We arrive at
\[I_3\lesssim \sigma \frac{1}{|Q|}\int_{Q_2^*}\frac{|\partial_t w(x,\eta^2t^2;s^*)|^2}{t} dxdt + \frac{1}{d^2}\int_{\frac{1}{\EPS_1+d}2^{-k-1}}^{(\EPS_1+d)2^{-k+2}}\fint_{\tilde{Q}^*}|\nabla v_1(x,\eta^2t^2;s)|^2sdxdtds.\]

Completely analogously for \(I_4\), we arrive at
\[I_4\lesssim \sigma \frac{1}{|Q|}\int_{Q_2^*}\frac{|\partial_t w(x,\eta^2t^2;s^*)|^2}{t} dxdt + \frac{1}{d^2}\int_{\frac{1}{\EPS_1+d}2^{-k-1}}^{(\EPS_1+d)2^{-k+2}}\fint_{\tilde{Q}^*}|\nabla v_2(x,\eta^2t^2;s)|^2sdxdtds.\]

Thus, we can put all pieces together to get
\begin{align*}
    &\fint_{Q_1^*}\frac{|\partial_t w(x,\eta^2t^2;s^*)|^2}{t}dxdt    
    \\
    &\lesssim \sigma \fint_{Q_2^*}|\partial_t w(x,\eta^2t^2;s^*)|^2 dxdt +\frac{1}{|Q|}\int_{\bar{Q}^*}\frac{|w_t(x,t)|^2}{t}dxdt 
    \\
    &\qquad +\frac{1}{d^2} \int_{\frac{1}{\EPS_1+d}2^{-k-1}}^{(\EPS_1+d)2^{-k+2}}\Vert \partial_s A(\cdot,s)\Vert_\infty^2 \Big(\fint_{\tilde{Q}^*}|\nabla w(x,\eta^2t^2;s)|^2dxdt\Big) s ds
    \\
    &\qquad +\frac{1}{d^2} \int_{\frac{1}{\EPS_1+d}2^{-k-1}}^{(\EPS_1+d)2^{-k+2}}\fint_{\tilde{Q}^*}|\nabla v_1(x,\eta^2t^2;s)|^2sdxdtds 
    \\
    &\qquad + \frac{1}{d^2}\int_{\frac{1}{\EPS_1+d}2^{-k-1}}^{(\EPS_1+d)2^{-k+2}}\fint_{\tilde{Q}^*}|\nabla v_2(x,\eta^2t^2;s)|^2sdxdtds.
\end{align*}

Applying  \reflemma{lemma:LEMMA5} we get
\begin{align*}
    &\fint_{5/4Q}\int_{2^{-k-1}}^{2^{-k+2}}\frac{|\partial_t w(x,t^2;s^*)|^2}{t}dxdt
    \\
    &\lesssim \frac{1}{|Q|}\int_{\hat{Q}^*}\frac{|w_t(x,t)|^2}{t}dxdt +\int_{\frac{1}{\EPS_1+d}2^{-k-1}}^{(\EPS_1+d)2^{-k+2}}\Vert \partial_s A(\cdot,s)\Vert_\infty^2 \Big(\fint_{\tilde{Q}^*}|\nabla w(x,\eta^2t^2;s)|^2dxdt\Big) s ds
    \\
    &\qquad + \int_{2^{-k-2}}^{2^{-k+3}}\fint_{\hat{Q}^*}|\nabla v_1(x,\eta^2t^2;s)|^2sdxdtds 
    \\
    &\qquad + \int_{2^{-k-2}}^{2^{-k+3}}\fint_{\hat{Q}^*}|\nabla v_2(x,\eta^2t^2;s)|^2sdxdtds
    \\
    &=:\frac{1}{|Q|}\int_{\hat{Q}^*}\frac{|w_t(x,t)|^2}{t}dxdt+J_1^Q+J_2^Q+J_3^Q.
\end{align*}

First, if we fix a scale \(k\) of the Whitney cubes and sum over all cubes in \(\mathcal{D}^\eta_k\), we obtain by \reflemma{lemma:nablaSemigroupBoundedByNablaf} and \eqref{equation:L2ofnablaphiBounded}
\begin{align*}
    &\sum_{Q\in\mathcal{D}^\eta_k(\Delta)}|Q|J_1^Q=\sum_{Q\in\mathcal{D}_k^\eta(\Delta)}|Q|\int_{\frac{1}{\EPS_1+d}2^{-k-1}}^{(\EPS_1+d)2^{-k+2}}\Vert \partial_s A(\cdot,s)\Vert_\infty^2 \Big(\fint_{\tilde{Q}^*}|\nabla w(x,\eta^2t^2;s)|^2dxdt\Big) s ds
    \\
    &\qquad\lesssim\int_{\frac{1}{\EPS_1+d}2^{-k-1}}^{(\EPS_1+d)2^{-k+2}}\Vert \partial_s A(\cdot,s)\Vert_\infty^2 \Big(\fint_{\frac{1}{\EPS_1+d}2^{-k-1},(\EPS_1+d)2^{-k+2}}\int_{3\Delta}|\nabla w(x,\eta^2t^2;s)|^2dxdt\Big) s ds
    \\
    &\qquad\lesssim \int_{\frac{1}{\EPS_1+d}2^{-k-1}}^{(\EPS_1+d)2^{-k+2}}\Vert \partial_s A(\cdot,s)\Vert_\infty^2 |\Delta| s ds.
\end{align*}

For \(J_2^Q\) we obtain also with \reflemma{lemma:nablaSemigroupBoundedByNablaf} and \eqref{equation:L2ofpartial_snablaphiBounded}
\begin{align*}
    \sum_{Q\in\mathcal{D}^\eta_k(\Delta)}|Q|J_2^Q=&\sum_{Q\in\mathcal{D}^\eta_k(\Delta)}|Q|\int_{\frac{1}{\EPS_1+d}2^{-k-1}}^{(\EPS_1+d)2^{-k+2}}\fint_{\tilde{Q}^*}|\nabla v_1(x,\eta^2t^2;s)|^2sdxdtds
    \\
    &\lesssim\int_{\frac{1}{\EPS_1+d}2^{-k-1}}^{(\EPS_1+d)2^{-k+2}}\fint_{\frac{1}{\EPS_1+d}2^{-k-1},(\EPS_1+d)2^{-k+2}}\int_{3\Delta}|\nabla v_1(x,\eta^2t^2;s)|^2sdxdtds
    \\
    &\lesssim \int_{\frac{1}{\EPS_1+d}2^{-k-1}}^{(\EPS_1+d)2^{-k+2}}\Vert \nabla\partial_s\phi^s\Vert^2 s ds
    \\
    &\lesssim \int_{\frac{1}{\EPS_1+d}2^{-k-1}}^{(\EPS_1+d)2^{-k+2}}\Vert \partial_s A(\cdot,s)\Vert_\infty^2 |\Delta| s ds.
\end{align*}

Lastly, if we use the Caccioppoli type inequality in \reflemma{Lemma:CacciopolliTypeInequalities} we can estimate 
\begin{align}
    \sum_{Q\in\mathcal{D}^\eta_k(\Delta)}|Q|J_3^Q&=\sum_{Q\in \mathcal{D}_k^\eta(\Delta)}|Q|\int_{2^{-k-2}}^{2^{-k+3}}\fint_{\hat{Q}^*}|\nabla v_2(x,\eta^2t^2;s)|^2 sdxdtds\nonumber
    \\
    &\lesssim \sum_{Q\in \mathcal{D}_k^\eta(\Delta)}|Q|\int_{2^{-k-2}}^{2^{-k+3}}\Big(\frac{1}{s}\fint_{2Q}\fint_{2^{-k-3}}^{2^{-k+4}}|v_2(x,\eta^2t^2;s)|^2dxdt \nonumber
    \\
    &\qquad\qquad + \Vert\partial_s A(\cdot,s)\Vert_{L^\infty(3\Delta)}^2s\fint_{2Q}\fint_{2^{-k-2}}^{2^{-k+3}}|\nabla e^{-\eta^2r^2L^s}\phi_s|^2drdx  \nonumber
    \\
    &\qquad\qquad + \Vert\partial_s A(\cdot,s)\Vert_{L^\infty(3\Delta)}^2s^3\fint_{2Q}\fint_{2^{-k-2}}^{2^{-k+3}}|\nabla \partial_r e^{-\eta^2r^2L^s}\phi_s|^2drdx  \Big)ds\nonumber
    \\
    &\lesssim \int_{2^{-k-2}}^{2^{-k+3}}\Big(\fint_{2^{-k-3}}^{2^{-k+4}}\frac{1}{s}\int_{\mathbb{R}^n} |v_2(x,\eta^2t^2;s)|^2 dx dt \Big) + \Vert\partial_s A(\cdot,s)\Vert_{L^\infty(3\Delta)}^2s \nonumber
    \\
    &\qquad\cdot \Big(\fint_{2^{-k-3}}^{2^{-k+4}}\int_{\mathbb{R}^n}|\nabla e^{-\eta^2r^2L^s}\phi_s|^2 + s^2|\nabla \partial_r e^{-\eta^2r^2L^s}\phi_s|^2 drdx\Big) ds.\label{eq:EquationWithv2}
\end{align}
Using \reflemma{lemma:nablaSemigroupBoundedByNablaf} and \eqref{equation:L2ofnablaphiBounded} we observe
\[\int_{3\Delta}|\nabla e^{-\eta^2r^2L^s}\phi_s|^2 + s^2|\nabla \partial_r e^{-\eta^2r^2L^s}\phi_s|^2dx\lesssim |\Delta|.\]
Further, we note that the proof of \eqref{lemma:w_s^2spatialL^2Bound} in \reflemma{lemma:w_s^2/nablaw_s^2spatialL^2Bound} works completely analogously for \(v_2(x,\eta^2t^2;s)\) instead of just \(w_s^{(2)}(x,s)=v_2(x,\eta^2s^2;s)\), which yields that
\[\int_{\mathbb{R}^n}|v_2(x,\eta^2t^2;s)|^2dx\lesssim \Vert\partial_s\AP(\cdot,s)\Vert_{L^\infty}^2|\Delta|s^2.\]
As a consequence, we can bound \eqref{eq:EquationWithv2} by
\[\int_{2^{-k-2}}^{2^{-k+3}}\Vert\partial_s A(\cdot,s)\Vert_{L^\infty(3\Delta)}^2s ds.\]

Similarly, we can see with \reflemma{lemma:nablaSemigroupBoundedByNablaf} and \eqref{equation:L2ofpartial_snablaphiBounded} that
\[\Vert \nabla v_2(\cdot,\eta^2t^2;s)\Vert_{L^2(\mathbb{R}^n)}\lesssim \Vert\nabla\partial_s\phi^s\Vert_{L^2(3\Delta)}\lesssim \Vert \partial_s \AP(\cdot,s)\Vert_{L^\infty(3\Delta)}|\Delta|^{1/2},\]
and hence
\begin{align*}
    &\sum_{Q\in \mathcal{D}^\eta_k(\Delta)}|Q|\int_{2^{-k-2}}^{2^{-k+3}}\fint_{\hat{Q}^*}|\nabla v_2(x,\eta^2t^2;s)|^2 sdxdtds
    \\
    &\lesssim \int_{2^{-k-2}}^{2^{-k+3}}\fint_{2^{-k-3}}^{2^{-k+4}}\int_{\mathbb{R}^n}|\nabla v_2(x,\eta^2t^2;s)|^2 sdxdtds
    \\
    &\lesssim \int_{2^{-k-2}}^{2^{-k+3}}\Vert\partial_s A(\cdot,s)\Vert_{L^\infty(3\Delta)}^2 |\Delta| s ds.
\end{align*}

Thus,
\begin{align*}
    &\sum_{k\geq k_0}\sum_{Q\in \mathcal{D}_k^\eta(\Delta)}|Q|\sup_{(x,s)\in Q\times (2^{-k},2^{-k+1}]}|w_t(x,s)|
    \\
    &\qquad\lesssim \sum_{k\geq k_0}\sum_{Q\in \mathcal{D}_k^\eta(\Delta)}\int_{\hat{Q}^*}\frac{|w_t(x,s)|^2}{s}dxds + |\Delta|\int_0^{3l(\Delta)}\Vert\partial_s A(\cdot,s)\Vert_\infty^2 s ds.
\end{align*}

\end{proof}

\subsection{Proof of \reflemma{lemma:L^2estimatesForSquareFunctions}}

Finally, we can turn to the proof of \reflemma{lemma:L^2estimatesForSquareFunctions}.

\begin{proof}[Proof of \reflemma{lemma:L^2estimatesForSquareFunctions}]

\begin{enumerate}[(a)]
    \item We split
    \begin{align*}
    \int_{T(\Delta)}|\nabla \partial_s \rho_\eta(x,s)|^2 sdxds&\leq \int_{T(\Delta)}|\nabla w_t(x,s)|^2 sdxds+\int_{T(\Delta)}|\nabla w_s^{(1)}(x,s)|^2 sdxds
    \\
    &\qquad+\int_{T(\Delta)}|\nabla w_s^{(2)}(x,s)|^2 sdxds,
    \end{align*}
    and apply \reflemma{lemma:SqFctBoundsForw_t}, \reflemma{lemma:SqFctBoundsForpartial_sPhiAndw_s^1} and \reflemma{lemma:SqFctBoundsForw_s^2} to get the required bound.
    
    \item We have
    \[\int_{T(\Delta)}|L^s \rho_\eta(x,s)|^2s dxds=\int_{T(\Delta)}\frac{|w_t(x,s)|^2}{\eta^2s} dxds\lesssim |\Delta|\]
    by \reflemma{lemma:SqFctBoundsForw_t}.

    \item We have
    \begin{align*}
    &\int_{T(\Delta)}|\partial_s A(x,s)\nabla\rho_\eta(x,s)|^2s dxds
    \\
    &\leq \int_0^{l(\Delta)}\Vert\partial_s A(\cdot,s)\Vert_\infty\Big(\int_\Delta|\nabla\phi^s|^2+|\nabla e^{-\eta^2s^2L^s}\phi^s|^2dx\Big) s ds
    \\
    &\lesssim \int_0^{l(\Delta)}\Vert\partial_s A(\cdot,s)\Vert_\infty|\Delta|s ds \lesssim |\Delta|
    \end{align*}
    by \eqref{equation:L2ofnablaphiBounded} and \eqref{condition:L^2Carlesontypecond}.

    \item We split
    \begin{align*}
        \int_{T(\Delta)}\frac{|\partial_s \theta_\eta(x,s)|^2}{s}dxds&\leq \int_{T(\Delta)}\frac{|(\partial_s\phi^s-w_s^{(1)})(x,s)|^2}{s}dxds
        \\
        &\qquad +\int_{T(\Delta)}\frac{|w_s^{(2)}(x,s)|^2}{s}dxds +\int_{T(\Delta)}\frac{|w_t(x,s)|^2}{s}dxds
    \end{align*}
    and the bound follows from \reflemma{lemma:SqFctBoundsForpartial_sPhiAndw_s^1}, \reflemma{lemma:SqFctBoundsForw_s^2}, and \reflemma{lemma:SqFctBoundsForw_t}.

    \item The proof works analogously to the proof of Lemma 9 in \cite{hofmann_dirichlet_2022} and uses a weighted Hardy inequality
    \[\int_0^\infty\Big(\frac{1}{s}\int_0^s |F(t)|dt\Big)^{p}\frac{ds}{s}\lesssim_p \int_0^\infty |F(s)|^p\frac{ds}{s} \qquad\textrm{for }1<p<\infty,\]
    which is proved there. With that and for \(p=2\), we have 
    \begin{align*}
        \int_{\mathbb{R}^n}\int_0^\infty \frac{|\theta_\eta(x,s)|^2}{s^3}dsdx&\leq \int_{\mathbb{R}^n}\int_0^\infty \Big(\frac{1}{s}\int_0^s |\partial_t e^{-\eta^2t^2L^s}\phi^s|dt\Big)^2\frac{ds}{s}dx
        \\
        &\lesssim \int_{\mathbb{R}^n}\int_0^\infty |w_t(x,s)|^2\frac{ds}{s}dx\lesssim |\Delta|.
    \end{align*}
    Here we also used \reflemma{lemma:SqFctBoundsForw_t} in the last inequality.
\end{enumerate}
\end{proof}

\section{Proof of \refthm{thm:MainTheorem}}

Recall that to show \(\omega\in A_\infty(d\sigma)\) it is enough to show \eqref{eq:SqFctonSawtooth} and we fixed a solution \(u\in W^{1,2}(\Omega)\) to the Dirichlet problem on \(\Omega\) for boundary data \(f\in C_c(\Omega)\) with \(\Vert u\Vert_{L^\infty}\leq \Vert f\Vert_{L^\infty}\leq 1\). Recall also, that for the fixed boundary ball \(\Delta\subset \partial\Omega\) we can fix a small constant \(\gamma\) for which we obtain a large constant \(\kappa_0\) and a good set \(F\subset \Delta\) from \reflemma{lemma:UniformBoundOnM[nablaphi^s]}.

To start with, we introduce a smooth cut-off function \(\psi\in C^{\infty}(\mathbb{R}^{n+1}_+)\) on \(T(\Delta)\cap \Omega_\alpha(F)\) with
\[\psi\equiv 1 \textrm{ on } T(\Delta)\cap\Omega_\alpha(F),\]
and
\[\psi\equiv 0 \textrm{ on } \Omega\setminus (T(3\Delta)\cap\Omega_{2\alpha}(F)).\]
Furthermore let \(\tilde{\delta}(x):=\mathrm{dist}(x,F)\) be the distance between the good set \(F\) and a point on the boundary \(x\in \partial\Omega\). To abbreviate, let
\begin{align*}
    &E_1=\big\{(x,t)\in T(3\Delta) ; \alpha t\leq \tilde{\delta}(x)\leq 2\alpha t \big\},
    \\
    &E_2=\big\{(x,t)\in 3\Delta\times [(l(\Delta),3l(\Delta)]; \tilde{\delta}(x)\leq 2\alpha t \big\}
\end{align*}
and we can choose \(\psi\) such that
\[|\nabla _{x,t}\psi(x,t)|\lesssim \frac{1}{\eta t}\chi_{E_1}(x,t)+\frac{1}{l(\Delta)}\chi_{E_2}(x,t) \qquad \textrm{for all }(x,t)\in \Omega.\]

We can start to estimate the left side of \eqref{eq:SqFctonSawtooth}. Since \(u\psi^2t\in W^{1,2}_0(\Omega)\) we have
\[\int_{\Omega}A\nabla u\cdot\nabla (u\psi^2t) dxdt=0,\]
and hence
\begin{align*}
    \int_{T(\Delta)\cap \Omega_\alpha(F)}|\nabla_{x,t} u|^2t dxdt&\leq \int_{\Omega}|\nabla_{x,t} u|^2\psi^2t dxdt
    \\
    &\lesssim \int_{\Omega}A\nabla_{x,t} u\cdot\nabla_{x,t} u\psi^2t dxdt
    \\
    &=\int_{\Omega}A\nabla_{x,t} u\cdot\nabla_{x,t} (u\psi^2t) dxdt - \int_{\Omega}A\nabla_{x,t} u\cdot \nabla_{x,t} \psi \psi u t dxdt 
    \\
    &\qquad - \int_{\Omega}A\nabla_{x,t} u\cdot \vec{e}_{n+1}u\psi^2 dxdt
    \\
    &=-\int_{\Omega}A\nabla_{x,t} u\cdot \nabla_{x,t} \psi \psi u t dxdt - \int_{\Omega}A\nabla_{x,t} u\cdot \vec{e}_{n+1}u\Psi^2 dxdt
    \\
    &\eqcolon J_1+J_2.
\end{align*}

For the first term \(J_1\) we have with boundedness of \(A\) that
\begin{align*} 
|J_1|&\leq \Big(\int_\Omega |\nabla_{x,t} u|^2\psi^2t dxdt\Big)^{1/2}\Big(\int_\Omega |\nabla_{x,t} \psi|^2u^2t dxdt\Big)^{1/2}
\\
&\leq \sigma \int_\Omega |\nabla_{x,t} u|^2\psi^2t dxdt + C_\sigma \int_\Omega |\nabla_{x,t} \psi|^2|u|^2t dxdt,
\end{align*}
and for sufficiently small \(\sigma>0\) we can hide the first term on the left side. For the other term we split into
\[\int_\Omega |\nabla_{x,t} \psi|^2u^2t dxdt\lesssim \int_\Omega \frac{|u|^2}{\eta^2 t}\chi_{E_1} dxdt+\int_\Omega \frac{|u|^2}{l(\Delta)}\chi_{E_2} dxdt=:\mathcal{E}_1 + \mathcal{E}_2.\]
We can bound \(\mathcal{E}_2\) by
\[\mathcal{E}_2\lesssim \int_{3\Delta\times [l(\Delta),3l(\Delta)]} \frac{|u|^2}{l(\Delta)}dxdt\lesssim \Vert u\Vert_{L^\infty(\Omega)}^2|\Delta|\leq |\Delta|.\]
For \(\mathcal{E}_1\) observe first that for \(x\notin F\)
\[\int_0^{3l(\Delta)}\frac{1}{t}\chi_{E_1}(x,t)dt=\int_{\alpha\tilde{\delta}(x)}^{2\alpha\tilde{\delta}(x)}\frac{1}{t}\chi_{E_1}(x,t)dt=\ln(2),\]
which leads to
\[\mathcal{E}_1\lesssim \Vert u\Vert_{L^\infty(\Omega)}\int_{3\Delta}\Big(\int_0^{3l(\Delta)}\frac{1}{t}\chi_{E_1}(x,t)dt\Big)dx\lesssim |\Delta|.\]

For \(J_2\) we continue to break up the integral into
\[-\int_{\Omega}A\nabla_{x,t} u\cdot \vec{e}_{n+1}u\psi^2 dxdt=-\int_{\Omega}c\cdot\nabla_x u u\psi^2 dxdt-\int_{\Omega}d \partial_t u u\psi^2 dxdt=:J_{21}+J_{22}.\]

We have for \(J_{22}\)
\begin{align*}
    -J_{22}&=\int_{\Omega}\partial_t(d u^2\psi^2) dxdt + \int_{\Omega}d u^2\partial_t\psi^2 dxdt + \int_{\Omega}\partial_t d u^2\psi^2
    \\
    &=\int_{3\Delta}d u^2\psi^2 dxdt + \int_{\Omega}\frac{u^2}{\eta t}\chi_{E_1}dxdt + \int_{\Omega}\frac{u^2}{l(\Delta)}\chi_{E_2}dxdt + \int_{\Omega}\partial_t d u^2\psi^2.
\end{align*}
While the second and third term are \(\mathcal{E}_1\) and \(\mathcal{E}_2\) respectively and hence can be bounded by \(|\Delta|\), the first term can be bounded by \(\Vert d\Vert_\infty\Vert u\Vert_\infty|\Delta|\) and hence \(|\Delta|\) directly. The last term can also be bounded with \eqref{cond:L1-Linfty}. Thus it remains to bound \(J_{21}\). From here on we drop the subscript of the gradient \(\nabla=\nabla_x\). 

We proceed with
\begin{align*}
    J_{21}=-\int_{\Omega}c\cdot\nabla (u^2\psi^2) dxdt + \int_{\Omega}c\cdot\nabla \psi u^2\psi dxdt:=J_{211}+J_{212}.
\end{align*}
Again, we have 
\[ |J_{212}|\lesssim\int_{\Omega}\frac{u^2}{\eta t}\chi_{E_1}dxdt + \int_{\Omega}\frac{u^2}{l(\Delta)}\chi_{E_2}dxdt\lesssim \mathcal{E}_1+\mathcal{E}_2.\]
For the other term \(J_{211}\) we apply the Hodge decomposition of \(c\) (see Section \ref{section:HodgeDecomposition} and Section \ref{section:rho}) to get
\begin{align*}
    -J_{211}&=\int_{\Omega}\AP\nabla\phi^s \cdot\nabla(u^2\psi^2)dxdt
    \\
    &=\int_{\Omega}\AP\nabla \theta_\eta \cdot\nabla(u^2\psi^2)dxdt - \int_{\Omega}\AP\nabla\rho_\eta \cdot\nabla(u^2\psi^2)dxdt
    \\
    &=\int_{\Omega}\AP\nabla \theta_\eta \cdot\nabla(u^2\psi^2)dxdt - \int_{\Omega}\partial_t\AP\nabla\rho_\eta \cdot\nabla(u^2\psi^2)tdxdt 
    \\
    &\qquad  -\int_{\Omega}\AP\nabla \partial_t\rho_\eta \cdot\nabla(u^2\psi^2)tdxdt - \int_{\Omega}\AP\nabla\rho_\eta \cdot\nabla\partial_t(u^2\psi^2)tdxdt
    \\
    &=:J_{2111}+I_1+I_2+I_3.
\end{align*}

First we deal with \(I_1,I_2\) and \(I_3\). By \reflemma{lemma:L^2estimatesForSquareFunctions} we have
\[I_1\lesssim \Big(\int_{T(3\Delta)}|\nabla\partial_t\rho_\eta|^2tdxdt\Big)^{1/2}\Big(\int_{T(3\Delta)}|\nabla(u^2\psi^2)|^2tdxdt\Big)^{1/2}\lesssim |\Delta|^{1/2} J_1^{1/2}\lesssim |\Delta|.\]

Next, for \(I_2\) we have by \reflemma{lemma:L^2estimatesForSquareFunctions}
\[I_2\lesssim \Big(\int_{T(3\Delta)}|\Div(\AP\nabla\rho_\eta)|^2tdxdt\Big)^{1/2}\Big(\int_{T(3\Delta)}|\partial_t(u^2\psi^2)|^2tdxdt\Big)^{1/2}\lesssim |\Delta|^{1/2} J_1^{1/2}\lesssim |\Delta|.\]

At last, for \(I_3\) we have by \reflemma{lemma:L^2estimatesForSquareFunctions}
\[I_3\lesssim \Big(\int_{T(3\Delta)}|\partial_t \AP|^2|\nabla\rho_\eta|^2tdxdt\Big)^{1/2}\Big(\int_{T(3\Delta)}|\nabla(u^2\psi^2)|^2tdxdt\Big)^{1/2}\lesssim |\Delta|^{1/2} J_1^{1/2}\lesssim |\Delta|.\]

We proceed with \(J_{2111}\) and integration by parts to obtain
\begin{align*}
    -\int_{\Omega}\AP\nabla \theta_\eta \cdot\nabla(u^2)\psi^2dxdt - \int_{\Omega}\AP\nabla \theta_\eta \cdot u^2\nabla(\psi^2)dxdt=:II_1+II_2.
\end{align*}

For \(II_2\) we observe that
\begin{align*} 
    |II_2|&\lesssim \Big(\int_{\Omega}\frac{|\nabla\theta_\eta|^2}{t}\chi_{E_1\cup E_2}dxdt\Big)^{1/2}\Big(\int_{T(\Delta)}|\nabla u|^2\psi^2 tdxdt\Big)^{1/2}
    \\
    &\lesssim C_\sigma \int_{\Omega}\frac{|\nabla\theta_\eta|^2}{t}\chi_{E_1\cup E_2}dxdt+\sigma \int_{T(\Delta)}|\nabla u|^2\psi^2 tdxdt,
\end{align*}
and for small \(\sigma>0\) we can hide the second term on the left side. For the first term we can calculate with \reflemma{lemma:UnifromBoundOnrholocally} and \reflemma{lemma:UniformBoundOnM[nablaphi^s]}
\begin{align*}
    \int_{\Omega}\frac{|\nabla\theta_\eta|^2}{t}\chi_{E_1\cup E_2}dxdt&\lesssim \sum_{k}\sum_{Q^\prime\in \mathcal{D}^\eta_k(3\Delta)}\int_{Q^\prime}\fint_{\eta 2^{-k}}^{\eta 2^{-k+1}}|\nabla \theta_\eta|^2\chi_{E_1\cup E_2}dxdt
    \\
    &\lesssim \sum_{k}\sum_{Q^\prime\in \mathcal{D}^\eta_k(3\Delta)}|Q^\prime|\fint_{Q^\prime}\fint_{\eta 2^{-k}}^{\eta 2^{-k+1}}|\nabla \rho_\eta|^2\chi_{E_1\cup E_2}dxdt
    \\
    &\qquad +\sum_{k}\sum_{Q^\prime\in \mathcal{D}^\eta_k(3\Delta)}|Q^\prime|\fint_{Q^\prime}\fint_{\eta 2^{-k}}^{\eta 2^{-k+1}}|\nabla \phi^t|^2\chi_{E_1\cup E_2}dxdt
    \\
    &\lesssim \sum_{k}\sum_{Q^\prime\in \mathcal{D}^\eta_k(3\Delta)}|Q^\prime| \kappa_0 \lesssim |\Delta|.
\end{align*}

For \(II_1\) we get by integration by parts
\begin{align*}
    II_1&=-\int_{\Omega}\AP\nabla u \cdot\nabla(\theta_\eta u \psi^2)dxdt + \int_{\Omega}\AP\nabla u \cdot\nabla u \theta_\eta \psi^2dxdt + \int_{\Omega}\AP\nabla u \cdot\nabla \psi^2 u \theta_\eta dxdt
    \\
    &=:II_{11}+II_{12}+II_{13}.
\end{align*}

Due to \reflemma{lemma:thetaPointwiseBoundOnOmega(F)} we can bound \(\theta_\eta\) pointwise by \(t\) and handle \(II_{13}\) like \(J_1\). For \(II_{12}\) the same lemma gives 
\[II_{12}=\eta \kappa_0 \int_{T(3\Delta)}|\nabla u|^2\psi^2 tdxdt.\]
Here is where we have to choose \(\eta\) sufficiently small so that we can hide this term on the left side. To deal with \(II_{11}\) we use that \(u\) is a weak solution and \(\theta_\eta u \psi^2\in W_0^{1,2}(\Omega)\) to conclude
\begin{align*}
    II_{11}&=-\int_{\Omega}\AP\nabla u \cdot\nabla(\theta_\eta u \psi^2)dxdt
    \\
    &= \int_{\Omega}b\partial_t u \cdot\nabla(\theta_\eta u \psi^2)dxdt
    + \int_{\Omega}c\cdot \nabla u \partial_t(\theta_\eta u \psi^2)dxdt
    + \int_{\Omega}d\partial_t u \partial_t(\theta_\eta u \psi^2)dxdt
    \\
    &=:II_{111}+II_{112}+II_{113}.
\end{align*}

For \(II_{113}\) we use the product rule and then use the bound on \(\theta_\eta\) from \reflemma{lemma:thetaPointwiseBoundOnOmega(F)} to get
\begin{align*}
    II_{113}&=\int_{\Omega}d\partial_t u \partial_t\theta_\eta u \psi^2 dxdt + \int_{\Omega}d\partial_t u \theta_\eta \partial_t u \psi^2 dxdt + \int_{\Omega}d\partial_t u \theta_\eta u \partial_t \psi^2 dxdt
    \\
    &\lesssim \Big(\int_{T(3\Delta)}\frac{|\partial_t\theta_\eta|^2}{t} dxdt\Big)^{1/2}\Big(\int_{T(3\Delta)}|\partial_t u|^2\psi^2 tdxdt\Big)^{1/2} 
    + \int_{T(3\Delta)}|\partial_t u|^2\psi^2 \theta_\eta dxdt
    \\
    &\qquad + \int_{\Omega}|\partial_t u \theta_\eta u \partial_t \psi^2| dxdt
    \\
    &\lesssim C_\sigma \int_{T(3\Delta)}\frac{|\partial_t\theta_\eta|^2}{t} dxdt + \sigma \int_{T(3\Delta)}|\partial_t u|^2\psi^2 tdxdt
    \\
    &\qquad +\eta \int_{T(3\Delta)}|\partial_t u|^2\psi^2 t dxdt +  \int_{\Omega}|\partial_t u^2 \partial_t \psi^2|t dxdt.
\end{align*}

Hence, for a sufficiently small choice of \(\eta\) and \(\sigma\) we can hide the second and third term on the left side, while the forth term is dealt with like \(J_1\) and boundedness of the first one by \(|\Delta|\) follows from \reflemma{lemma:L^2estimatesForSquareFunctions}.
\smallskip

Analogoulsy, we get for \(II_{112}\)
\begin{align*}
    II_{112}&=\int_{\Omega}c\cdot\nabla u \partial_t\theta_\eta u \psi^2 dxdt + \int_{\Omega}c\cdot\nabla u \theta_\eta \partial_t u \psi^2 dxdt + \int_{\Omega}c\cdot\nabla u \theta_\eta u \partial_t \psi^2 dxdt
    \\
    &\lesssim \Big(\int_{T(3\Delta)}\frac{|\partial_t\theta_\eta|^2}{t} dxdt\Big)^{1/2}\Big(\int_{T(3\Delta)}|\nabla u|^2\psi^2 tdxdt\Big)^{1/2} 
    + \int_{T(3\Delta)}|\nabla u|^2\psi^2 \theta_\eta dxdt
    \\
    &\qquad + \int_{T(3\Delta)}|\nabla u \theta_\eta u \partial_t \psi^2| dxdt
    \\
    &\lesssim C_\sigma \int_{T(3\Delta)}\frac{|\partial_t\theta_\eta|^2}{t} dxdt + \sigma \int_{T(3\Delta)}|\nabla u|^2\psi^2 tdxdt
    \\
    &\qquad +\eta \int_{T(3\Delta)}|\nabla u|^2\psi^2 t dxdt +  \int_{T(3\Delta)}|\nabla u^2 \partial_t \psi^2|t dxdt,
\end{align*}

and hence \(II_{112}\lesssim |\Delta|\) with the same arguments as for \(II_{113}\). For \(II_{111}\) however, we split the integral into
\begin{align*}
    II_{111}&=\int_{\Omega}b\cdot \nabla\theta_\eta \partial_t u^2 \psi^2 dxdt + \int_{\Omega}b\cdot \nabla u \partial_t u^2 \psi^2 \theta_\eta dxdt + \int_{\Omega}b\cdot \nabla \psi u\partial_t u\psi \theta_\eta dxdt
    \\
    &=:II_{1111}+II_{1112}+II_{1113}.
\end{align*}
Since \reflemma{lemma:thetaPointwiseBoundOnOmega(F)} implies
\[II_{1112}\lesssim \eta \int_{T(3\Delta)}|\nabla u|^2\psi^2 t dxdt,\]
and
\[II_{1113}\lesssim C_\sigma\int_{T(3\Delta)}\frac{|u|^2}{t}\chi_{E_1\cup E_2} dxdt +\sigma\int_{T(3\Delta)}|\partial_t u|^2\psi^2 tdxdt\]
we get boundedness of these two terms like before in \(\mathcal{E}_1,\mathcal{E}_2\) and through hiding of terms on the left side with small choices of \(\eta\) and \(\sigma\).
\smallskip

Next for \(II_{1111}\), we integrate by parts in \(\partial_t\) direction and obtain
\begin{align*}
     II_{1111}&=-\int_{\Omega}b\cdot \nabla\partial_t\theta_\eta u^2 \psi^2 dxdt - \int_{\Omega}b\cdot \nabla \theta_\eta u^2\psi \partial_t\psi  dxdt - \int_{\Omega}\partial_t b\cdot \nabla \theta_\eta u^2\psi^2  dxdt.
\end{align*}
Recalling that \(\partial_t\theta_\eta=\partial_t\phi^t-(w_t+w_s^{(1)}+w_s^{(2)})\) allows to conclude further
\begin{align*}
     II_{1111}&=-\int_{\Omega}b\cdot \nabla w_t u^2 \psi^2 dxdt - \int_{\Omega}b\cdot \nabla(\partial_t\phi^t-w_s^{(1)}+w_s^{(2)}) u^2 \psi^2 dxdt
    \\
    &\qquad
    - \int_{\Omega}b\cdot \nabla \theta_\eta u^2\psi \partial_t\psi  dxdt - \int_{\Omega}\partial_t b\cdot \nabla \theta_\eta u^2\psi^2  dxdt
     \\
     &\lesssim -\int_{\Omega}b\cdot \nabla(w_t u^2 \psi^2) dxdt - \int_{\Omega}b\cdot \nabla(\partial_t\phi^t-w_s^{(1)}+w_s^{(2)}) u^2 \psi^2 dxdt
    \\
    &\qquad +\int_{\Omega}b\cdot \nabla (u^2 \psi^2) w_t dxdt
    - \int_{\Omega}b\cdot \nabla \theta_\eta u^2\psi \partial_t\psi  dxdt - \int_{\Omega}\partial_t b\cdot \nabla \theta_\eta u^2\psi^2  dxdt
    \\
    &=:III_{1}+III_{2}+III_{3}+III_{4}+III_5.
\end{align*}
We can bound \(III_{4}\) like \(II_2\). For \(III_3\) Hölder's inequality yields 
\begin{align*}
    III_{3}&\lesssim \int_{T(3\Delta)}|w_t\nabla u u\psi^2|+|w_t u^2\nabla\psi \psi|dxdt
    \\
    &\lesssim \Big(\int_{T(3\Delta)}\frac{|w_t|^2}{t} dxdt\Big)^{1/2}\Big(\int_{T(3\Delta)}|\nabla u|^2\psi^2 t dxdt\Big)^{1/2}
    \\
    &\qquad + \Big(\int_{T(3\Delta)}\frac{|w_t|^2}{t} dxdt\Big)^{1/2}\Big(\int_{T(3\Delta)}|u\nabla\psi|^2 t dxdt\Big)^{1/2}
    \\
    &\lesssim C_\sigma\int_{T(3\Delta)}\frac{|w_t|^2}{t} dxdt + \sigma\int_{T(3\Delta)}|\nabla u|^2\psi^2 t dxdt + \int_{T(3\Delta)}|u\nabla\psi|^2 t dxdt.
\end{align*} 
Bounding the first term by \reflemma{lemma:SqFctBoundsForw_t}, hiding the second one for a small choice of \(\sigma\) on the left hand side, and bounding the third term like in \(\mathcal{E}_1\) and \(\mathcal{E}_2\) we obtain boundedness of \(III_3\) by \(|\Delta|\). 

For \(III_2\) we observe that due to \reflemma{lemma:SqFctBoundsForw_s^2}, \reflemma{lemma:nablaSemigroupBoundedByNablaf}, \eqref{equation:L2ofpartial_snablaphiBounded} and \eqref{equation:L2ofnablaphiBounded} we have
\begin{align*}
    III_2&\lesssim\int_{\Omega}|\nabla\partial_t\phi^t| +|\nabla w_s^{(1)}|+|\nabla w_s^{(2)}| dxdt
    \\
    &\lesssim |\Delta|^\frac{1}{2}\Big(\int_0^{l(\Delta)}\Vert \nabla\partial_t\phi^t\Vert_{L^2(3\Delta)}dt + \int_0^{l(\Delta)}\Vert \nabla w_s^{(1)}\Vert_{L^2(3\Delta)}dt + \int_0^{l(\Delta)}\Vert \nabla w_s^{(2)}\Vert_{L^2(3\Delta)}dt\Big)
    \\
    &\lesssim |\Delta|\int_0^{l(\Delta)}\Vert \partial_t A(\cdot,t)\Vert_{L^\infty(\mathbb{R}^n)}dt\lesssim |\Delta|.
\end{align*}

For \(III_5\) we get with \eqref{cond:L1-Linfty}, \reflemma{lemma:nablaSemigroupBoundedByNablaf}, and \eqref{equation:L2ofnablaphiBounded}
\[III_5=\int_{\Omega}|\partial_t b\cdot \nabla \theta_\eta u^2\psi^2|  dxdt\lesssim \int_0^{3l(\Delta)}\Vert\partial_t b(\cdot,t)\Vert_\infty|\Delta|^{1/2}\Big(\int_{3\Delta}|\nabla \theta_\eta|^2\Big)^{1/2}dxdt\lesssim |\Delta|.\]

For \(III_1\) however, we need the Hodge decomposition of \(b\). Since \(w_t\psi^2 u^2\in W_0^{1,2}(T(3\Delta))\) we have

\begin{align*}
    III_1= \int_{\Omega}\AP\nabla\tilde{\theta}_\eta\cdot \nabla(w_t u^2 \psi^2) dxdt + \int_{\Omega}\AP\nabla\tilde{\rho}_\eta\cdot \nabla(w_t u^2 \psi^2) dxdt
    =:III_{11}+III_{12}.
\end{align*}

With integration by parts the second term becomes
\[III_{12}\lesssim \Big(\int_{T(3\Delta)}\frac{|w_t|^2}{t} dxdt\Big)^{1/2}\Big(\int_{T(3\Delta)}|L^t\tilde{\rho}_\eta|^2 t dxdt\Big)^{1/2}\lesssim |\Delta|,\]
where the last inequality follows from \reflemma{lemma:L^2estimatesForSquareFunctions} and \reflemma{lemma:SqFctBoundsForw_t}.

For \(III_{11}\) we obtain
\begin{align*}
    III_{11}&=\int_{\Omega}\AP\nabla\tilde{\theta}_\eta\cdot \nabla u^2 w_t \psi^2 dxdt + \int_{\Omega}\AP\nabla\tilde{\theta}_\eta\cdot \nabla \psi^2 w_t u^2  dxdt + \int_{\Omega}\AP\nabla\tilde{\theta}_\eta\cdot \nabla w_t u^2 \psi^2 dxdt
    \\
    &=:III_{111}+III_{112}+III_{113}.
\end{align*}

First, we see that 
\[III_{112}\lesssim\Big(\int_{T(3\Delta)}\frac{|w_t|^2}{t} dxdt\Big)^{1/2}\Big(\int_{T(3\Delta)}|\nabla\tilde{\theta}_\eta|^2|\nabla\psi|^2 t dxdt\Big)^{1/2}\] 
can be dealt with by \(II_2\) and \reflemma{lemma:SqFctBoundsForw_t}.

Next, for \(III_{113}\) we obtain by integration by parts
\begin{align*}
    III_{113}&=\int_{\Omega} L^t w_t u^2\psi^2 \tilde{\theta}_\eta dxdt + \int_{\Omega} \AP \nabla w_t \cdot\nabla(u^2\psi^2) \tilde{\theta}_\eta dxdt,
\end{align*}
where the first term can be bounded by
\[\Big(\int_{T(3\Delta)}\frac{|\tilde{\theta}_\eta|^2}{t^3} dxdt\Big)^{1/2}\Big(\int_{T(3\Delta)}|L^tw_t|^2 t^3 dxdt\Big)^{1/2}\lesssim |\Delta|\]
due to \reflemma{lemma:L^2estimatesForSquareFunctions} and \reflemma{lemma:SqFctBoundsForw_t}, while the second term is bounded by
\begin{align*}
    &\int_{T(3\Delta)}|\nabla w_t\nabla u u\psi^2|t+|\nabla w_t u^2\nabla\psi \psi|tdxdt
    \\
    &\lesssim C_\sigma\int_{T(3\Delta)}|\nabla w_t|^2t dxdt+\sigma\int_{T(3\Delta)}|\nabla u|^2\psi^2 t dxdt + \int_{T(3\Delta)}|u\nabla\psi|^2 t dxdt,
\end{align*}
which is bounded by \(|\Delta|\) using the same arguments as before in \(III_3\) (hiding terms on the left side, \reflemma{lemma:SqFctBoundsForw_t}, \(\mathcal{E}_1\) and \(\mathcal{E}_2\)).

At last, it remains to bound \(III_{111}\). For that we have
\[III_{111}\lesssim \sigma\int_{T(3\Delta)}|\nabla u|^2\psi^2 t dxdt + C_\sigma\int_{T(3\Delta)}\frac{|\nabla\tilde{\theta}_\eta|^2|w_t|^2\chi_{\mathrm{supp}(\psi)}}{t} dxdt.\]
The first term can be hiden on the left side with sufficiently small choice of \(\sigma\). For the second one we write 
\begin{align*}
    &\int_{T(3\Delta)}\frac{|\nabla\tilde{\theta}_\eta|^2|w_t|^2\chi_{\mathrm{supp}(\psi)}}{t} dxdt
    =\sum_k\sum_{Q\in \mathcal{D}^\eta_k}\int_{Q}\int_{2^{-k}}^{2^{-k+1}}\frac{|\nabla\tilde{\theta}_\eta|^2|w_t|^2\chi_{\mathrm{supp}(\psi)}}{t}dxdt
    \\
    &\qquad=\sum_k\sum_{Q\in \mathcal{D}^\eta_k}\Big(\fint_{Q}\int_{2^{-k}}^{2^{-k+1}}\frac{|\nabla\tilde{\theta}_\eta|^2\chi_{\mathrm{supp}(\psi)}}{t}dxdt\Big)\Big(|Q|\sup_{Q\times (2^{-k},2^{-k+1}]}|w_t|\Big)
\end{align*}
Using \reflemma{lemma:UnifromBoundOnrholocally} and \reflemma{lemma:UniformBoundOnM[nablaphi^s]} we see that
\begin{align*}
&\fint_{Q}\fint_{ 2^{-k}}^{2^{-k+1}}|\nabla\tilde{\theta}_\eta|^2\chi_{\mathrm{supp}(\psi)}dxdt
\\
&=\fint_{Q}\fint_{ 2^{-k}}^{ 2^{-k+1}}|\nabla\tilde{\rho}_\eta|^2\chi_{\mathrm{supp}(\psi)}dxdt+\fint_{Q}\fint_{2^{-k}}^{2^{-k+1}}|\nabla\tilde{\phi}^t|^2\chi_{\mathrm{supp}(\psi)}dxdt\lesssim \kappa_0,
\end{align*}

and hence with \reflemma{lemma:localHarnackTypeInequalityForw_t}
\begin{align*}
    \int_{T(3\Delta)}\frac{|\nabla\tilde{\theta}_\eta|^2|w_t|^2\chi_{\mathrm{supp}(\psi)}}{t} dxdt
    &\lesssim\sum_k\sum_{Q\in \mathcal{D}^\eta_k}\int_{2Q}\int_{ 2^{-k-1}}^{ 2^{-k+2}}\frac{|w_t|^2}{t}dxdt
    \\
    &\lesssim\int_{T(3\Delta)}\frac{|w_t|^2}{t} dxdt + |\Delta|\int_0^{5l(\Delta)}\Vert\partial_s A(\cdot,s)\Vert_\infty^2 s ds\lesssim |\Delta|.
\end{align*}
The last inequality follows again from \reflemma{lemma:L^2estimatesForSquareFunctions}. In total we verified \eqref{eq:SqFctonSawtooth}.

\section{The improvement in \(n=1\)}

\subsection{Proof of \refthm{thm:MainTheoremn=1}}
The proof of \refthm{thm:MainTheoremn=1} is analogous to the proof of \refthm{thm:MainTheorem} with minor modifications. For the \(L^1-L^\infty\) condition \eqref{cond:L1-Linfty} we collected the bounds we needed on quantities involving \(\theta\) and \(\rho\) in \reflemma{lemma:L^2estimatesForSquareFunctions} and we will prove that in \(n=1\) the same results hold assuming merely \eqref{condition:L^1Carlesontypecond}.
\smallskip

For this we first replace the uniform pointwise bound of \(M[\nabla\phi^s]\) (\reflemma{lemma:UniformBoundOnM[nablaphi^s]}) and all the area function estimates of the different partial derivatives of \(\theta\) (\reflemma{lemma:SqFctBoundsForpartial_sPhiAndw_s^1}, \reflemma{lemma:SqFctBoundsForw_s^2}, and  \reflemma{lemma:SqFctBoundsForw_t}) by analogous results assuming \eqref{condition:L^1Carlesontypecond} instead of \eqref{cond:L1-Linfty}. These are \reflemma{lemma:UniformBoundOnM[nablaphi^s]n=1}, \reflemma{lemma:SqFctBoundsForpartial_sPhiAndw_s^1n=1}, \reflemma{lemma:SqFctBoundsForw_s^2n=1}, and  \reflemma{lemma:SqFctBoundsForw_tn=1} and will appear in the next sections. The key improvement that allows proving the same results under a weaker condition \eqref{condition:L^1Carlesontypecond} comes from the fact that in \(n=1\) the uniform pointwise bound of \(M[\nabla\phi^s]\) (\reflemma{lemma:UniformBoundOnM[nablaphi^s]}), that only held on a good set, now actually holds everywhere (see \reflemma{lemma:UniformBoundOnM[nablaphi^s]n=1}).
\smallskip

Now let us address the minor modifications that are needed in following the main proof of \refthm{thm:MainTheorem} in the previous section. Since we have an analogous version of \reflemma{lemma:L^2estimatesForSquareFunctions} that only assumes \eqref{condition:L^1Carlesontypecond}, there are only two terms that need a closer look:
\begin{itemize}
    \item \(III_5\): In contrast to the proof for \eqref{cond:L1-Linfty}, we do have a pointwise uniform bound on \(M[\nabla\phi^s]\) not only on a good set but on all of \(3\Delta\) (see \reflemma{lemma:UniformBoundOnM[nablaphi^s]n=1}). Hence we have a pointwise everywhere bound on \(\nabla \theta\), and using this and the newly imposed \eqref{condition:L^1Carlesontypecond} we obtain the same bound for \(III_5\).
    \item \(III_2\): In contrast to the proof for \eqref{cond:L1-Linfty}, we cannot integrate in the \(x\)-direction first and expect good bounds on the \(L^2(3\Delta)\) norm of the quantities \(\nabla\partial_t\phi^t,w_s^{(1)}, w_s^{(2)}\). However, in the following we establish \eqref{lemma:w_s^1SqFctBound3n=1} in \reflemma{lemma:SqFctBoundsForpartial_sPhiAndw_s^1n=1}, and \eqref{lemma:w_s^2SqFctBound3n=1} which are the needed bounds for \(\partial_x\partial_s\phi^s, w_s^{(1)}\) and \(w_s^{(2)}\). Again, they essentially follow from the stronger pointwise everywhere bound \reflemma{lemma:UniformBoundOnM[nablaphi^s]n=1} and \eqref{condition:L^1Carlesontypecond}.
\end{itemize}

We can also note that the estimates in \reflemma{lemma:UniformBoundOnM[nablaphi^s]n=1} to  \reflemma{lemma:SqFctBoundsForw_tn=1} are only bounded for small boundary balls. Without loss of generality, we can restrict to showing \eqref{SquarefctBound} for small boundary balls and hence we assume for the following \(l(\Delta)\leq r_0\).

\subsection{Hodge decomposition}
The improvement for \(n=1\) comes from the Hodge decomposition which takes on an easier and more explicit form.
\smallskip

For each \(s>0\) we can find a Hodge decomposition consisting of \(\phi^s\in W_0^{1,2}(3\Delta)\) and \( h^s\in L^2(3\Delta)\), where \(h^s\) is divergence free and
\[c(x,s)\chi_{3\Delta}(x)=\AP(x,s)\partial_x\phi^s(x)+h^s(x).\]
Since this is a PDE in one dimension, if we write \(3\Delta=(a, a+3l(\Delta))\), the divergence free function is the constant 
\[h^s=\frac{\fint_{a}^{a+3l(\Delta)} \frac{c(x,s)}{\AP(x,s)}dx}{\fint_{a}^{a+3l(\Delta)} \frac{1}{\AP(x,s)}dx},\]
and 
\[\phi^s(y)=\int_{a}^y \frac{c(x,s)- h^s}{\AP(x,s)} dx\qquad \textrm{ for }y\in 3\Delta. \]

Hence we obtain the following uniform bounds on \(\partial_x \phi^s\) and \(\partial_s\partial_x\phi^s\):
\begin{lemma}\label{lemma:UniformBoundOnM[nablaphi^s]n=1}
There exists \(\kappa_0>0\) such that
\[M[|\partial_x\phi^s|](x)\leq C(\lambda,\Lambda_0)\leq \kappa_0,\]
and hence
\[\Vert M[\partial_x\phi^s]\Vert_{L^2}\leq C(\lambda,\Lambda_0)\leq \kappa_0.\]
Furthermore, we have for every \(x\in 3\Delta\)
\[|\partial_s\partial_x \phi^s(x)|\lesssim |\partial_s A(x,s)| + \fint_{3\Delta}|\partial_s A(y,s)| dy.\]
\end{lemma}

The relaxation of condition \eqref{cond:L1-Linfty} to \eqref{condition:L^1Carlesontypecond} comes now from the fact that all the pointwise estimates for \reflemma{lemma:UniformBoundOnM[nablaphi^s]} to \reflemma{lemma:UnifromBoundOnrholocally} not only hold on a good set \(F\) but on all of \(\Delta\). With this in hand, we can improve the \(L^2\)-area function bounds.

\subsection{\(L^2\) area function bounds}

The first lemma is analogous to \reflemma{lemma:SqFctBoundsForpartial_sPhiAndw_s^1}.

\begin{lemma}\label{lemma:SqFctBoundsForpartial_sPhiAndw_s^1n=1}
Assuming \eqref{condition:L^1Carlesontypecond} the following area function bounds hold
\begin{enumerate}[(i)]
    \item \(\int_{T(3\Delta)}\frac{|\partial_s\phi^s-w_s^{(1)}|^2}{s}dxds\lesssim |\Delta|;\)\label{lemma:w_s^1SqFctBound1n=1}
    \item \(\int_{T(3\Delta)}\frac{|s\partial_x\partial_s\phi^s|^2}{s}dxds+\int_{T(3\Delta)}\frac{|s\partial_x w_s^{(1)}|^2}{s}dxds\lesssim |\Delta|; \label{lemma:w_s^1SqFctBound2n=1}\) and 
    \item \(\int_{T(3\Delta)}|\partial_x\partial_s\phi^s|dxds + \int_{T(3\Delta)}|\partial_x w_s^{(1)}|dxds\lesssim |\Delta|.\label{lemma:w_s^1SqFctBound3n=1}\)
\end{enumerate}
\end{lemma}

\begin{proof}
   We only need to prove \eqref{lemma:w_s^1SqFctBound3n=1} since \eqref{lemma:w_s^1SqFctBound1n=1} and \eqref{lemma:w_s^1SqFctBound2n=1} follow as in \reflemma{lemma:SqFctBoundsForpartial_sPhiAndw_s^1} using  \reflemma{lemma:UniformBoundOnM[nablaphi^s]n=1} instead of \reflemma{lemma:UniformBoundOnM[nablaphi^s]}. 
   
   For \eqref{lemma:w_s^1SqFctBound3n=1}, we have by \reflemma{lemma:UniformBoundOnM[nablaphi^s]n=1} and \reflemma{lemma:UnifromBoundOnrholocally} 
    \[\int_{T(3\Delta)}|\partial_x\partial_s\phi^s|dxds\lesssim \int_{T(3\Delta)}|\partial_s\AP|dxds\lesssim |\Delta|,\]
    and for a scale \(k_0\) such that \(3l(\Delta)\leq 2^{-k_0} \leq 5l(\Delta)\)
    \begin{align*}
        \int_{T(3\Delta)}|\partial_x w_s^{(1)}|dxds&=\int_0^{3l(\Delta)}\sum_{k\geq k_0, Q\in\mathcal{D}_k(3\Delta)}\Big(\fint_Q |\partial_x e^{-\eta^2s^2L^s}\partial_s\phi^s|^2 dx\Big)^{1/2}|Q|
        \\
        &\lesssim \int_0^{3l(\Delta)}\sum_{k\geq k_0, Q\in\mathcal{D}_k(3\Delta)}\sup_{B(x,s,s/2)}|\partial_s\AP||Q|
        \\
        &\lesssim \int_{T(5\Delta)}\sup_{B(x,s,s/2)}|\partial_s\AP|\lesssim |\Delta|.
    \end{align*}
\end{proof}

We replace \reflemma{lemma:w_s^2/nablaw_s^2spatialL^2Bound} by
\begin{lemma}\label{lemma:w_s^2/nablaw_s^2spatialL^2Boundn=1}
    Assuming \eqref{condition:L^1Carlesontypecond} it holds that
\begin{align}
\int_{T(3\Delta)}\frac{|w_s^{(2)}|^2}{s}dxds\lesssim |\Delta|.\label{lemma:w_s^2SqFctBound1n=1}
\end{align}
\end{lemma}

\begin{proof}
    As in the proof of \reflemma{lemma:w_s^2/nablaw_s^2spatialL^2Bound} we obtain
    \begin{align*}
        &\int_{T(3\Delta)}\frac{|w_s^{(2)}|^2}{s}dxds
        \\
        &\lesssim \eta^2\int_0^{3l(\Delta)}\frac{1}{s}\Big(\int_0^{s}\frac{\tau}{\sqrt{s^2-\tau^2}}\Big(\int_{\mathbb{R}^n} |\partial_s A(x,s)\partial_x w(x,\eta^2\tau^2,s))|^2 dx\Big)^{1/2} d\tau\Big)^2 ds.
    \end{align*}
    
    Let us now fix a scale \(l\) such that \(2^{-l}\approx \tau\) and use off-diagonal estimates (\refprop{prop:off-diagonal}) to obtain 
    \begin{align*}
        &\Big(\int_{\mathbb{R}^n} |\partial_s A(x,s)\partial_x e^{-\eta^2\tau^2L^s}\phi^s(x)|^2 dx\Big)^{1/2}
        \\
        &=\sum_{k=1}^\infty \Big(\int_{2^k 3\Delta \setminus 2^{k-1}3\Delta}|\partial_s A(x,s)\partial_x e^{-\eta^2\tau^2L^s}\phi^s(x)|^2 dx\Big)^{1/2}
        \\
        &\qquad \qquad + \Big(\int_{3\Delta} |\partial_s A(x,s)\partial_x e^{-\eta^2\tau^2L^s}\phi^s(x)|^2 dx\Big)^{1/2}
        \\
        &=\sum_{k=1}^\infty \frac{1}{s}e^{-c\frac{2^{k} l(\Delta)^2}{\tau^2}}\Vert\phi^s\Vert_{L^2} + \Big(\sum_{\substack{Q\in\mathcal{D}_l(5\Delta) \\ l(Q)\approx 2^{-l}\approx \tau}} \sup_{x\in Q}|\partial_s A(x,s)|^2\int_Q|\partial_x e^{-\eta^2\tau^2L^s}\phi^s(x)|^2 dx\Big)^{1/2},
    \end{align*}
    where we made use of the pointwise bound of \(|\partial_s A|\lesssim \frac{1}{s}\). By Poincar\'{e} inequality, we obtain \(\Vert\phi^s\Vert_{L^2}\lesssim l(\Delta)\Vert\partial_x \phi^s\Vert_{L^2}\lesssim l(\Delta)\), whence
    
    \begin{align*}
        \sum_{k=1}^\infty \frac{l(\Delta)}{s}e^{-c\frac{2^{k} l(\Delta)^2}{\tau}}\leq \sum_{k=1}^\infty \frac{l(\Delta)}{\tau}e^{-c\frac{2^{k} l(\Delta)^2}{\tau^2}}\lesssim \frac{\tau}{l(\Delta)}\lesssim C.
    \end{align*}

    For the second term we apply\reflemma{lemma:UnifromBoundOnrholocally} and \reflemma{lemma:UniformBoundOnM[nablaphi^s]n=1} and obtain
    \begin{align*}
        &\Big(\sum_{Q\in\mathcal{D}_l(5\Delta)} \sup_{x\in Q}|\partial_s A(x,s)|^2\int_Q|\partial_x e^{-\eta^2\tau^2L^s}\phi^s(x)|^2 dx\Big)^{1/2}
        \\
        &\qquad\lesssim\Big(\sum_{Q\in\mathcal{D}_l(5\Delta)} \sup_{x\in Q}|\partial_s A(x,s)|^2\int_Q|M[\partial_x\phi^s]|^2 dx\Big)^{1/2}
        \\
        &\qquad\lesssim\Big(\int_{5\Delta}\sup_{x\in B(x,s,s/2)}|\partial_s A(x,s)|^2 dx\Big)^{1/2}.
    \end{align*}

    Hence in total we get
    \begin{align*}
        &\int_{T(3\Delta)}\frac{|w_s^{(2)}|^2}{s}dxds
        \\
        &\lesssim \eta^2\int_0^{3l(\Delta)}\frac{1}{s}\Big(\int_0^{s}\frac{\tau}{\sqrt{s^2-\tau^2}}\Big( \int_{5\Delta}\sup_{x\in B(x,s,s/2)}|\partial_s A(x,s)|^2 dx\Big)^{1/2} d\tau\Big)^2 ds
        \\
        &\lesssim \eta^2\int_0^{3l(\Delta)}s\int_{5\Delta}\sup_{x\in B(x,s,s/2)}|\partial_s A(x,s)|^2 dx ds
        \\
        &\lesssim \int_{T(5\Delta)}\sup_{x\in B(x,s,s/2)}|\partial_s A(x,s)|^2s dx ds
        \lesssim |\Delta|.
    \end{align*}   
\end{proof}

Recall the definition of \(v_2\) in Section \ref{section:rho}. We replace \reflemma{lemma:SqFctBoundsForw_s^2} by

\begin{lemma}\label{lemma:SqFctBoundsForw_s^2n=1}
Let \(Q\subset \partial\Omega\) be a boundary cube of size \(s\), and the index \(i\in \mathbb{Z}\) such that \(s\in [2^{-i},2^{-i+1})\). We have the following local bound involving \(w_s^{(2)}\)
\begin{align}
    \int_Q |\partial_x w_s^{(2)}(x,s)|^2dx&\lesssim \frac{1}{s^2}\fint_{2^{-i-1}}^{2^{-i+2}}\int_{2Q}|v_2(x,\eta^2 k^2;s)|^2dkdx\nonumber
    \\
    &\qquad+ \Vert\partial_s A\Vert_{L^\infty(2Q\times[2^{-i-3},2^{-i+4}])}^2|Q|\kappa_0^2.\label{lemma:localEstimateNablaw_s^2n=1}
\end{align}
As a consequence, if we assume \eqref{condition:L^1Carlesontypecond}, we have the area function bound
\begin{align}
   \int_{T(3\Delta)}|\partial_x w_s^{(2)}|^2sds\lesssim |\Delta|\label{lemma:w_s^2SqFctBound2n=1},
\end{align}
and the \(L^1\) version
\begin{align}
   \int_{T(3\Delta)}|\partial_x w_s^{(2)}|ds\lesssim |\Delta|.\label{lemma:w_s^2SqFctBound3n=1}
\end{align}

\end{lemma}

\begin{proof}

Applying \reflemma{lemma:CaccTypeForw_t} and \refcor{cor:PointwiseBoundsofw_tAndrhoByM(nablaphi^s)} on the second term, \reflemma{lemma:UnifromBoundOnrholocally} on the first and then \reflemma{lemma:UniformBoundOnM[nablaphi^s]n=1} yields
\[\fint_{2^{-i-3}}^{2^{-i+4}}\int_{2Q} |\partial_x e^{-\eta^2k^2L^s}\phi_s|^2 + s^2|\partial_x\partial_ke^{-\eta^2k^2L^s}\phi_s|^2dxdk\lesssim |Q|\inf_{x\in Q} M[\partial_x\phi^s](x)\lesssim \kappa_0|Q|.\]
Combining this with the Caccioppoli type inequality \reflemma{Lemma:CacciopolliTypeInequalities} implies \eqref{lemma:localEstimateNablaw_s^2n=1}. 

\medskip
The proof of \eqref{lemma:w_s^2SqFctBound2n=1} simplifies compared to the proof of \eqref{lemma:w_s^2SqFctBound2}. We note that \refprop{prop:L2NormBoundsOfHeatSemigroup} implies
\begin{align*}
    \int_{3\Delta}|v_2(x,\eta^2k^2;s)|^2dx&\leq\Big(\int_0^k2\eta^2\tau \Vert e^{-\eta^2(k^2-\tau^2)L^s}\partial_x(\partial_s\AP\partial_x e^{-\eta^2\tau^2L^s}\phi^s)\Vert_{L^2} d\tau\Big)^2
    \\
    &\leq\Big(\int_0^k \frac{2\eta^2\tau}{\sqrt{k^2-\tau^2}} \Vert\partial_s\AP\partial_x e^{-\eta^2\tau^2L^s}\phi^s\Vert_{L^2} d\tau\Big)^2.
\end{align*}
By \reflemma{lemma:UnifromBoundOnrholocally} and \reflemma{lemma:UniformBoundOnM[nablaphi^s]n=1} we obtain for the scale \(i\in\mathbb{Z}\) with \(2^{-i}\approx s\)
\begin{align*}
    \Vert\partial_s\AP\partial_x e^{-\eta^2\tau^2L^s}\phi^s\Vert_{L^2}&\lesssim\Big(\sum_{Q\in \mathcal{D}_i(5\Delta)}\sup_Q|\partial_s\AP|^2|Q| \kappa_0^2\Big)^{1/2}
    \\
    &\hspace{30mm}+ \frac{1}{s}\Vert\partial_x e^{-\eta^2\tau^2L^s}\phi^s\Vert_{L^2(\mathbb{R}^n\setminus 5\Delta)}.
\end{align*}
Off-diagonal estimates (\refprop{prop:off-diagonal}), Poincar\'{e}'s inequality and the pointwise bound \(|\partial_tA|t\leq C\) yield
\begin{align*}
    \Vert\partial_s\AP\partial_x e^{-\eta^2\tau^2L^s}\phi^s\Vert_{L^2}&\lesssim\Big(\int_{5\Delta} \sup_{B(x,s,s/2)}|\partial_s\AP|^2dx\Big)^{1/2} + \frac{1}{s}\frac{1}{\tau}e^{-c\frac{l(\Delta)^2}{\tau^2}}\Vert\phi^s\Vert_{L^2(\mathbb{R}^n\setminus 5\Delta)}
    \\
    &\lesssim\Big(\int_{5\Delta} \sup_{B(x,s,s/2)}|\partial_s\AP|^2dx\Big)^{1/2} + \frac{1}{l(\Delta)^2}\Vert\phi^s\Vert_{L^2(\mathbb{R}^n\setminus 5\Delta)}
    \\
    &\lesssim\Big(\int_{5\Delta} \sup_{B(x,s,s/2)}|\partial_s\AP|^2dx\Big)^{1/2} + \frac{1}{\sqrt{l(\Delta)}}.
\end{align*}
This gives 
\begin{align*}
    \int_{3\Delta}|v_2(x,\eta^2k^2
    ;s)|^2dx&\leq s^2\int_{5\Delta} \sup_{B(x,s,s/2)}|\partial_s\AP|^2dx + \frac{s^2}{l(\Delta)} .
\end{align*}
Hence in total, we obtain
\begin{align*}
    \int_{T(3\Delta)} |\partial_x w_s^{(2)}(x,s)|^2 sdxds&\lesssim \int_{T(3\Delta)}2\sup_{B(x,s,s/2)}|\partial_s\AP|^2sdxds + \int_0^{l(\Delta)}\frac{s}{l(\Delta)}ds
    \\
    &\lesssim |\Delta|,
\end{align*}
which completes the proof of \eqref{lemma:w_s^2SqFctBound2n=1}.

\medskip

Lastly for \eqref{lemma:w_s^2SqFctBound3n=1}, we first fix \(s>0\) and let \(i\in\mathbb{Z}\) be the scale such that \(2^{-i}\approx s\approx k\). We would like to consider the collection of boundary cubes \(\mathcal{D}_i(4\Delta)\), but need to adjust and clarify the structure of this collection to make the following argument rigorous. Without loss of generality assume that \(0\) is the center point of \(\Delta\). Let \(\beta=(\beta_1,...,\beta_n)\in \mathbb{Z}^n\) now be a multi index and set \(Q_\beta:=\Delta(2^{-i-2}\beta,2^{-i-1})\). Then the collection of \(\{Q_\beta\}_{\beta\in \mathbb{Z}^n}\) covers \(\mathbb{R}^n\) and the collection has finite overlap, i.e. there exists \(N\in\mathbb{N}\) with \(\sum_{\beta\in\mathbb{Z}^n}\chi_{Q_\beta}\leq N\). Furthermore, we can note that the finite subcollection \(\{Q_\beta\}_{|\beta|_2\leq 2^{i+4}l(\Delta) }\) covers \(4\Delta\) and each of those \(Q_\beta\subset 5\Delta\).

Now, we can use a family of dualising functions \(h_\beta\in L^2(Q_\beta), \Vert h_\beta\Vert_{L^2}=1\) and integration by parts to estimate

\begin{align*}
    &\Big(\int_{Q_\beta}|v_2(x,\eta^2k^2;s)|^2dx\Big)^{1/2}
    \\
    &\leq \int_0^k2\eta^2\tau \Vert e^{-\eta^2(k^2-\tau^2)L^s}\partial_x(\partial_s\AP\partial_x e^{-\eta^2\tau^2L^s}\phi^s)\Vert_{L^2(Q_\beta)} d\tau
    \\
    &\lesssim\int_0^k\tau \int_{\mathbb{R}^n}\partial_s\AP\partial_x e^{-\eta^2\tau^2L^s}\phi^s  \cdot\partial_x e^{-\eta^2(k^2-\tau^2)L^s} h_\beta dx d\tau
    \\
    &\leq\int_0^k\tau\Big( \int_{4\Delta}\partial_s\AP\partial_x e^{-\eta^2\tau^2L^s}\phi^s  \cdot\partial_x e^{-\eta^2(k^2-\tau^2)L^s} h_\beta dx
    \\
    &\qquad + \sum_{j\geq 2} \int_{2^{j+1}\Delta\setminus 2^j\Delta}\partial_s\AP\partial_x e^{-\eta^2\tau^2L^s}\phi^s  \cdot\partial_x e^{-\eta^2(k^2-\tau^2)L^s} h_\beta dx\Big) d\tau
    \\
    &=:\int_0^k \tau \Big(I^\beta+\sum_{j\geq 2}J_j^\beta \Big)d\tau.
\end{align*}
For the first terms \(I^\beta\) and \(J_2^\beta\)  we have by \reflemma{lemma:UnifromBoundOnrholocally}, Cauchy-Schwarz inequality and \(L^2-L^2\) off-diagonal estimates (\refprop{prop:off-diagonal})
\begin{align*}
    I^\beta, J_2^\beta&\leq\int_{8\Delta}|\partial_s\AP\partial_x e^{-\eta^2\tau^2L^s}\phi^s  \cdot\partial_x e^{-\eta^2(k^2-\tau^2)L^s} h_\beta| dx
    \\
    &\leq\sum_{|\alpha|\leq 2^{l+6}l(\Delta)}\int_{Q_\alpha}|\partial_s\AP\partial_x e^{-\eta^2\tau^2L^s}\phi^s  \cdot\partial_x e^{-\eta^2(k^2-\tau^2)L^s} h_\beta| dx
    \\
    &\leq\sum_{|\alpha|\leq 2^{l+6}l(\Delta)}\sup_{Q_\alpha}|\partial_s\AP| \Vert\partial_x e^{-\eta^2\tau^2L^s}\phi^s\Vert_{L^2(Q_\alpha)}  \Vert\partial_x e^{-\eta^2(k^2-\tau^2)L^s} h_\beta\Vert_{L^2(Q_\alpha)}
    \\
    &\lesssim \sum_{|\alpha|\leq 2^{l+6}l(\Delta)}\sup_{Q_\alpha} |\partial_s\AP|\kappa_0\sqrt{|Q_\alpha|}\frac{1}{\sqrt{k^2-\tau^2}}e^{-c\frac{\mathrm{dist}(Q_\alpha,Q_\beta)^2}{k^2-\tau^2}}\Vert h_\beta\Vert_{L^2}
    \\
    &\lesssim \sum_{|\alpha|\leq 2^{l+6}l(\Delta)} \sup_{Q_\alpha} |\partial_s\AP|\kappa_0\frac{\sqrt{|Q_\alpha|}}{\sqrt{k^2-\tau^2}}e^{-c|\alpha-\beta|_2^2}.
\end{align*}
Here we used in the last line that if \(0\leq \tau\leq k\) then \(e^{-c\frac{\mathrm{dist}(Q_\alpha,Q_\beta)^2}{k^2-\tau^2}}=e^{-c\frac{|\alpha-\beta|_2^2s^2}{k^2-\tau^2}}\leq e^{-c|\alpha-\beta|_2^2s^2}\) and the Euclidean length of a vector \(|\alpha|=|\alpha|_2=\sqrt{\alpha_1^2+...+\alpha_n^2}\).

Summing this over \(\beta\) yields
\begin{align*}
    \sum_{|\beta|\leq 2^{l+4}l(\Delta)}I^\beta+J_2^\beta&\leq\sum_{|\beta|\leq 2^{l+4}l(\Delta)} \int_{6\Delta}|\partial_s\AP\partial_x e^{-\eta^2\tau^2L^s}\phi^s  \cdot \partial_x e^{-\eta^2(k^2-\tau^2)L^s} h_Q| dx
    \\
    &\lesssim \sum_{|\alpha|\leq 2^{l+6}l(\Delta)} \sup_{Q_\alpha} |\partial_s\AP|\kappa_0\frac{\sqrt{|Q_\alpha|}}{\sqrt{k^2-\tau^2}}\sum_{|\beta|\leq 2^{l+4}l(\Delta)} e^{-c|\alpha-\beta|_2^2}
    \\
    &\lesssim \sum_{|\alpha|\leq 2^{l+6}l(\Delta)} \sup_{Q_\alpha} |\partial_s\AP|\kappa_0\frac{\sqrt{|Q_\alpha|}}{\sqrt{k^2-\tau^2}}.
\end{align*}

For the remaining terms \(J_j^\beta, j\geq 3\) we have by \(L^2-L^2\) off-diagonal estimates (\refprop{prop:off-diagonal}), \reflemma{lemma:nablaSemigroupBoundedByNablaf}, and \reflemma{lemma:UniformBoundOnM[nablaphi^s]n=1}
\begin{align*}
    J^\beta_j=&\int_{2^{j+1}\Delta\setminus 2^j\Delta}\partial_s\AP\partial_x e^{-\eta^2\tau^2L^s}\phi^s  \cdot\partial_x e^{-\eta^2(k^2-\tau^2)L^s} h_Q dx
    \\
    &\lesssim \frac{1}{s}\Vert \partial_x e^{-\eta^2\tau^2L^s}\phi^s\Vert_{L^2(2^{j+1}\Delta\setminus 2^j\Delta)}\Vert\partial_x e^{-\eta^2(k^2-\tau^2)L^s} h_Q\Vert_{L^2(2^{j+1}\Delta\setminus 2^j\Delta)}
    \\
    &\lesssim \frac{1}{s}\Vert\partial_x\phi^s\Vert_{L^2}\frac{1}{\sqrt{k^2-\tau^2}}e^{-c\frac{2^{2j}l(\Delta)^2}{k^2-\tau^2}}\Vert h_Q\Vert_{L^2(Q)}
    \\
    &\lesssim \frac{1}{s}|\Delta|^{1/2}\frac{1}{\sqrt{k^2-\tau^2}}e^{-c\frac{2^{2j}l(\Delta)^2}{k^2-\tau^2}}\Vert h_Q\Vert_{L^2(Q)}.
\end{align*}
Here we used that \(h_Q\) is supported away from the annulus \(2^{j+1}\Delta\setminus 2^j\Delta\).

Before we sum over \(j\), we see that \(r\mapsto \frac{1}{r}e^{-c\frac{2^{2j}l(\Delta)^2}{r^2}}\) is maximized for \(r=2^jl(\Delta)\), whence
\begin{align*}
    \sum_{j\geq 3}J_\beta^j&\lesssim \sum_{j\geq 3}\frac{1}{s}|\Delta|^{1/2}\frac{1}{\sqrt{k^2-\tau^2}}e^{-c\frac{2^{2j}l(\Delta)^2}{k^2-\tau^2}}
    \\
    &\lesssim \sum_{j\geq 3}\frac{1}{s}|\Delta|^{1/2}\frac{1}{2^jl(\Delta)}
    \\
    &\lesssim \frac{1}{s\sqrt{l(\Delta)}}. 
\end{align*}

Finally, we get in total with \eqref{lemma:localEstimateNablaw_s^2n=1}
\begin{align*}
    &\int_{T(\Delta)}|\partial_x w_s^{(2)}| dxds\lesssim\sum_{i\in \mathbb{Z}, 2^{-i}\leq 3l(\Delta)}\int_{2^{-i-1}}^{2^{-i}}\int_\Delta|\partial_x w_s^{(2)}| dxds
    \\
    &\lesssim\sum_{i\in \mathbb{Z}, 2^{-i}\leq 3l(\Delta)}\int_{2^{-i-1}}^{2^{-i}}\sum_{Q_\beta\in \mathcal{D}_l(3\Delta)}\Big(\int_{Q_\beta}|\partial_x w_s^{(2)}|^2dx\Big)^{\frac{1}{2}}|Q_\beta|^\frac{1}{2} ds
    \\
    &\lesssim\sum_{i\in \mathbb{Z}, 2^{-i}\leq 3l(\Delta)}\int_{2^{-i-1}}^{2^{-i}}\sum_{Q_\beta\in \mathcal{D}_l(3\Delta)}\big(\frac{1}{s}\fint_{2^{-i-1}}^{2^{-i}}\int_{Q_\beta} |v_2(x,\eta^2k^2;s)|^2dxdk
    \\
    &\hspace{50mm}+ \sup_{B(x,s,s/2)}|\partial_s\AP||Q_\beta|^{1/2}\big)|Q_\beta|^\frac{1}{2} ds
\end{align*}
which allows us to finish the proof of \eqref{lemma:w_s^2SqFctBound3n=1} by bounding above by
\begin{align*}
    &\lesssim\sum_{i\in \mathbb{Z}, 2^{-i}\leq 3l(\Delta)}\int_{2^{-i-1}}^{2^{-i}}\sum_{Q_\beta\in \mathcal{D}_l(3\Delta)}\fint_{2^{-i-1}}^{2^{-i}}\int_0^k \frac{\tau}{s} \Big(I^\beta+\sum_{j\geq 2}J_j^\beta \Big)|Q_\beta|^{1/2}d\tau dk ds \\
    &\qquad + \sum_{i\in \mathbb{Z}, 2^{-i}\leq 3l(\Delta)}\int_{2^{-i-1}}^{2^{-i}}\sum_{Q_\beta\in \mathcal{D}_l(3\Delta)} \sup_{B(x,s,s/2)}|\partial_s\AP||Q_\beta| ds
    \\
    &\lesssim\sum_{i\in \mathbb{Z}, 2^{-i}\leq 3l(\Delta)}\Big(\int_{2^{-i-1}}^{2^{-i}}\sum_{Q_\alpha\in \mathcal{D}_l(6\Delta)} \sup_{Q_\alpha} |\partial_s\AP||Q_\alpha| \big(\fint_{2^{-i-1}}^{2^{-i}}\int_0^k\frac{1}{\sqrt{k^2-\tau^2}}d\tau dk\big) ds
    \\
    &\qquad+ \int_{2^{-i-1}}^{2^{-i}}\fint_{2^{-i-1}}^{2^{-i}}\int_0^k\frac{\tau \sqrt{s}}{s^2\sqrt{l(\Delta)}}d\tau dk ds\Big) + \int_{T(5\Delta)}\sup_{B(x,s,s/2)}|\partial_s\AP| dxdsdkds
    \\
    &\lesssim\int_{T(10\Delta)}\sup_{B(x,s,s/2)}|\partial_s\AP| dx ds + 6l(\Delta)\lesssim |\Delta|.
\end{align*}

\end{proof}

Next, we have the analog of \reflemma{lemma:SqFctBoundsForw_t}.
\begin{lemma}\label{lemma:SqFctBoundsForw_tn=1}
The following area function bounds involving \(w_t\) hold:
\begin{enumerate}[(i)]
    \item \(\int_{T(\Delta)}\frac{|w_t|^2}{s}dxds=\int_{T(\Delta)}\frac{|sL^se^{-\eta^2s^2L^s}\phi^s|^2}{s}dxds\lesssim |\Delta|;\)\label{lemma:w_tSqFctBound1n=1}
    \item \(\int_{T(\Delta)}\frac{|s\partial_x w_t|^2}{s}dxds\lesssim |\Delta|;\label{lemma:w_tSqFctBound2n=1}\)
    \item \(\int_{T(\Delta)}\frac{|s^2L^sw_t|^2}{s}dxds\lesssim |\Delta|.\label{lemma:w_tSqFctBound3n=1}\)
\end{enumerate}
\end{lemma}

\begin{proof}
We start with proving \eqref{lemma:w_tSqFctBound1n=1}. In contrast to \eqref{lemma:w_tSqFctBound1} in \reflemma{lemma:SqFctBoundsForw_t} the \(1\) times \(1\) matrix \(\AP\) is always symmetric, which allows to write \(\AP(x,s)=B(x,s)^2\) for \(B>0\). And hence the proof simplifies and we obtain
\begin{align*}
    \int_{T(\Delta)}\frac{|sL^se^{-\eta^2s^2L^s}\phi^s|^2}{s}dxds&\leq\int_{\mathbb{R}^n\times(0,l(\Delta))}\frac{|sL^se^{-\eta^2s^2L^s}\phi^s|^2}{s}dxds
    \\
    &=\int_{\mathbb{R}^n\times (0,l(\Delta))} B\partial_x L^{s}e^{-\eta^2s^2L^s}\phi^s \cdot B\partial_x e^{-\eta^2s^2L^s}\phi^s s dxds
    \\
    &=\int_{\mathbb{R}^n\times (0,l(\Delta))} B \partial_x \partial_s(e^{-\eta^2s^2L^s}\phi^s) \cdot B\partial_x e^{-\eta^2s^2L^s}\phi^s dxds
    \\
    &\qquad - \int_{\mathbb{R}^n\times (0,l(\Delta))} B \partial_x w_s^{(1)} \cdot B\partial_x e^{-\eta^2s^2L^s}\phi^s dxds
    \\
    &\qquad - \int_{\mathbb{R}^n\times (0,l(\Delta))} B \partial_x w_s^{(2)} \cdot B\partial_x e^{-\eta^2s^2L^s}\phi^s dxds 
    \\
    &=\int_{0} ^{l(\Delta)}\partial_s\Vert B\partial_x e^{-\eta^2s^2L^s}\phi^s\Vert_{L^2(\mathbb{R}^n)}^2 dxds
    \\
    &\qquad + \int_{\mathbb{R}^n\times (0,l(\Delta))} \partial_s\AP \partial_x e^{-\eta^2s^2L^s}\phi^s \cdot \partial_x e^{-\eta^2s^2L^s}\phi^s dxds
    \\
    &\qquad - \int_{\mathbb{R}^n\times (0,l(\Delta))} B \partial_x w_s^{(1)} \cdot B\partial_x e^{-\eta^2s^2L^s}\phi^s dxds
    \\
    &\qquad - \int_{\mathbb{R}^n\times (0,l(\Delta))} B \partial_x w_s^{(2)} \cdot B\partial_x e^{-\eta^2s^2L^s}\phi^s dxds 
    \\
    &=:I+II+III+IV.
\end{align*}

Looking at the integrals separately, first we have by \reflemma{lemma:nablaSemigroupBoundedByNablaf} and \reflemma{lemma:UniformBoundOnM[nablaphi^s]n=1}
\[I=\Vert B\partial_x e^{-\eta^2(l(\Delta))^2L^{l(\Delta)}}\phi^{l(\Delta)}\Vert_{L^2(\mathbb{R}^n)}^2 - \Vert B\partial_x \phi^0\Vert_{L^2(\mathbb{R}^n)}^2 \lesssim |\Delta|.\]
For the next integral we use duality between Carleson measure and nontangential maximal function to get
\[II\lesssim \Vert \sup_{B(x,s,s/2)} |\partial_s \AP|\Vert_{\mathcal{C}}\int_{\mathbb{R}^n} \tilde{N}^{l(\Delta)}(|\partial_x e^{-\eta^2s^2L^s}\phi^s|)^2 dx. \]
From \reflemma{lemma:UniformBoundOnM[nablaphi^s]n=1} we have
\[\int_{8\Delta}\tilde{N}^{l(\Delta)}(|\partial_x e^{-\eta^2s^2L^s}\phi^s|)^2 dx\lesssim |\Delta|.\]

Since for a cone that is truncated at height \(l(\Delta)\) with aperture \(1\) and tip in \(2^j\Delta\setminus 2^{j-1}\Delta\) all points of this cone lie in \(2^{j+1}\Delta\setminus 2^{j-2}\Delta\) and the nontangential maximal function here is truncated at height \(l(\Delta)\), we can use off-diagonal estimates (\refprop{prop:off-diagonal}) and Poincar\'{e}'s inequality to obtain for the away part
\begin{align*}
    \int_{\mathbb{R}^n\setminus 8\Delta}\tilde{N}(|\partial_x e^{-\eta^2s^2L^s}\phi^s|)^2 dx
    &\lesssim \sum_{j\geq 4}\int_{2^j \Delta\setminus 2^{j-1}\Delta}\tilde{N}(|\partial_x e^{-\eta^2s^2L^s}\phi^s|)^2 dx
    \\
    & \lesssim \sum_{j\geq 4}\int_{2^{j+1} \Delta\setminus 2^{j-2}\Delta}\sup_{0<s\leq 3l(\Delta)}\frac{1}{s^3}e^{-c\frac{2^{2j}l(\Delta)^2}{s^2}}\Vert \phi^s\Vert_{L^2}^2 dx
    \\
    & \lesssim \sum_{j\geq 4}\int_{2^{j+1} \Delta\setminus 2^{j-2}\Delta}\frac{1}{2^{3j}}\frac{1}{l(\Delta)^3}\Vert \phi^s\Vert_{L^2}^2 dx
    \\
    & \lesssim \sum_{j\geq 4}\int_{2^{j+1} \Delta\setminus 2^{j-2}\Delta}\frac{1}{2^{3j}} dx
    \\
    &\lesssim |\Delta|.
\end{align*}
Next, we have by integration by parts, \(L^s=(L^s)^*\) and \reflemma{lemma:UniformBoundOnM[nablaphi^s]n=1} 
\begin{align*}
    III&= \int_0^{l(\Delta)}\int_{\mathbb{R}^n}w_s^{(1)}(x,s)\cdot L^se^{-\eta^2s^2L^s}\phi^s(x) dx ds
    \\
    &= \int_0^{l(\Delta)}\int_{\mathbb{R}^n}e^{-\eta^2s^2L^s}(\partial_x\AP(x,s)\partial_x\partial_s\phi^s)(x) \cdot e^{-\eta^2s^2L^s}\phi^s(x) dx ds
    \\
    &= \int_0^{l(\Delta)}\int_{\mathbb{R}^n}\AP(x,s)\partial_x\partial_s\phi^s(x)\cdot\partial_xe^{-2\eta^2s^2L^s}\phi^s(x) dx ds
    \\
    &\leq\int_0^{l(\Delta)}\sum_{\substack{Q\in \mathcal{D}_i(Q)\\ 2^{-i}\leq s\leq 2^{-i+1}}} \sup_{y\in Q} \big(|\partial_s\AP(y,s)| + \fint_{3\Delta}|\partial_s A|dx\big)\big(\int_Q |\partial_x e^{-2\eta^2s^2L^s}\phi^s|^2\big)^{1/2}ds
    \\
    &\leq\int_{T(5\Delta)}\sup_{(y,t)\in B(x,s,s/2)}|\partial_s A(y,t)| dxds\lesssim |\Delta|.
\end{align*}

At last, we have
\begin{align*}
    |IV|&\leq\int_0^{l(\Delta)}\int_0^s\int_{\mathbb{R}^n}|2\tau\partial_s\AP\partial_x e^{-\eta^2\tau^2L^s}\phi^s\cdot \partial_x e^{-\eta^2(s^2-\tau^2)L^s}L^se^{-\eta^2 s^2L^s}\phi^s| dx dt ds
    \\
    &\leq \Vert \sup_{B(x,s,s/2)} |\partial_s A|\Vert_{\mathcal{C}}
    \\
    &\qquad\cdot\int_{\mathbb{R}^n}\tilde{N}^{l(\Delta)}\Big(\int_0^s \tau|\partial_x e^{-\eta^2\tau^2L^s}\phi^s| |\partial_x e^{-\eta^2(s^2-\tau^2)L^s}L^se^{-\eta^2 s^2L^s}\phi^s|d\tau\Big)(x) dx
\end{align*}

We note that the appearing nontangential maximal function will be dealt with similarly to before where we split the integration over all of \(\mathbb{R}^n\) into a local and an away part. First let us look at the local part and take \(x\in 10\Delta, t>0\). We have
\begin{align}
    &\fint_{B(x,t,t/2)} \int_0^s \tau|\partial_x e^{-\eta^2\tau^2L^s}\phi^s| |\partial_x e^{-\eta^2(s^2-\tau^2)L^s}L^se^{-\eta^2 s^2L^s}\phi^s|d\tau dxds\nonumber
    \\
    &\lesssim \fint_{t/2}^{3t/2} \int_0^s \tau\Big(\fint_{B(x,t/2)} |\partial_x e^{-\eta^2\tau^2L^s}\phi^s|^2dx\Big)^{1/2} \nonumber
    \\
    &\qquad\qquad \qquad \cdot\Big(\fint_{B(x,t/2)}|\partial_x e^{-\eta^2(s^2-\tau^2)L^s}L^se^{-\eta^2 s^2L^s}\phi^s|^2 dx\Big)^{1/2} d\tau ds\label{eq:splitinnontangential}
\end{align}  

To bound the expression \(\Big(\fint_{B(x,t/2)} |\partial_x e^{-\eta^2\tau^2L^s}\phi^s|^2dx\Big)^{1/2}\) we cannot directly apply \reflemma{lemma:UnifromBoundOnrholocally} since \(t\) and \(\tau\) are not comparable. Instead let \(l\in\mathbb{Z}\) be the scale with \(2^{-l}\approx\tau\), then 
\begin{align*}
    \fint_{B(x,t/2)} |\partial_x e^{-\eta^2\tau^2L^s}\phi^s|^2dx&\lesssim \frac{1}{t}\sum_{Q\in \mathcal{D}_l(B(x,t/2))}\int_Q |\partial_x e^{-\eta^2\tau^2L^s}\phi^s|^2dx
    \\
    &\lesssim \frac{\tau}{t}\sum_{Q\in \mathcal{D}_l(B(x,t/2))}\kappa_0^2.
\end{align*}
Since \(\# \mathcal{D}_l(B(x,t/2))\approx \frac{t}{\tau}\), we obtain \(\Big(\fint_{B(x,t/2)} |\partial_x e^{-\eta^2\tau^2L^s}\phi^s|^2dx\Big)^{1/2}\lesssim \kappa_0\).

For the second expression \(\Big(\fint_{B(x,t/2)}|\partial_x e^{-\eta^2(s^2-\tau^2)L^s}L^se^{-\eta^2 s^2L^s}\phi^s|^2 dx\Big)^{1/2}\) an analogous argument gives
\begin{align*}
    \Big(\fint_{B(x,t/2)}|\partial_x e^{-\eta^2(s^2-\tau^2)L^s}L^se^{-\eta^2 s^2L^s}\phi^s|^2 dx\Big)^{1/2}\lesssim \Big(\fint_{B(x,t/2)}|\partial_x L^se^{-\eta^2 s^2L^s}\phi^s|^2 dx\Big)^{1/2}.
\end{align*}
Further we can bound
\begin{align*}
    &\fint_{B(x,t/2)}|\partial_x L^se^{-\eta^2 s^2L^s}\phi^s|^2 dx
    \\
    &\lesssim \Big(\fint_{B(x,t)}|L^s L^se^{-\eta^2 s^2L^s}\phi^s|^2 dx\Big)^{1/2} \Big(\fint_{B(x,t)}|L^se^{-\eta^2 s^2L^s}\phi^s|^2 dx\Big)^{1/2}
    \\
    &\lesssim \frac{1}{s^{4}}\sup_{y\in B(x,t)}M[\partial_x\phi^s](y)\lesssim \frac{\kappa_0}{s^{4}},
\end{align*}
where we used \refcor{cor:PointwiseBoundsofw_tAndrhoByM(nablaphi^s)}.

Hence in total for \(x\in 10\Delta,t>0\)
\begin{align*}\fint_{B(x,t,t/2)} &\int_0^s \tau|\partial_x e^{-\eta^2\tau^2L^s}\phi^s| |\partial_x e^{-\eta^2(s^2-\tau^2)L^s}L^se^{-\eta^2 s^2L^s}\phi^s|d\tau dxds
\\
&\qquad\lesssim  \fint_{t/2}^{3t/2}\int_0^s\tau \frac{1}{s^{2}}d\tau ds\leq C, \end{align*}

and

\begin{align*}
    &\int_{8\Delta}\tilde{N}^{l(\Delta)}\Big(\int_0^s \tau|\partial_x e^{-\eta^2\tau^2L^s}\phi^s| |\partial_x e^{-\eta^2(s^2-\tau^2)L^s}L^se^{-\eta^2 s^2L^s}\phi^s|d\tau\Big)(x) dx
    \lesssim |\Delta|.
\end{align*}

For the away part, we can equally do the step \eqref{eq:splitinnontangential}. For \(x\in 2^j \Delta\setminus 2^{j-1}\Delta\subset\mathbb{R}^n\setminus 4\Delta\) with \(0<t\leq 3l(\Delta)\) and \(j\geq 3\) however, we obtain with off-diagonal estimates (\refprop{prop:off-diagonal})
\begin{align*}
    \Big(\fint_{B(x,t/2)} |\partial_x e^{-\eta^2\tau^2L^s}\phi^s|^2dx\Big)^{1/2}&\lesssim \frac{1}{\sqrt{t}}\frac{1}{\tau}e^{-c\frac{2^{2j}l(\Delta)^2}{\tau^2}}\Vert \phi^s\Vert_{L^2(3\Delta)}\lesssim  \frac{1}{\sqrt{t}}\frac{1}{2^{j}l(\Delta)}\Vert \phi^s\Vert_{L^2(3\Delta)}
    \\
    &\lesssim  \frac{1}{2^j\sqrt{t}}\Vert \partial_x\phi^s\Vert_{L^2(3\Delta)}\lesssim  \frac{1}{2^j\sqrt{t}}\sqrt{|\Delta|},
\end{align*}
where we used that \(\tau\mapsto \frac{1}{\tau}e^{-c\frac{2^{2j}l(\Delta)^2}{\tau^2}}\) is maximized for \(\tau=2^jl(\Delta)\) and Poincar\'{e}'s inequality. 

For the second expression we can split the integral into
\begin{align*}
&\Big(\fint_{B(x,t/2)}|\partial_x e^{-\eta^2(s^2-\tau^2)L^s}L^se^{-\eta^2 s^2L^s}\phi^s|^2 dx\Big)^{1/2}
\\
&\leq\Big(\fint_{B(x,t/2)}|\partial_x e^{-\eta^2(s^2-\tau^2)L^s}(\chi_{2^{j-2}\Delta}L^se^{-\eta^2 s^2L^s}\phi^s)|^2 dx\Big)^{1/2}  
\\
&\qquad + \Big(\fint_{B(x,t/2)}|\partial_x e^{-\eta^2(s^2-\tau^2)L^s}(\chi_{\mathbb{R}^n\setminus 2^{j-2}\Delta}L^se^{-\eta^2 s^2L^s}\phi^s)|^2 dx\Big)^{1/2}.
\end{align*}
By off-diagonal estimates (\refprop{prop:off-diagonal}), maximizing \(r\mapsto e^{-c\frac{2^{2j-4}l(\Delta)^2}{r^2}}\), and \refcor{cor:PointwiseBoundsofw_tAndrhoByM(nablaphi^s)} we obtain for the first more local part
\begin{align*}
    &\Big(\fint_{B(x,t/2)}|\partial_x e^{-\eta^2(s^2-\tau^2)L^s}(\chi_{2^{j-2}l(\Delta)}L^se^{-\eta^2 s^2L^s}\phi^s)|^2 dx\Big)^{1/2}
    \\
    &\lesssim \frac{1}{\sqrt{t}}\frac{1}{\sqrt{s^2-\tau^2}}e^{-c\frac{2^{2j-4}l(\Delta)^2}{s^2-\tau^2}}\Vert L^se^{-\eta^2 s^2L^s}\phi^s\Vert_{L^2}
    \\
    &\lesssim \frac{1}{\sqrt{t}}\frac{1}{2^jl(\Delta)}\frac{1}{s}\Vert M[\partial_x\phi^s]\Vert_{L^2(3\Delta)}
    \\
    &\lesssim \frac{1}{2^{j}\sqrt{t}}\frac{1}{s}\frac{1}{\sqrt{|\Delta|}}\kappa_0.
\end{align*}
For the second away part, we use \reflemma{lemma:nablaSemigroupBoundedByNablaf} and off-diagonal estimates (\refprop{prop:off-diagonal}) to deduce
\begin{align*}
    &\Big(\fint_{B(x,t/2)}|\partial_x e^{-\eta^2(s^2-\tau^2)L^s}(\chi_{\mathbb{R}^n\setminus 2^{j-2}l(\Delta)}L^se^{-\eta^2 s^2L^s}\phi^s)|^2 dx\Big)^{1/2}
    \\
    &\lesssim \frac{1}{\sqrt{t}}\frac{1}{\sqrt{s^2-t^2}}\Vert L^se^{-\eta^2s^2L^s}\phi^s\Vert_{L^2(\mathbb{R}^n\setminus 2^{j-2}\Delta)}
    \\
    &\lesssim \frac{1}{\sqrt{t}}\frac{1}{\sqrt{s^2-\tau^2}}\frac{1}{s^2}e^{-c\frac{2^{2j-4}l(\Delta)^2}{s^2}}\Vert \phi^s\Vert_{L^2}
    \\
    &\lesssim \frac{1}{\sqrt{t}}\frac{1}{\sqrt{s^2-\tau^2}}\frac{1}{2^{2j}l(\Delta)^2}\Vert \phi^s\Vert_{L^2(3\Delta)}
    \\
    &\lesssim \frac{1}{\sqrt{t}}\frac{1}{\sqrt{s^2-\tau^2}}\frac{1}{2^{2j}l(\Delta)}\Vert \partial_x\phi^s\Vert_{L^2(3\Delta)}
    \\
    &\lesssim \frac{1}{2^{2j}\sqrt{t}}\frac{1}{\sqrt{s^2-\tau^2}}\frac{1}{\sqrt{|\Delta|}}\kappa_0.
\end{align*}

Hence we obtain for \(x\in 2^j \Delta\setminus 2^{j-1}\Delta\)
\begin{align*}\fint_{B(x,t,t/2)} &\int_0^s \tau|\partial_x e^{-\eta^2\tau^2L^s}\phi^s| |\partial_x e^{-\eta^2(s^2-\tau^2)L^s}L^se^{-\eta^2 s^2L^s}\phi^s|d\tau dxds
\\
&\qquad\lesssim  \fint_{t/2}^{3t/2}\int_0^s\tau \frac{1}{2^j\sqrt{t}}\sqrt{|\Delta|}\frac{1}{2^{j}\sqrt{t}}\frac{1}{\sqrt{|\Delta|}}\big(\frac{1}{s} + \frac{1}{\sqrt{s^2-\tau^2}}\big) d\tau ds\leq \frac{1}{2^{2j}}
\end{align*}

and

\begin{align*}
    &\int_{\mathbb{R}^n\setminus 8\Delta}\tilde{N}\Big(\int_0^s \tau|\partial_x e^{-\eta^2\tau^2L^s}\phi^s| |\partial_x e^{-\eta^2(s^2-\tau^2)L^s}L^se^{-\eta^2 s^2L^s}\phi^s|d\tau\Big)(x) dx
    \\
    &\lesssim \sum_{j\geq 4}\int_{2^j\Delta\setminus 2^{j-1}\Delta}\tilde{N}\Big(\int_0^s \tau|\partial_x e^{-\eta^2\tau^2L^s}\phi^s| |\partial_x e^{-\eta^2(s^2-\tau^2)L^s}L^se^{-\eta^2 s^2L^s}\phi^s|d\tau\Big)(x) dx
    \\
    &\lesssim  \sum_{j\geq 4}\int_{2^{j+1}\Delta\setminus 2^{j-2}\Delta}\frac{1}{2^{2j}}dx\lesssim |\Delta|,
\end{align*}
whence the integral \(|IV|\lesssim |\Delta|\).

\hfill\\
The proofs of \eqref{lemma:w_tSqFctBound2n=1} and \eqref{lemma:w_tSqFctBound3n=1} rely on \eqref{lemma:w_tSqFctBound1n=1} and follow as in \reflemma{lemma:SqFctBoundsForw_t}.
\end{proof}

Lastly with all lemmas in this section together, we obtain that \reflemma{lemma:L^2estimatesForSquareFunctions} holds under assumptions \eqref{condition:L^1Carlesontypecond}. The only modification needed is in the proof of \((c)\), which is trivial using the pointwise bound \reflemma{lemma:UniformBoundOnM[nablaphi^s]n=1}.

\bibliographystyle{alpha}
\bibliography{references} 
\end{document}